\documentclass{amsart}
\usepackage{microtype}
\usepackage{colordvi}
\usepackage{amsmath,amsfonts,amsthm,amssymb,appendix,latexsym}
\usepackage{textcomp}
\allowdisplaybreaks[4]
\usepackage{graphicx}
\usepackage{verbatim}
\usepackage{color}
\usepackage{amscd}
\usepackage{verbatim}
\usepackage{graphicx}
\usepackage {float}

\usepackage{url}

\usepackage{setspace}

\usepackage[left=1.25in, right=1.25in, top=1in, bottom=1in]{geometry}
\usepackage{longtable}

\usepackage{CJKutf8}

\usepackage{amsmath,amssymb,amsthm,color,mathrsfs}
\usepackage{enumitem,anysize}%
\usepackage{verbatim}
\usepackage{amssymb}
\usepackage{tikz}
\usepackage{pgfplots}
\marginsize{1in}{1in}{1in}{1in}%

\newtheorem{theorem}{Theorem}[section]
\newtheorem{corollary}[theorem]{Corollary}

\newtheorem{Thm}[theorem]{Theorem}

\newtheorem{lem}[theorem]{Lemma}
\newtheorem{lemma}[theorem]{Lemma}

\newtheorem{example}[theorem]{Example}

\newtheorem{problem}[theorem]{Problem}
\newtheorem{definition}[theorem]{Definition}
\newtheorem{Def}[theorem]{Definition}

\newtheorem{assumption}[theorem]{Assumption}
\newtheorem{assumptions}[theorem]{Assumptions}

\newtheorem{remark}[theorem]{Remark}
\newtheorem{Rem}[theorem]{Remark}
\newtheorem{remarks}[theorem]{Remarks}
\newtheorem{claim*}{Claim}
\newtheorem{notation}[theorem]{Notation}

\newcommand{\R}{\mathbb{R}}

\newcommand{\N}{\mathbb{N}}

\newcommand{\blemma}{\begin{lemma}}
	\newcommand{\elemma}{\end{lemma}}
\newcommand{\bnotation}{\begin{notation}}
	\newcommand{\enotation}{\end{notation}}
\newcommand{\bproof}{\begin{proof}}
	\newcommand{\eproof}{\end{proof}}
\newcommand{\bremark}{\begin{remark}}
	\newcommand{\eremark}{\end{remark}}
 \newcommand{\bremarks}{\begin{remarks}}
	\newcommand{\eremarks}{\end{remarks}}
\newcommand{\bcorollary}{\begin{corollary}}
	\newcommand{\ecorollary}{\end{corollary}}
\newcommand{\btheorem}{\begin{theorem}}
	\newcommand{\etheorem}{\end{theorem}}
 \newcommand{\bproblem}{\begin{problem}}
	\newcommand{\eproblem}{\end{problem}}
\newcommand{\bdefinition}{\begin{definition}}
	\newcommand{\edefinition}{\end{definition}}
\newcommand{\bequation}{\begin{equation}}
	\newcommand{\eequation}{\end{equation}}
\newcommand{\bequationn}{\begin{equation*}}
	\newcommand{\eequationn}{\end{equation*}}
\newcommand{\beqnarray}{\begin{eqnarray}}
	\newcommand{\eeqnarray}{\end{eqnarray}} \newcommand{\beqnarrayn}{\begin{eqnarray*}}
	\newcommand{\eeqnarrayn}{\end{eqnarray*}}
 \newcommand{\bexample}{\begin{example}}
	\newcommand{\eexample}{\end{example}}
  \newcommand{\bassumption}{\begin{assumption}}
	\newcommand{\eassumption}{\end{assumption}}
\newcommand{\loc}{{\rm loc}}
\numberwithin{equation}{section}
\newcommand\red[1]{\textcolor{red}{#1}}

\newcommand{\Hmm}[1]{\leavevmode{\marginpar{\tiny%
			$\hbox to 0mm{\hspace*{-0.5mm}$\leftarrow$\hss}%
			\vcenter{\vrule depth 0.1mm height 0.1mm width \the\marginparwidth}%
			\hbox to
			0mm{\hss$\rightarrow$\hspace*{-0.5mm}}$\\\relax\raggedright #1}}}

\DeclareMathOperator{\diam}{diam}
\DeclareMathOperator{\supp}{supp}
\DeclareMathOperator{\capacity}{Cap}
\DeclareMathOperator{\vol}{Vol}
\newcommand{\core}{C_c^{\infty}(\Omega)}

\newcommand{\dx}{\,\mathrm{d}x}

\newcommand{\dt}{\,\mathrm{d}t}

\newcommand{\myd}{\displaystyle}

\newcommand{\dsigma}{\,\mathrm{d}\sigma}

\newcommand{\dnu}{\,\mathrm{d}\nu}
\newcommand{\dnuk}{\,\mathrm{d}\nu_{k}}
\newcommand{\dmu}{\,\mathrm{d}\mu}

\newcommand{\dgamma}{\,\mathrm{d}\gamma}
\newcommand{\dgammak}{\,\mathrm{d}\gamma_{k}}
\newcommand{\dgammaa}{\,\mathrm{d}\gamma'}
\newcommand{\dgammaak}{\,\mathrm{d}\gamma_{k}'}

\DeclareMathOperator{\dive}{div}

\def\<{\langle}
\def\>{\rangle}

\long\def\prob#1\soln#2\endps{{\color{blue}#1}\medskip\par
	\noindent\underline{\sc Solution}:\hspace*{1em}\parindent=2em #2}

\AtBeginDocument{\begin{CJK}{UTF8}{gbsn}}
	\AtEndDocument{\end{CJK}}

            \def\gg{\gamma}

       \def\vgf{\varphi}

      \def\gw{\omega}
                
\def\Gg{\Gamma}

\def\Gw{\Omega}              
\begin{document}
	\pagenumbering{gobble}
	\title[Maz'ya-type characterization and Hardy constant]{\textbf{Finsler $p$-Laplace equation with a potential: Maz'ya-type characterization and attainments of the Hardy constant}}
\author{Yongjun Hou}
\thanks{2020 \emph{Mathematics Subject Classification.} Primary 35J20; Secondary 35B09, 35J62, 35J70.\\
	\emph{Key words and phrases.} Bregman distance, concentration compactness,  Finsler $p$-Laplace equation, Hardy constant,  Maz'ya-type characterization, Morrey space, spectral gap}
 \date{October 23, 2024}

	\pagenumbering{arabic}
	\vspace{-10mm}		
 \begin{abstract} 
 We study positive properties of the quasilinear elliptic equation 
 $$-\mathrm{div}\mathcal{A}(x,\nabla u)+V|u|^{p-2}u=0\quad (1<p<\infty)\qquad \mbox{ in } \Omega,$$
where the function $\mathcal{A}(x,\xi)$ is induced by a family of norms on $\mathbb{R}^{n}$ ($n\geq 2$) parameterized by points in the domain $\Omega\subseteq\mathbb{R}^{n}$, and $V$ belongs to a certain local Morrey space.
 We first establish some two-sided estimates for the Bregman distances of $|\xi|^{p}_{s,a}$ ($1<s<\infty$), where $a=(a_{1},a_{2},\ldots,a_{n})$ and $a_{1},a_{2},\ldots,a_{n}$ are certain functions with  positive local lower and upper bounds in $\Omega$. These estimates lead to a Maz'ya-type characterization for Hardy-weights of the corresponding functionals. Then we prove three types of sufficient conditions for the attainment of the Hardy constant in a certain space $\widetilde{W}^{1,p}_{0}(\Omega)$.
\end{abstract}
\maketitle
\medskip

 \section{Introduction}
We study some positivity properties  of the energy functional
\begin{align*}
Q[\phi]\triangleq Q_{p,\mathcal{A},V}[\phi]\triangleq\int_{\Omega}\left(\mathcal{A}(x,\nabla \phi)\cdot\nabla\phi+ V(x)|\phi|^{p}\right)\dx\quad(1<p<\infty),
\end{align*}
which is assumed to be nonnegative on $\core$, the space of infinitely differentiable functions with compact support in~$\Omega$, and its local Euler-Lagrange equation
\begin{equation*}
Q'[u]\triangleq Q'_{p,\mathcal{A},V}[u]\triangleq -\dive\mathcal{A}(x,\nabla u)+V|u|^{p-2}u=0,
\end{equation*}
in a domain $\Gw\subseteq \R^n$ ($n\geq 2$), where the operator 
$\dive\mathcal{A}(x,\nabla u)$, intensively investigated in the widely cited  book \cite{HKM}, is called the \emph{Finsler 
$p$-Laplacian} \cite{Combete} (also called the \emph{$\mathcal{A}$-Laplacian} \cite{HPR} and when~$\mathcal{A}$ does not depend on~$x$, the \emph{anisotropic~$p$-Laplacian} \cite{Jaros}), and the potential $V$ belongs to a certain local Morrey space~$M^{q}_{\loc}(p;\Omega)$ or its enhanced versions in some contexts.  The foregoing equation is homogeneous of degree $p-1$, which is crucial to our study. A prototype of our equation is the celebrated $p$-Laplace equation 
$$-\dive(|\nabla u|^{p-2}\nabla u))=0.$$

The study of positivity properties related to the equation  $Q'[u]=0$ is called {\em criticality theory} which has gained much attention.
For criticality theory in the linear case, see the two review papers \cite{Murata,Psimon} respectively by Murata and Pinchover, and Pinsky's monograph \cite{Pinsky} from the probabilistic perspective. For the evolution of criticality theory in the quasilinear case, see Table~\ref{table1} below. For a discrete counterpart, see \cite{Fischer1,Fischer2,Keller2}, where criticality theory is developed for Schr\"odinger operators and the pseudo $p$-Laplacian with a potential on discrete infinite graphs. Throughout this paper, unless otherwise stated,~$A:\Omega\rightarrow\R^{n^{2}}$ is a measurable, symmetric, locally bounded, and locally uniformly positive definite matrix function, and~$|\xi|_{A}\triangleq\sqrt{A(x)\xi\cdot\xi}$ for almost all~$x\in\Omega$ and all~$\xi\in\R^{n}$.

\setcounter{table}{0}
\begin{table}[ht]\caption{The evolution of criticality theory in the quasilinear case}\label{table1}
\renewcommand\arraystretch{1.5}
\noindent
	\centering
	\begin{tabular}{||c c c||} 
		\hline
		The second-order term & The potential & The reference(s)\\ [0.5ex] 
		\hline
		the \emph{$p$-Laplacian}: $\dive(|\nabla u|^{p-2}\nabla u)$ &$L^{\infty}_{\loc}(\Omega)$& \cite{Pinliou,PCVPDE}\\
		\hline
		the \emph{$(p,A)$-Laplacian}: $\dive(|\nabla u|_{A}^{p-2}A(x)\nabla u)$ & $L^{\infty}_{\loc}(\Omega)$& \cite{Regev}\\
		\hline
		the $(p,A)$-Laplacian& $M^{q}_{\loc}(p;\Omega)$& \cite{PPAPDE}\\
		\hline
		the $(p,A)$-Laplacian& $\mathfrak{W}^{p}_{\loc}(\Omega)$& \cite{Giri2}\\ 
		\hline
		the Finsler $p$-Laplacian: $\dive\mathcal{A}(x,\nabla u)$ & $M^{q}_{\loc}(p;\Omega)$& \cite{Hou,HPR}\\
		\hline
	\end{tabular}
\end{table}

In recent years, the Finsler~$p$-Laplace equation, related to Finsler geometry (see \cite{Combete}),  has attracted numerous mathematicians. It is studied from a wide range of perspectives, e.g., overdetermined problems and Brunn-Minkowski type inequalities (see   \cite{Bianchini,Cianchi,Xianew,Xiabm}). See also Remark \ref{Finslerr} below. We also mention \cite{Figallia}, where Ciraolo, Figalli, and Roncoroni studied critical anisotropic~$p$-Laplace equations 
in convex cones.

Definition \ref{hweight} below is a generalization of the counterpart in \cite[Definition 6.3]{HPR}.
\begin{definition}\label{hweight} 
 \emph{A function~$g\in L^{1}_{\loc}(\Omega)$ is called a \emph{Hardy-weight of~$Q$} in $\Gw$ if there exists a constant~$C>0$ such that the following \emph{Hardy-type inequality}
$$Q[\phi] \geq C\int_{\Omega}|g||\phi|^{p}\dx$$
holds for all~$\phi\in\core$.} 
\end{definition}
The following characterization by Maz'ya of Hardy-weights for the $p$-Laplacian is extracted from \cite[Theorem 8.5]{Maz'ya}.
\btheorem
Let~$\mu$ be an arbitrary measure on~$\Omega$. Then the following Hardy-type inequality
$$ \int_{\Omega}|\nabla\phi|^{p}\dx \geq C\int_{\Omega}|\phi|^{p}\dmu$$
holds for some positive constant~$C$ and all~$\phi\in\core$
if and only if for some positive constant~$C'$ and all open sets with smooth boundary such that $\bar{\omega}\subseteq\Omega$,
$$\mu(\omega)\leq C'\capacity_{p}(\bar{\omega}),$$
where~$\capacity_{p}(\cdot)$ is the classical~$p$-capacity.
\etheorem

For a standard reference on Hardy-type inequalities, see \cite{Balinsky}. In \cite{Das}, Das and Pinchover demonstrated a Maz'ya-type characterization for the~$(p,A)$-Laplacian with a potential in Morrey spaces and showed three types of sufficient conditions such that the best Hardy constant is achieved in a Beppo Levi space. See \cite{DKP} for similar results in the context of graphs. In \cite{Das2}, the authors proved that on~$C^{1,\gamma}$-domains with compact boundary, a certain weighted Hardy-type inequality has a unique (up to a positive multiplicative constant) positive minimizer if and only if the corresponding weighted~$p$-Laplacian admits a spectral gap. 
See also \cite{Mercaldo} for anisotropic Hardy-type inequalities.

Our central results, as some generalizations of the ones in \cite{Das}, consist of three parts. The first is some two-sided estimates for the Bregman distances (see Definition \ref{BREGMAN}) of~$|\xi|_{s,a}^{p}\triangleq \left(\sum_{i=1}^{n}a_{i}(x)|\xi_{i}|^{s}\right)^{p/s}$ ($\xi\in\R^{n}$) for almost all~$x\in\Omega$ (see Definition \ref{sanorm}) and some related simplified energies. The second is a Maz'ya-type characterization for Hardy-weights of two functionals respectively based on the norm~$|\cdot|_{s,a}$ and the norm~$\sqrt[p]{|\cdot|_{s,a}^{p}+|\cdot|_{A}^{p}}$. The last is three sufficient conditions for the attainment 
of the (best)~$Q$-Hardy constant (see Definition \ref{hardy_weight_norm} (1)).

The energy functional associated with the~$p$-Laplacian with a potential $V$ was simplified in \cite{Pinliou} via a Picone-type identity and a two-sided estimate for the Bregman 
distance  of~$|\cdot|^{p}$. The derived estimate, called the {\em simplified energy}, was extended in \cite{Regev} to the~$(p,A)$-Laplacian with a potential.  The simplified energy led to further intriguing results including some Liouville comparison theorems (see \cite[Theorem 1.9]{Pinliou} and \cite[Theorem 4.21]{PPAPDE}) and the Maz'ya-type characterization in \cite{Das} for Hardy-weights of the corresponding functional. 

Our context also entails simplified energies, which motivates us to investigate relevant Bregman distance. The Bregman distances (also called the \emph{Bregman divergence} \cite{Sprung}) is  involved in a variety of mathematical works (see, e.g., \cite{BI00,BR06,Lindqvist,Pinliou,Shafrir}), in particular, in statistics, information geometry, and machine learning (see, e.g., \cite{Information,Sprung} and references therein),  even though the term `Bregman distance' is not mentioned directly in various contexts. For example, some upper and lower bounds of some Bregman distances were the cornerstone of \cite{Figalli}, where Figalli and Zhang proved a kind of sharp stability of the Sobolev inequality (see also \cite{Figalli3}). For the history of the Bregman distance, we refer to \cite{Burger,BI00}. 

The two-sided estimates of the involved Bregman distances in \cite{Regev,Pinliou} rely heavily on the inner product structure. 
In \cite{Sprung}, some upper and lower bounds were respectively established for uniformly smooth and uniformly convex norms (see \cite{Sprung} for definitions). However, for the norm~$|\cdot|_{s,a}$, the corresponding upper and lower bounds provided by \cite{Sprung} have essentially different forms except the special case~$s=2$. In particular, the two bounds can not control each other up to a positive multiplicative constant unless~$s=2$. 
In this work, we establish some two-sided estimates for the Bregman distances of~$|\cdot|_{s,a}^{p}$ (see Section~\ref{dbregman}), which eliminates the difficulty. Exploiting these estimates, we prove the foregoing Maz'ya-type characterization. 

The outline of the paper is as follows. In Section \ref{equation}, we introduce a norm family~$H$, a variational Lagrangian~$F$, the mapping~$\mathcal{A}$, and the local Morrey space~$M^{q}_{\loc}(p;\Omega)$ together with two enhanced versions. Then we define the energy functional~$Q$ explicitly, some fundamental notions in criticality theory, and the associated Finsler~$p$-Laplace equation with a potential~$Q'[u]=0$. Furthermore, we give two extra assumptions on~$H$ and the potential~$V$. One is a lower bound for the Bregman distances of~$|\cdot|^{p}_{\mathcal{A}}$ locally in~$\Omega$ (see Assumption \ref{ass2}). The other is that in some particular cases,~$H$ and~$V$ are so regular that positive solutions of certain equations are in~$W^{1,\infty}_{\loc}$ (see Assumption \ref{ngradb}). We also define a~$Q$-capacity and the $Q$-Hardy constant. Some related fundamental properties are also listed. At the end of this section, as a generalization of \cite[Proposition 3.4]{Das}, we prove a necessary condition for Hardy-weights by virtue of positive solutions of minimal growth.

In Section \ref{ready}, we define the Bregman distance, discuss the differentiability of~$|\cdot|_{s,a}$ ($1<s<\infty$), and give some examples such that solutions of the relevant equations are in~$W^{1,\infty}_{\loc}$.
Then we prove the aforementioned two-sided estimates in detail. Furthermore, we obtain three relevant simplified energies and the Maz'ya-type characterization.

In Section \ref{attainmts}, we first define a~$Q$-Sobolev space~$\widetilde{W}^{1,p}(\Omega)$ and a subspace~$\widetilde{W}^{1,p}_{0}(\Omega)$. We also prove some basic properties of the two spaces. Then we provide three attainments of the $Q$-Hardy constant in the space~$\widetilde{W}^{1,p}_{0}(\Omega)$. The attainments are respectively based on closed images of bounded sequences (up to a subsequence) under an operator~$T_{g}$, a spectral gap, and a concentration compactness method.

In the appendix, we recall two two-sided estimates in \cite{Pinliou,Regev} for the Bregman distances of~$|\cdot|^{p}$ and~$|\cdot|_{A}^{p}$ respectively, where~$A$ is a symmetric positive definite matrix and~$|\xi|_{A}\triangleq\sqrt{A\xi\cdot\xi}$ for all~$\xi\in\R^{n}$.

\textbf{Throughout the paper, unless otherwise stated, we fix $1\!<\!p<\!\infty$, $n\geq 2$ ($n\in \N$), and a  domain (a nonempty, connected, and open set) $\Omega\! \subseteq \!\R^{n}$.}
 \section*{Basic symbols}
 \normalsize
\begin{longtable}[1]{p{60pt} p{350pt} }
 $\vol(\Gg)$ & the Lebesgue measure of a measurable subset~$\Gg$ of~$\R^{n}$\\
 $p^{*}$  & the Sobolev critical exponent~$np/(n-p)$ of~$1<p<n$\\
	$\omega \Subset \omega'$ & The closure of the domain $\omega$ is a compact subset of the domain $\omega'$. \\ 
	$B_r(x)$  & the open ball centered at $x\in\mathbb{R}^{n}$ with a radius $r>0$\\
 $\supp f$& the support of a function~$f: U\rightarrow \mathbb{R}$ or~$f: U\rightarrow \mathbb{R}^{n}$, namely~$\overline{\{x\in U\,|f(x)\neq 0\}}$,  where~$U$ is an open subset of~$\R$ or~$\R^{n}$\\
$\mathcal{F}_{c}(U)$ & the space of all functions in~$\mathcal{F}(U)$ with compact support in~$U$, where~$\mathcal{F}(U)$ is a function space over an open set~$U\subseteq\R^n$\\
	
 $f^{+}$  & the positive part of a real-valued function $f$, namely $\max\{f,0\}$\\
 $f^{-}$  & the negative part of a real-valued function $f$, namely $-\min\{f,0\}$\\ 
 $\partial_{i}f$& the $i$-th partial derivative of a differentiable function~$f$ defined in an open subset of~$\R^{n}$, where~$i=1,2,\ldots,n$\\
$c, C$ &
positive constants which may vary from line to line\\
$f\asymp g$&For some positive constants~$c$ and~$C$, the inequalities $cg\leq f\leq Cg$ hold in the same domain of the nonnegative functions $f$ and $g$; in this case, $f$ and $g$ are called \emph{equivalent}; the constants~$c$ and~$C$ are called \emph{equivalence constants}.\\
		\end{longtable}
\section{The functional~$Q$ and the equation~$Q'[u]=0$}\label{equation}
In this section, we define the functional~$Q$ precisely, the associated Finsler~$p$-Laplace equation with a potential~$Q'[u]=0$, a~$Q$-capacity, and the~$Q$-Hardy constant, clarify relevant assumptions, and enumerate some related properties. We also review criticality/subcriticality, null-sequences, ground states, and positive solutions of minimal growth. In the end, we prove a necessary condition for Hardy-weights.
	\subsection{The norm family~$H$, the operator~$\mathcal{A}$, and the potential~$V$}
In this subsection, we first define uniformly convex Banach spaces and introduce a family of norms~$(H(x,\cdot))_{x\in\Omega}$ on $\R^{n}$ which induce a variational Lagrangian~$F$ and the operator~$\mathcal{A}$. Then we give a brief exposition of the Morrey space~$M^{q}_{\loc}(p;\Omega)$. We also define two subspaces of~$M^{q}_{\loc}(p;\Omega)$.
 
 The following definition is found on \cite[p.\,192]{Roach}. See also \cite{Z02} for more on uniform convexity.
\begin{definition}\label{modcon}
\emph{The \emph{modulus of convexity} of a Banach space~$(X,\Vert\cdot\Vert)$, whose dimension is greater than or equal to two, is defined as $$\delta_{X}(\varepsilon)\triangleq \inf\bigg\{1-\bigg\Vert\frac{x+y}{2}\bigg\Vert~\bigg|~x,y\in S_{X},\Vert x-y\Vert\geq \varepsilon\bigg\},$$ for~$0\leq\varepsilon\leq 2$, where~$S_{X}$ is the unit sphere in~$(X,\Vert\cdot\Vert)$. The Banach space~$(X,\Vert\cdot\Vert)$ is called \emph{uniformly convex} if~$\delta_{X}(\varepsilon)>0$ for all~$0<\varepsilon\leq 2$.}
\end{definition}
\bremark
\emph{See \cite[p.~420]{Danes} for some alternative definitions of the modulus of convexity.}
\eremark 
The following assumptions are the basis of this work.	
\begin{assumptions}\label{ass9} 
		{\em Let~$H:\Omega\times\R^{n}\rightarrow\R$ be a mapping such that for almost all~$x\in\Omega$,~$H(x,\cdot)$ is a norm on~$\R^{n}$,	
  and satisfy  the following conditions:   
			\begin{description}
				\item[Measurability] For all~$\xi\in\mathbb{R}^{n}$, the mapping $x\mapsto H(x,\xi)$ is measurable in $\Gw$;
				\item[Local uniform equivalence] for  every domain $\omega\Subset \Gw$, there exist two constants~$0<\kappa_\omega\leq\nu_\omega<\infty$ such that for almost all $x\in \omega$ and all $\xi\in \R^n$,$$\kappa_\omega |\xi|\leq H(x,\xi)\leq\nu_\omega|\xi|;$$
				\item[Uniform convexity] for almost all~$x\in\Omega$, the Banach space $\mathbb{R}^{n}_{x}\triangleq(\mathbb{R}^{n},H(x,\cdot))$ is uniformly convex; 
				\item[Differentiability with respect to $\xi$] for almost all~$x\in \Gw$, the mapping $\xi\mapsto H(x,\xi)$ is differentiable in $\R^n\setminus\{0\}$.
			\end{description}		
		}
	\end{assumptions}
Now we introduce the Lagrangian~$F$. See also \cite[Assumption 2.1]{HPR}.
\begin{theorem}\label{F}
  Let~$(H(x,\cdot))_{x\in\Omega}$ be a family of norms on $\mathbb{R}^{n}$ fulfilling Assumption \ref{ass9}. 
	For almost all
 ~$x\in \Omega$ and all~$\xi\in \R^{n}$, we define
	$$F(x,\xi)\triangleq\frac{1}{p}H(x,\xi)^{p}.$$
		Then $F$ satisfies the following conditions:
			\begin{description}
				\item[Measurability] For all~$\xi\in\mathbb{R}^{n}$, the mapping $x\mapsto F(x,\xi)$ is measurable in $\Gw$;
				\item[Local uniform ellipticity and boundedness] for  every domain $\omega\Subset \Gw$, there exist two constants $0<\kappa'_\omega\leq\nu'_\omega<\infty$ such that for almost all $x\in \omega$ and all $\xi\in \R^n$,$$\kappa'_\omega|\xi|^{p}\leq F(x,\xi)\leq\nu'_\omega|\xi|^{p};$$
				\item[Strict convexity and $C^1$ with respect to $\xi$] for almost all~$x\in \Gw$, the mapping $\xi\mapsto F(x,\xi)$ is strictly convex and continuously differentiable in $\R^n$;
				\item[Homogeneity]for almost all~$x\in \Gw$, all~$\lambda\in\mathbb{R}$, and all~$\xi\in\mathbb{R}^{n}$, $F(x,\lambda\xi)=|\lambda|^{p}F(x,\xi)$.
			\end{description}		
	\end{theorem}
	\begin{proof}
		The measurability, the local uniform ellipticity and boundedness, and the homogeneity are obvious by Assumptions \ref{ass9} and the definition of the function~$F$. By \cite[Theorem 3.7.7 (ii)]{Z02}, for almost all~$x\in\Omega$, the function~$F(x,\xi)$ is uniformly convex on bounded sets with respect to~$\xi$ and hence strictly convex by \cite[Theorem 2.14, (i) $\Leftrightarrow$ (vii)]{BR06}.  Note that for almost all~$x\in\Omega$,~$\nabla_{\xi}F(x,0)=0$. Then the continuous differentiability follows from  the differentiability of~$H(x,\cdot)$ in~$\R^{n}\setminus\{0\}$ for almost all~$x\in\Omega$ postulated in Assumption~\ref{ass9}, the chain rule, and \cite[p.~97]{HKM}. 
	\end{proof}
 \bremark\label{revaf}
\emph{We might also start with the function~$F$ and regard the preceding theorem as our basic assumptions in place of Assumptions~\ref{ass9}. Then such a function~$F$ induces a family of norms~$\sqrt[p]{pF(x,\cdot)}$ (see \cite[Remark 2.4]{HPR}) satisfying Assumptions \ref{ass9}. See \cite[Theorem 2.14]{BR06} and \cite[Theorem 3.7.7]{Z02} for the equivalence between strict convexity of~$F(x,\xi)$ and uniform convexity of~$H(x,\xi)$ with respect to~$\xi$ for almost all~$x\in\Omega$}.
 \eremark
 \begin{notation}\label{Fgrad}
		\emph{Given a family of norms~$(H(x,\cdot))_{x\in\Omega}$  fulfilling Assumptions~\ref{ass9},  we write, for almost all~$x\in\Omega$ and all~$\xi\in\mathbb{R}^{n}$,
			$$ \mathcal{A}(x,\xi) \triangleq \nabla_\xi F(x,\xi).$$ 
			Moreover, frequently, we also write~$H(x,\xi)$ simply as~$|\xi|_{\mathcal{A}}$. Furthermore, when the norm family~$H$ does not depend on~$x$, by mild abuse of notation,
    we write~$H(\xi)$ and~$\mathcal{A}(\xi)$.}
	\end{notation}
The $p$-homogeneity of $F$ implies the following  relationship between~$F$ and $\mathcal{A}$.
	\begin{lemma}[{\cite[p.\,100]{HKM}}]
	Let~$(H(x,\cdot))_{x\in\Omega}$ be a family of norms satisfying Assumptions~\ref{ass9}. For almost all~$x\in\Omega$ and all~$\xi\in\mathbb{R}^{n}$,~$\mathcal{A}(x,\xi)\cdot\xi =pF(x,\xi).$
	\end{lemma}
 We list some basic properties of the operator~$\mathcal{A}$.
		\begin{Thm}[{\cite[Lemma 5.9]{HKM}}]\label{thm_1}
			Let~$(H(x,\cdot))_{x\in\Omega}$ be a family of norms satisfying Assumptions~\ref{ass9}. For every domain $\omega\Subset\Omega$, let $\alpha_{\omega}\triangleq\kappa'_{\omega}$ and $\beta_{\omega}\triangleq2^{p}\nu'_{\omega}$. Then the vector-valued function~$\mathcal{A}(x,\xi):  \Gw\times \mathbb{R}^{n}\rightarrow \mathbb{R}^{n}$ satisfies the following conditions:
				\begin{description}
					\item[Continuity and measurability] For almost all $x\in \Gw$, the function
					$\mathcal{A}(x,\xi ): \mathbb{R}^{n} \rightarrow \mathbb{R}^{n}$
					is continuous with respect to $\xi$, and  $x \mapsto \mathcal{A}(x,\xi)$ is Lebesgue measurable in $\Gw$ for all~$\xi\in \mathbb{R}^{n}$;
					%
					\item[Local uniform ellipticity and boundedness] for all domains $\omega\Subset \Gw$, all $\xi \in \mathbb{R}^{n}$, and almost all $x\in \omega$,
					\begin{equation*}\label{structure}
						\alpha_\omega|\xi|^{p}\le\mathcal{A}(x,\xi)\cdot\xi\quad\mbox{and}\quad
						|\mathcal{A}(x,\xi)|\le \beta_\omega\,|\xi|^{{p}-1};
					\end{equation*}
					\item[Strict monotonicity]  for almost all~$x\in\Gw$ and all~$\xi,\eta\in\mathbb{R}^{n}$($\xi\neq\eta$),
					$$\big(\mathcal{A}(x,\xi)-\mathcal{A}(x,\eta)\big)\cdot(\xi-\eta)>0;$$
     \item[Homogeneity] for all~$\lambda\in {\mathbb{R}\setminus\{0\}}$,
					$\mathcal{A}(x,\lambda \xi)=\lambda\,|\lambda|^{p-2}\,\mathcal{A}(x,\xi).$
			\end{description}
		\end{Thm}
Now we are in a position to define the local Morrey space~$M^{q}_{\loc}(p;\Omega)$. See also \cite[Definitions 2.1 and 2.3]{PPAPDE}
  \begin{Def}\label{Morreydef1}{\em
	Suppose that~$\omega\Subset \Omega$ is a domain and $f$ is a real-valued measurable function in~$\omega$. 
  Then, with~$\mathrm{diam}(\omega)$ standing for the diameter of~$\omega$,
				\begin{itemize}
					\item for $p<n$ and $q>n/p$, we define, letting~$q'\triangleq q/(q-1)$,$$\Vert f\Vert_{M^{q}(p;\omega)}\triangleq \sup_{\substack{y\in\gw\\0<r<\diam(\gw)}}
					\frac{1}{r^{n/q'}}\int_{\omega\cap B_{r}(y)}|f|\dx,$$ and 
					$$M^{q}(p;\omega)\triangleq\{f\in L^{1}_{\loc}(\omega)~|~\Vert f\Vert_{M^{q}(p;\omega)}<\infty\};$$ 
					\item for $p=n$ and $q>n$ , we define$$\Vert f\Vert_{M^{q}(n;\omega)}\triangleq \sup_{\substack{y\in\gw\\0<r<\diam(\gw)}} \varphi_{q}(r)\int_{\omega\cap B_{r}(y)}|f|\dx,$$
					where $\varphi_{q}(r)\triangleq \left(\log\big(\mathrm{diam}(\omega)/r\big)\right)^{q/n'}$, and
$$M^{q}(n;\omega)\triangleq\{f\in L^{1}_{\loc}(\omega)~|~\Vert f\Vert_{M^{q}(p;\omega)}<\infty\};$$
					\item for $p>n$ and $q=1$, we define~$M^{q}(p;\omega)\triangleq L^{1}(\omega)$.
				\end{itemize}
    Finally, we define the \emph{local Morrey space} by $$M^{q}_{\loc}(p;\Omega)\triangleq \bigcap_{\substack{\omega\Subset\Omega\\\omega~\mbox{is a domain}}}M^{q}(p;\omega).$$
			}
		\end{Def}
  It can be easily verified that~$L^{q}_{\loc}(\Omega)\subseteq M^{q}_{\loc}(p;\Omega)$. The following example demonstrates that $M^{q}_{\loc}(p; \R^{n})\neq  L^{q}_{\loc}(\R^{n})$.
The calculation is elementary and hence omitted.

\bexample\label{MLebesgue}
 \emph{Let~$1<p<n$,~$n/p<q<\infty$,~$\Omega=\R^{n}$, and~$f(x)\triangleq 1/|x|^{n/q}$.
Then~$f(x)\in M^{q}_{\loc}(p;\R^{n})\setminus L^{q}_{\loc}(\R^{n})$.
}
  \eexample
  \textbf{Throughout the paper,  $(H(x,\cdot))_{x\in\Gw}$ is a family of norms on $\mathbb{R}^{n}$ fulfilling Assumptions~\ref{ass9}, and~$V\in M^{q}_\loc(p;\Omega)$, the local Morrey space}.
  
  Next we define two subspaces of~$M^{q}_{\loc}(p;\Omega)$. 
  \begin{definition}\label{strm}
\emph{In the definition of the Morrey space $M^{q}_{\loc}(p;\Omega)$,
	 we require further that  when~$p<n$,~$q>n$ and 
when~$p=n$, for some~$\theta\in (n-1,n)$ and all domains~$\omega\Subset\Omega$,
$$\sup_{\substack{y\in\omega\\0<r<\diam(\omega)}}\frac{1}{r^{\theta}}\int_{B_{r}(y)\cap\omega}|V|\dx<\infty.$$
We denote the new space by~$\widetilde{M}^{q}_{\loc}(p;\Omega)$. In addition, we define the space~$\widehat{M}^{q}_{\loc}(p;\Omega)$ as the set of all the functions~$V$ in~$\widetilde{M}^{q}_{\loc}(p;\Omega)$ such that when~$p>n$, for some~$\vartheta\in (p-1,p)$ and all domains~$\omega\Subset\Omega$,
$$\sup_{\substack{y\in\omega\\0<r<\diam(\omega)}}\frac{1}{r^{n-p+\vartheta}}\int_{B_{r}(y)\cap\omega}|V|\dx<\infty.$$}
\end{definition}
\subsection{The energy functional~$Q$}
 In this subsection, we define the  energy functional $Q$. We also recall criticality/subcriticality, null-sequences, and ground states.
\begin{definition}\label{efdfn}\emph{ The\emph{ energy functional~$Q$} is defined  as 
		\begin{align*}
Q[\phi]\triangleq Q[\phi;\Omega]&\triangleq Q_{p,\mathcal{A},V}[\phi]\\
&\triangleq\int_{\Omega}\left(H(x,\nabla \phi)^{p}+ V(x)|\phi|^{p}\right)\dx\\
&=\int_{\Omega}\left(\mathcal{A}(x,\nabla \phi)\cdot\nabla\phi+ V(x)|\phi|^{p}\right)\dx \qquad \phi\in\core.
\end{align*}
The functional~$Q$ is called \emph{nonnegative} in~$\Omega$ if for all~$\phi\in \core$,~$Q[\phi]\geq 0,$ and in this case, we write~$Q\geq 0$ in~$\Omega$.}
		\end{definition}
  \bremark
 \emph{When it comes to the nonnegativity of~$Q$, by an approximation argument (see \cite[Lemma 4.18]{HPR}), our test functions can also belong to~$W^{1,p}(\Omega)\cap C_{c}(\Omega)$ instead of~$\core$. In some cases, we may consider such test functions. Besides, when the integration region is replaced by an arbitrary measurable subset~$\Gg$ of~$\Omega$, we also write the functional as~$Q_{p,\mathcal{A},V}[\phi;\Gg]$ or~$Q[\phi;\Gg]$.}
  \eremark
  \textbf{  Throughout the paper, unless otherwise stated, we invariably assume that the functional~$Q\geq 0$ in~$\Omega$.}
  
The following lemma can be easily verified.
\blemma
All Hardy-weights of~$Q$ in $\Gw$ form a linear space, which is denoted by~$\mathcal{H}_{p,\mathcal{A},V}(\Omega)$ or~$\mathcal{H}(\Omega)$.
\elemma
\begin{Def}[{Criticality/Subcriticality \cite[Definition 6.3]{HPR}}]
  	\ {\em 
  	   	  \begin{itemize}
    \item If~$Q$ is nonnegative in~$\Omega$ and~$\mathcal{H}(\Omega)=\{0\}$, 
the functional $Q$ is called \emph{critical} in $\Gw$;
    \item if $Q$ is nonnegative in~$\Omega$ and~$\mathcal{H}(\Omega)\neq\{0\}$, the functional $Q$ is called  \emph{subcritical} in $\Gw$;
    \item if there exists~$\phi\in C^{\infty}_{c}(\Omega)$ such that $Q[\phi]<0$, $Q$ is called \emph{supercritical} in $\Gw$.
  \end{itemize}
}
\end{Def}
The following definition of null-sequences requires continuity, which is superficially different from \cite[Definition 6.4]{HPR}. For an explanation, see Remark \ref{nullrem} below. 
\begin{Def}
   \emph{A nonnegative sequence~$\{\phi_{k}\}_{k\in\mathbb{N}}\subseteq W^{1,p}(\Omega)\cap C_{c}(\Gw)$ is called a \emph{null-sequence} with respect to the nonnegative functional~$Q$ in~$\Omega$ if
    \begin{itemize}
      \item there exists a fixed open set~$U\Subset\Omega$ such that~$\Vert \phi_{k}\Vert_{L^{p}(U)}=1$ for all~$k\in\mathbb{N}$;
      \item $\displaystyle{\lim_{k\rightarrow\infty}}Q[\phi_{k}]=0.$
    \end{itemize}}
  \end{Def}
 \begin{Def}
  \emph{A \emph{ground state} of the nonnegative functional~$Q$ is a positive  function~$\phi\in W^{1,p}_{\loc}(\Omega)$ which is an~$L^{p}_{\loc}(\Omega)$ limit of a null-sequence.} 
  \end{Def}
  \subsection{The Finsler~$p$-Laplace equation with a potential}
  In this subsection, we introduce the main equation and two important assumptions in criticality theory. Inter alia, we present a theorem concerning gradient estimates, a  simple lemma concerning null-sequences and ground states, and the Agmon-Allegretto-Piepenbrink (AAP) theorem.
  \begin{Def}\label{def_sol}
			{\em  
A function~$v\in W^{1,p}_{\loc}(\Omega)$ is a {\em (weak) solution} of the equation
				\begin{equation*}\label{half}
					Q'[u]\triangleq Q'_{p,\mathcal{A},V}[u]\triangleq -\dive\mathcal{A}(x,\nabla u)+V|u|^{p-2}u=0,
				\end{equation*}
				in~$\Omega$ if for all~$\phi \in C_{c}^{\infty}(\Omega)$,$$\int_{\Omega}\mathcal{A}(x,\nabla v)\cdot \nabla \phi\dx+\int_{\Omega}V|v|^{p-2}v \phi\dx=0.$$
    A function~$v\in W^{1,p}_{\loc}(\Omega)$ is a {\em (weak) supersolution} of the above equation
				in~$\Omega$ if $$\int_{\Omega}\mathcal{A}(x,\nabla v)\cdot \nabla \phi\dx+\int_{\Omega}V|v|^{p-2}v \phi\dx\geq 0 \quad \mbox{ for all nonnegative } \phi \in C_{c}^{\infty}(\Omega).$$ A function $v\!\in \!W^{1,p}_{\loc}(\Omega)$ is a {\em (weak) subsolution} of $Q'[u]=0$
				in~$\Omega$ if~$-v$ is a supersolution.	}		\end{Def}
\bremark\label{Finslerr}
\emph{By an approximation argument, the test function space can expand into~$W^{1,p}_{c}(\Omega)$ (see \cite[Lemma 4.20]{Hou}). For~$V=0$, the equation~$-\dive \mathcal{A}(x,\nabla u)=0$ is called the \emph{Finsler~$p$-Laplace equation} \cite{Combete}.
 See \cite{Chern,Shen} for a treatment of the Finsler geometry and its relationship to the equation $-\dive \mathcal{A}(x,\nabla u)=0$.}
\eremark
The following two assumptions are two alternatives to \cite[Assumptions 2.8 and 6.1]{HPR}.
\begin{assumption}\label{ass2}
	{\em For some locally bounded measurable function~$\tilde{\delta}:\Omega\rightarrow [2,\infty)$ 
and every domain $\gw\Subset \Gw$, there exists a positive constant $C_{\omega}$ (independent of~$x,\xi$, or~$\eta$) such that for almost all~$x\in \gw$ and all~$\xi,\eta\in \R^n$,	$$|\xi+\eta|^{p}_{\mathcal{A}}-|\xi|^{p}_{\mathcal{A}}-p\mathcal{A}(x,\xi)\cdot\eta\geq C_{\omega}|\eta|^{\tilde{\delta}}_\mathcal{A}(|\eta|_{\mathcal{A}}+|\xi|_{\mathcal{A}})^{p-\tilde{\delta}}.$$
	}
\end{assumption}
\begin{assumption}\label{ngradb}	\emph{In addition to Assumption~\ref{ass9} and  Assumption \ref{ass2}, we assume that if $1<p<\tilde{\delta}$ over a set with positive measure, then~$H$ and $V$ are so regular that positive solutions of the equation~$Q'_{p,\mathcal{A},V-W}[u]=0$ in all subdomains~$\Omega'$ of~$\Omega$, for all $W\in L^{\infty}_{\loc}(\Omega')$,  are in $W^{1,\infty}_{\loc}(\Omega')$, and positive solutions of $Q'[u]=f$ in~$\Omega$ for $0\leq f\in \core\setminus\{0\}$  are in $W^{1,\infty}_{\loc}(\Omega)$.} 
\end{assumption}
\bremark
\emph{Note that solutions of~$Q'[u]=f$ in~$\Omega$ are defined in the natural way (see \cite[Definition 2.17]{HPR}). 
For certain gradient estimates, see \cite{Dong,Mingione,PPAPDE} and references therein. See also Theorem \ref{gethm} and Example \ref{gbexa}. 
Moreover, if $1<p<\tilde{\delta}$ over a set with positive measure, we only consider positive supersolutions in~$W^{1,\infty}_{\loc}$ when it comes to the criticality theory in \cite[Section 6]{HPR}.}
\eremark
\bremark\label{nullrem}
\emph{By \cite[Lemma 6.20]{HPR}, under Assumptions \ref{ass9}, \ref{ass2}, and \ref{ngradb}, a critical functional has a null-sequence in~$\core$.} 
\eremark
The following theorem is easily concluded from \cite[Theorem 5.3]{Lieberman93} which provides~$C^{1,\alpha}$ estimates for certain quasilinear equations with measure data. Recall that solutions of the relevant equations in the following theorem are locally H\"older continuous, which is due to \cite[Theorem 4.11]{Maly97} for~$p\leq n$ and the Sobolev embedding theory for~$p>n$,  and hence locally bounded.
\btheorem\label{gethm} 
Suppose that for all~$x\in\Omega$,~$\mathcal{A}(x,\cdot)\in C^{1}(\R^{n}\setminus\{0\})$ and that for every domain~$\omega\Subset\Omega$, there exist positive constants~$\vartheta_{\omega}\leq 1$,~$\Lambda_{1,\omega}$,~$\Lambda_{2,\omega}$, and~$\Lambda_{3,\omega}$ such that for all~$x,y\in\omega$, all~$\xi\in\R^{n}\setminus\{0\}$, and all~$\eta\in\R^{n}$,
\begin{align*}
D_{\xi}\mathcal{A}(x,\xi)\eta\cdot\eta\geq \Lambda_{1,\omega}|\xi|^{p-2}|\eta|^{2},\quad |D_{\xi}\mathcal{A}(x,\xi)|\leq \Lambda_{2,\omega}|\xi|^{p-2},
\end{align*}
and \begin{align*}
|\mathcal{A}(x,\xi)-\mathcal{A}(y,\xi)|\leq \Lambda_{3,\omega}|\xi|^{p-1}|x-y|^{\vartheta_{\omega}}.
\end{align*} In addition, assume that~$V\in \widehat{M}^{q}_{\loc}(p;\Omega)$. Then in all subdomains~$\Omega'$ of~$\Omega$, solutions of the equations $Q'_{p,\mathcal{A},V-W}[u]=0$ for all~$W\in L^{\infty}_{\loc}(\Omega')$ and~$Q'[u]=f$ for all~$f\in C_{c}^{\infty}(\Omega')$ are in~$W^{1,\infty}_{\loc}(\Omega')$.
\etheorem
The proof of the following lemma is similar to that of \cite[Lemma 5.5]{Das} and hence omitted.
\blemma\label{inull}
Suppose that Assumptions \ref{ass2} and \ref{ngradb} hold. Let~$\{\phi_{k}\}_{k\in\mathbb{N}}\subseteq W^{1,p}(\Omega)\cap C_{c}(\Omega)$ be a null sequence of~$Q$ converging in~$L^{p}_{\loc}(\Omega)$ to a ground state~$\Phi\in W^{1,p}_{\loc}(\Omega)\cap C(\Omega)$. 
Then~$\{\hat{\phi}_{k}\triangleq \min\{\phi_{k},\Phi\}\}_{k\in\mathbb{N}}\subseteq W^{1,p}(\Omega)\cap C_{c}(\Omega)$ 
and~$\lim_{k\rightarrow\infty}Q[\hat{\phi}_{k}]=0$. Furthermore, up to a subsequence,~$\{\hat{\phi}_{k}\}_{k\in\mathbb{N}}$ converges to~$\Phi$ a.e.~in~$\Omega$ and in~$L^{p}_{\loc}(\Omega)$, and~$\lim_{k\rightarrow\infty}\Vert\hat{\phi}_{k}\Vert_{L^{p}(U)}=1$ for some open set~$U\Subset\Omega$.
\elemma
Let us review the Agmon-Allegretto-Piepenbrink (AAP) theorem which will be used in Lemma \ref{crucial}.
\btheorem[{The AAP theorem \cite[Theorem 5.3]{HPR}}]
The functional $Q$ is nonnegative in $\Gw$ if and only if the  equation $Q'[u]=0$ in $\Gw$ admits a positive (super-)solution.
\etheorem
\subsection{The $Q$-capacity and the~$Q$-Hardy constant}
In this subsection, we define a $Q$-capacity and the~$Q$-Hardy constant. We also enumerate some fundamental properties of the $Q$-capacity and consider two norms on~$\mathcal{H}(\Omega)$. 
Then we recall positive solutions of minimal growth and discuss a necessary condition for Hardy-weights. We first give a useful notation.

\bnotation
\emph{Fix a positive function~$u\in W^{1,p}_{\loc}(\Omega)\cap C(\Omega)$. For every compact $K\subseteq\Omega$, let~$$\mathcal{N}(K,u)\triangleq\{\phi\in W^{1,p}\cap C_{c}(\Omega)~|~\phi|_{K}\geq u\}.$$}
\enotation
The following definition is an extension of \cite[Definition 1.1]{Das}.
\begin{definition}[$Q$-capacity]\label{pavcap}
\emph{Fix a positive function~$u\!\in\!W^{1,p}_{\loc}(\Omega)\cap C(\Omega)$. For every compact subset~$K$ of $\Omega$, we define the \emph{$Q$-capacity} of~$K$ with respect to~$(u,\Omega)$ as 
\begin{align*}
\capacity(K,u)\triangleq\capacity(K,u,\Omega)\triangleq\capacity_{Q}(K,u) \triangleq\capacity_{p,\mathcal{A},V}(K,u)\triangleq \inf\big\{Q[\phi]~\big|~\phi\in \mathcal{N}(K,u)\big\}.
\end{align*}
}
\end{definition}
\bremark
\emph{When~$u=1$, the~$Q$-capacity coincides with the~$(\mathcal{A},V)$-capacity in \cite{HPR}.}
\eremark
Following \cite[Remarks 2.5 and 3.1 and Proposition 3.7]{Das}, we present some properties of the~$Q$-capacity. The proofs are analogous to the counterparts in \cite{Das} and hence omitted. 
\blemma\label{hfund}
\begin{enumerate}
    \item Suppose that~$\Omega_{1}\subseteq\Omega_{2}$ are both subdomains of~$\R^{n}$, that the norm family~$H$ satisfies Assumption \ref{ass9} in~$\Omega_{2}$, and that the potential~$V\in M^{q}_{\loc}(p;\Omega_{2})$. Fix a positive function~$u\!\in\!W^{1,p}_{\loc}(\Omega_{2})\cap C(\Omega_{2})$.  Then~$\capacity(K,u,\Omega_{2})\leq \capacity(K,u,\Omega_{1})$ for every compact~$K\subseteq\Omega_{1}\subseteq\Omega_{2}$.
    
    Next we fix a positive function~$u\!\in\!W^{1,p}_{\loc}(\Omega)\cap C(\Omega)$.
    \item If~$K_{1}\subseteq K_{2}$ are both compact subsets of~$\Omega$, then~$\capacity(K_{1},u)\leq \capacity(K_{2},u)$. 
    \item For all positive constants~$\alpha$ and all compact subsets~$K$ of~$\Omega$,$$\capacity(K,\alpha u)=\alpha^{p}\capacity(K,u).$$
\item For every compact subset~$K$ of~$\Omega$,~$\capacity(K,u)=0$ if and only if~$\capacity(K,1)=0$.
    \item For every compact~$K\subseteq\Omega$, let~$$\mathcal{N}'(K,u)\triangleq\{\psi~\mbox{is measurable}~|~\psi u\in W^{1,p}(\Omega)\cap C_{c}(\Omega),\mbox{and}~ \psi|_{K}\geq 1\}.$$
  Then  $$\capacity(K,u)=\inf\left\{Q[\psi u]~|~\psi\in\mathcal{N}'(K,u)\right\}.$$
  \item Assume that~$u$ is also a supersolution of the equation~$Q'[v]=0$ in~$\Omega$. Let~$K\subseteq\Omega$ be compact. Then
$$\capacity(K,u)=\inf\left\{Q[\phi]~|~\phi\in\mathcal{N}''(K,u)\right\},$$
where$$\mathcal{N}''(K,u)\triangleq\mathcal{N}_{p,\mathcal{A},V}''(K,u)\triangleq\{\phi\in W^{1,p}\cap C_{c}(\Omega)~|~\phi|_{K}=u, \mbox{and}~0\leq\phi\leq u\}.$$
\item Suppose, in addition to Assumptions \ref{ass9}, that there exists a positive constant~$\kappa_{\Omega}$ such that~$\kappa_{\Omega}|\xi|\leq H(x,\xi)$ for all~$\xi\in\R^{n}$ and almost all~$x\in\Omega$. When~$V$ is nonnegative in~$\Omega$ and~$1<p<n$, there exists a positive constant~$C$ such that for all compact subsets~$K$ of~$\Omega$, $$\vol(K)^{p/p^{*}}\leq C\capacity(K,1),$$
 and for all~$g\in L^{n/p}(\Omega)$,
 $$\int_{K}|g|\dx\leq C\capacity(K,1)\Vert g\Vert_{L^{n/p}(\Omega)}.$$
\item Suppose, in addition to Assumptions \ref{ass9}, that Assumptions \ref{ass2} and \ref{ngradb} hold. Then the functional~$Q$ is critical in~$\Omega$ if and only if~$\capacity(K,u)=0$ for all (some) compact subsets~$K$ of~$\Omega$ such that~$\vol(K)>0$. Furthermore, under all the conditions in (7),~$Q$ is subcritical in~$\Omega$.
\end{enumerate}
\elemma
Now we define the $Q$-Hardy constant~$S_{g}$,~$\Vert g\Vert_{\mathcal{H}}$, and~$\Vert g\Vert_{u}$ for all~$g\in\mathcal{H}(\Omega)$ even though we also consider some larger function spaces.
\begin{definition}\label{hardy_weight_norm}
\begin{enumerate}
\item [\emph{(1)}] \emph{Let~$\mathcal{M}(\Omega)$ be the linear space of all measurable functions on~$\Omega$. For every $g\in\mathcal{M}(\Omega)$, we define: 
\begin{equation*}
S_{g}\triangleq S_{g}(\Omega)\triangleq\inf\bigg\{\frac{Q[\phi]}{\int_{\Omega}|g||\phi|^{p}\dx}~\bigg|~\phi\in W^{1,p}(\Omega)\cap C_{c}(\Omega) \mbox{~and} \int_{\Omega}|g||\phi|^{p}\dx>0\bigg\}
\end{equation*}
and~$\Vert g\Vert_{\mathcal{H}}\triangleq(S_{g})^{-1}$ (possibly~$\infty$}). For every~$g\in\mathcal{H}(\Omega)$,~$S_{g}$ is called the \emph{$Q$-Hardy constant or simply, the \emph{Hardy constant}.} 
\item[\emph{(2)}]\emph{Fix a positive~$u\in W^{1,p}_{\loc}(\Omega)\cap C(\Omega)$. For every $g\in L^{1}_{\loc}(\Omega)$, let
	$$\Vert g\Vert_{u}\triangleq \inf\bigg\{L\in [0,\infty)~\bigg|~\int_{K}|g||u|^{p}\dx\leq L\capacity(K,u)~\mbox{for all compact}~K\subseteq\Omega\bigg\},$$
where $\Vert g\Vert_{u}$ may be~$\infty$.}
\end{enumerate}
\end{definition}
The lemma below easily follows from Definition \ref{hardy_weight_norm}. Thus the proof is omitted.
\blemma\label{hcrem}
\begin{enumerate}
\item For all~$g\in\mathcal{M}(\Omega)$,
$$\Vert g\Vert_{\mathcal{H}}=\sup\bigg\{\frac{\int_{\Omega}|g||\phi|^{p}\dx}{Q[\phi]}~\bigg|~\phi\in W^{1,p}(\Omega)\cap C_{c}(\Omega) \mbox{~and} \int_{\Omega}|g||\phi|^{p}\dx>0\bigg\}.$$
\item For all~$g\in\mathcal{H}(\Omega)$, we have~$0<S_{g}\leq\infty$,~$0\leq \Vert g\Vert_{\mathcal{H}}<\infty$, and
\begin{align*}
S_{g}=\inf\bigg\{Q[\phi]~\bigg|~\phi\in W^{1,p}(\Omega)\cap C_{c}(\Omega) \mbox{~and} \int_{\Omega}|g||\phi|^{p}\dx=1\bigg\}.
\end{align*}
\item It holds that~$\mathcal{H}(\Omega)=\{g\in\mathcal{M}(\Omega)\,|\,\Vert g\Vert_{\mathcal{H}}<\infty\}=\{g\in L^1_\loc(\Omega)\,|\,\Vert g\Vert_{\mathcal{H}}<\infty\}$.
\item For all~$g\in L^{1}_{\loc}(\Omega)$,$$\frac{1}{\Vert g\Vert_{u}}=\sup\bigg\{\ell\in (0,\infty]~\bigg|~\ell \int_{K}|g||u|^{p}\dx\leq\capacity(K,u)~\mbox{for all compact}~K\subseteq\Omega\bigg\}.$$
\end{enumerate}
\elemma
Motivated by \cite[pp.~333--334]{Das}, we discuss the functions $\Vert \cdot\Vert_{\mathcal{H}}$ and $\Vert \cdot\Vert_{u}$ in the following lemma. 
\blemma\label{bhome}
\begin{enumerate}
\item The map $\Vert \cdot\Vert_{\mathcal{H}}:\mathcal{M}(\Omega)\rightarrow[0,\infty]$ 
is positive definite and homogeneous and satisfies the triangle inequality. Moreover,~$\Vert \cdot\Vert_{\mathcal{H}}$ is a norm on~$\mathcal{H}(\Omega)$.
\item 
For every positive~$u\in W^{1,p}_{\loc}(\Omega)\cap C(\Omega)$, the map $\Vert \cdot\Vert_{u}:L^{1}_{\loc}(\Omega)\rightarrow[0,\infty]$ is positive definite and homogeneous and satisfies the triangle inequality. Moreover, $\Vert g\Vert_{u}\leq \Vert g\Vert_{\mathcal{H}}$ for all~$g\in\mathcal{H}(\Omega)$. Furthermore, $\Vert \cdot\Vert_{u}$ is a norm on~$\mathcal{H}(\Omega)$.
\end{enumerate}
\elemma
\bproof
(1): The proof is easy and hence omitted.

(2):
Suppose that~$g\in\mathcal{H}(\Omega)$.
The proof for~$\Vert g\Vert_{u}<\infty$ is similar to the counterpart in \cite[Theorem 1.2]{Das}. Let~$\psi$ be an arbitrary measurable function such that~$\psi u\in W^{1,p}(\Omega)\cap C_{c}(\Omega)$ and~$\psi\geq 1$ on~$K$ for a fixed compact subset~$K$ of~$\Omega$. Then
$$\int_{K}|g||u|^{p}\dx\leq \int_{\Omega}|g||\psi u|^{p}\dx\leq\Vert g\Vert_{\mathcal{H}}Q[\psi u].$$ It follows that
$$\int_{K}|g||u|^{p}\dx\leq \Vert g\Vert_{\mathcal{H}}\capacity_{Q}(K,u).$$ Furthermore, $\Vert g\Vert_{u}\leq \Vert g\Vert_{\mathcal{H}}<
\infty$. The other conditions can be easily proved.
\eproof
Now we aim to obtain the completeness of~$(\mathcal{H}(\Omega),\Vert\cdot\Vert_{\mathcal{H}})$. 
\bdefinition[{\cite[Section 30]{Bfs61}}]\label{Bfs}
\emph{Let~$(\mathcal{M}(\Omega),\Vert\cdot\Vert)$ 
	(here $\Vert\cdot\Vert$ can take~$\infty$) be a normed linear space. Suppose that
\begin{itemize}
    \item for all~$f\in \mathcal{M}(\Omega)$,~$\Vert f\Vert=\Vert |f|\Vert$;
    \item for every nonnegative sequence~$\{f_{k}\}_{k\in\mathbb{N}}\subseteq \mathcal{M}(\Omega)$ increasing to~$f$,~$\{\Vert f_{k}\Vert\}_{k\in\mathbb{N}}$ increases to~$\Vert f\Vert$.
\end{itemize}
Then~$\mathcal{M}'(\Omega)\triangleq\{f\in\mathcal{M}(\Omega)~|~\Vert f\Vert<\infty\}$ is called a \emph{Banach function space}.
 }
\edefinition
\blemma[{\cite[Section 30, Theorem 2]{Bfs61}}]\label{Bfbs}
The normed linear space~$\mathcal{M}'(\Omega)$ in Definition \ref{Bfs} is a Banach space.
\elemma
The following corollary follows immediately from Lemma \ref{hcrem} (1) and (3) and Lemma \ref{Bfbs}. See also \cite[p.~333]{Das}.
\bcorollary\label{pshsp}
The normed linear space~$(\mathcal{H}(\Omega),\Vert\cdot\Vert_{\mathcal{H}})$ of Hardy-weights is a Banach function space and hence a Banach space.
\ecorollary
We define smooth and Lipschitz  exhaustions of a domain, where `smooth' means~$C^{1}$. 
\bdefinition
\emph{A \emph{smooth (Lipschitz) exhaustion} of~$\Omega$ is a sequence of smooth (Lipschitz) domains~$\{\omega_{k}\}_{k\in\mathbb{N}}$ such that for all~$k\in\mathbb{N}$,~$\omega_{k}\Subset\omega_{k+1}\Subset\Omega$ and~$\cup_{k\in\mathbb{N}}\omega_{k}=\Omega$.}
\edefinition
We recall the definition of positive solutions of minimal growth.
 \begin{Def}\label{dfnmg}
 \emph{
 Let $K_{0}$ be a compact subset of~$\Omega$ such that~$\Omega\setminus K_{0}$ is a domain. A positive solution~$u$ of~$Q'[w]=0$ in~$\Omega\setminus K_{0}$ is called a \emph{positive solution of minimal growth in a neighborhood of infinity} in $\Omega$ if for all compact subsets~$K$ of~$\Omega$ such that~$K_{0}\Subset \mathring{K}\neq\emptyset$ and~$\omega_{i}\setminus K$ is a Lipschitz domain for some Lipschitz exhaustion~$\{\omega_{i}\}_{i\in\mathbb{N}}$ of~$\Omega$ and all~$i\in\mathbb{N}$, and all positive solutions~$v\in C\left(\Omega\setminus \mathring{K}\right)$ of~$Q'[w]=g$ in~$\Omega\setminus K$ satisfying~$u\leq v$ on~$\partial K$, where $g\in M^{q}_{\loc}(p;\Omega\setminus K)$ is nonnegative, it holds that~$ u\leq v$ in~$\Omega\setminus K$. 
 If~$K_{0}=\emptyset$, then $u$ is called a \emph{global minimal positive solution} of~$Q'[w]=0$ in~$\Omega$.}
  \end{Def}
  The following theorem, a counterpart of \cite[Proposition 3.4]{Das}, provides a necessary condition for Hardy-weights. 
  See also \cite[Theorem 7.9]{HPR} and \cite{Kovarik}.
\btheorem\label{mdecay}
Suppose, in addition to Assumptions \ref{ass9}, that Assumptions \ref{ass2} and\vspace{1mm} \ref{ngradb} hold and $V\in M^{q}_{\loc}(p;\Omega)$. Let~$K$ be a compact subset of~$\Omega$ with nonempty interior such that~$\Omega\setminus K$ is a domain, and~$u\in W^{1,p}_{\loc}(\Omega\setminus K)$ be a positive solution of~$Q'[v]=0$ of minimal growth in a neighborhood of infinity in~$\Omega$. 

Then for all~$g\in\mathcal{H}(\Omega)$, and all compact subsets~$\mathcal{K}$ of~$\Omega$ such that~$K\Subset\mathring{\mathcal{K}}$ and~$\omega_{i}\setminus \mathcal{K}$ is a Lipschitz domain for some Lipschitz exhaustion~$\{\omega_{i}\}_{i\in\mathbb{N}}$ of~$\Omega$ and all~$i\in\mathbb{N}$,
$$\int_{\Omega\setminus \mathcal{K}}|g|u^{p}\dx<\infty.$$
\etheorem
 \bproof
 The proof is similar to that of \cite[Proposition 3.4]{Das}. Obviously, we may assume that~$Q$ is subcritical in~$\Omega$. By \cite[Proposition 6.18]{HPR}, we may find a nonnegative function~$V_{K}\in C_{c}^{\infty}(\Omega)$ supported in~$\mathring{K}$ such that~$Q_{p,\mathcal{A},V-V_{K}}$ is critical in~$\Omega$. By Lemma \ref{hfund} (8),~$\capacity_{p,\mathcal{A},V-V_{K}}(G,\psi)=0$ for all compact sets~$G\subseteq\Omega$ with positive measure, where~$\psi$ is the ground state of~$Q_{p,\mathcal{A},V-V_{K}}$. \vspace{1mm}Let~$\{\omega_{k}'\}_{k\in\mathbb{N}}$ be a smooth exhaustion of~$\Omega$. Then for all~$k\in\mathbb{N}$, $\capacity_{p,\mathcal{A},V-V_{K}}\left(\overline{\omega_{k}'},\psi\right)=0$. Hence, for every~$k\in\mathbb{N}$, there exists~$\phi_{k}\in \mathcal{N}''_{p,\mathcal{A},V-V_{K}}\left(\overline{\omega_{k}'},\psi\right)$ such that~$Q_{p,\mathcal{A},V-V_{K}}[\phi_{k}]<1/k$. It follows that
 \begin{align*}
 \int_{\omega_{k}'}|g|\psi^{p}\dx= \int_{\omega_{k}'}|g|\phi_{k}^{p}\dx\leq CQ[\phi_{k}]&=CQ_{p,\mathcal{A},V-V_{K}}[\phi_{k}]+C\int_{\Omega}V_{K}\phi_{k}^{p}\dx\leq \frac{C}{k}+C\int_{K}V_{K}\psi^{p}\dx.
 \end{align*}
 The estimate~$\int_{\Omega}|g|\psi^{p}\dx<\infty$ follows when~$k$ goes to~$\infty$. 
 As~$V_{K}$ is supported in~$\mathring{K}$,~$\psi$ is a positive solution of~$Q'[v]=0$ in~$\Omega\setminus K$. 
 By virtue of the homogeneity of the equation,~$u\leq C'\psi$ in~$\Omega\setminus \mathcal{K}$ for some positive constant~$C'$. Thus~$\int_{\Omega\setminus \mathcal{K}}|g|u^{p}\dx<\infty$.
 \eproof
\section{The Bregman distances of~$|\cdot|^{p}_{s,a}$ and the Maz'ya-type characterization}\label{ready}
In this section, we first prove some two-sided estimates for the Bregman distances of~$|\cdot|^{p}_{s,a}$ (see Definition \ref{sanorm} below) and three corresponding simplified energies contingent on~$p=s$, $p>s$, or~$p<s$. Then we demonstrate a Maz'ya-type characterization concerning~$|\cdot|_{s,a}$ by virtue of these simplified energies. Finally, we also consider the  norm~$\sqrt[p]{|\cdot|^{p}_{s,a}+|\cdot|^{p}_{A}}$ and prove the Maz'ya-type characterization in this context.
\subsection{The Bregman distance}\label{bregmans}
In this subsection, we define the Bregman distances and discuss the differentiability of~$|\cdot|_{s,a}$ and~$|\cdot|^{p}_{s,a}$. 

For a continuously differentiable function~$f$ defined in~$\R$, the Bregman distance of~$f$ is the difference between~$f$ and its tangent line. See \cite{BR06} for more on the Bregman distances. 
\begin{definition}\label{BREGMAN}
	\emph{Let~$f: X\rightarrow \mathbb{R}$ be a G\^{a}teaux differentiable convex function on a Banach space~$(X,\Vert\cdot\Vert)$, and denote by $T_x$ its G\^{a}teaux differential at $x\in X$. Then the function
		\begin{eqnarray*}
			D_{f}(y,x)\triangleq  f(y)-f(x)-T_{x}(y-x) \qquad (x,y)\in X \times X,
		\end{eqnarray*}
		is called the \emph{Bregman distance} of~$f$.}
\end{definition}
Next, we define the \emph{weighted $s$-norm}~$|\cdot|_{s,a}$ ($1<s<\infty$). 
\bdefinition\label{sanorm}
\emph{For every~$1<s<\infty$, almost all~$x\in\Omega$, and all~$\xi\in\R^{n}$, let \bequationn
	|\xi|_{s,a}\triangleq \bigg(\sum_{i=1}^{n}a_{i}(x)|\xi_{i}|^{s}\bigg)^{1/s},
	\eequationn
	where~$a_{1},a_{2},\ldots,a_{n}$ are positive Lebesgue measurable functions defined in $\Omega$ and equivalent to~$1$ locally in~$\Omega$, and
	$a\triangleq(a_{1},a_{2},\ldots,a_{n})$.}
\edefinition
In the present paper, we consider the Bregman distances of~$|\xi|_{s,a}^{p}$ in~$\R^{n}$, namely
$$D_{|\cdot|_{s,a}^{p}}(\xi+\eta,\xi)=|\xi+\eta|_{s,a}^{p}-|\xi|_{s,a}^{p}-p|\xi|_{s,a}^{p-s}\sum_{i=1}^{n}a_{i}(x)\left| \xi_{i}\right|^{s-2}\xi_{i}\eta_{i}$$ for almost all~$x\in\Omega$ and all~$\xi,\eta\in\R^{n}$.

We recall the partial derivatives of~$|\cdot|_{s,a}$ whose calculation is nontrivial due to vanishing coordinates.
\begin{lemma}[{\cite[pp. 55--56]{Coleman}}]\label{qnorm}
	\emph{ For all~$s\in (1,\infty)$ and almost all~$x\in\Omega$, $|\cdot |_{s,a}\in C^1(\mathbb{R}^{n}\setminus\{0\})$. Indeed, for all~$i=1,2,\ldots,n$, and all $\xi\in\mathbb{R}^{n}\setminus\{0\}$, 
		$$\partial_{i}|\xi|_{s,a}=a_{i}(x)|\xi|_{s,a}^{1-s}|\xi_{i}|^{s-2}\xi_{i}.$$
	}
\end{lemma}
With the continuous differentiability of~$|\cdot|_{s,a}$ in~$\R^{n}\setminus\{0\}$ at our disposal, we may readily obtain the continuous differentiability of~$|\cdot|_{s,a}^{p}$ in~$\R^{n}$.
\begin{corollary}\label{diffe}
	\emph{ For all $p,s\in (1,\infty)$ and almost all~$x\in\Omega$, $|\cdot|_{s,a}^{p}\in C^1(\mathbb{R}^{n})$. In fact, for all~$i=1,2,\ldots,n$, and all $\xi\in\mathbb{R}^{n}$, $$\partial_{i}|\xi|_{s,a}^{p}=pa_{i}(x)|\xi|_{s,a}^{p-s}\left| \xi_{i}\right|^{s-2}\xi_{i}.$$}
\end{corollary}
Now we give some specific examples such that solutions of the relevant equations are in~$W^{1,\infty}_{\loc}$. In particular, these examples satisfy Assumptions \ref{ass9}, \ref{ass2}, and \ref{ngradb}. See
 \cite[Remark 2.2]{Mezei} for the particular class of norms $\sqrt{\lambda|\xi|_{4}^{2}+\mu|\xi|^{2}}$ ($\lambda,\mu>0$).
\bexample\label{gbexa}
\emph{Let~$H(x,\xi)=\sqrt[p]{|\xi|_{s,a}^{p}+|\xi|_{A}^{p}}$ ($2\leq s<\infty$) for all~$x\in\Omega$ and~$\xi\in\R^{n}$, where the measurable matrix function~$A: \Omega\rightarrow\R^{n^{2}}$ satisfies: 
\begin{enumerate}
\item For all~$x\in\Omega$,~$A(x)$ is symmetric;
\item for every domain~$\omega\Subset\Omega$, there exist positive constants~$\theta_{1,\omega}$ and~$\theta_{2,\omega}$ such that for all~$x\in\omega$ and~$\xi\in\R^{n}$, $$\theta_{1,\omega}|\xi|\leq|\xi|_{A}=\sqrt{A(x)\xi\cdot\xi}\leq \theta_{2,\omega}|\xi|.$$
\end{enumerate} 
In addition, suppose that for every domain~$\omega\Subset\Omega$, there exist positive constants~$\vartheta_{\omega}\leq 1$ and~$\Lambda'_{3,\omega}$ such that~$a$ is H\"older continuous in~$\omega$ of order~$\vartheta_{\omega}$, and for all~$x,y\in\omega$ and~$\xi\in\R^{n}$,
\begin{align*}
||\xi|^{p-2}_{A(x)}A(x)\xi-|\xi|^{p-2}_{A(y)}A(y)\xi|\leq \Lambda'_{3,\omega}|\xi|^{p-1}|x-y|^{\vartheta_{\omega}}\quad\mbox{and}\quad|A(x)|\leq \Lambda'_{3,\omega}.
\end{align*} 
Then~$H$ satisfies Assumptions \ref{ass9} and \ref{ass2} with~$\tilde{\delta}=2$ by  Remark \ref{revaf}, Corollary \ref{diffe}, and
Lemmas \ref{bregul} and \ref{ABE}, and for all~$x\in\Omega$,~$\mathcal{A}(x,\xi)\in C^{1}(\R^{n}\setminus\{0\})$. Moreover, for every domain~$\omega\Subset\Omega$, there exist positive constants~$\vartheta_{\omega}\leq 1$,~$\Lambda_{1,\omega}$,~$\Lambda_{2,\omega}$, and~$\Lambda_{3,\omega}$ such that for all~$x,y\in\omega$, all~$\xi\in\R^{n}\setminus\{0\}$, and all~$\eta\in\R^{n}$,
\begin{align*}
D_{\xi}\mathcal{A}(x,\xi)\eta\cdot\eta\geq \Lambda_{1,\omega}|\xi|^{p-2}|\eta|^{2},\quad |D_{\xi}\mathcal{A}(x,\xi)|\leq \Lambda_{2,\omega}|\xi|^{p-2},
\end{align*}
and \begin{align*}
|\mathcal{A}(x,\xi)-\mathcal{A}(y,\xi)|\leq \Lambda_{3,\omega}|\xi|^{p-1}|x-y|^{\vartheta_{\omega}}.
\end{align*}
We also assume that~$V\in \widehat{M}^{q}_{\loc}(p;\Omega)$.
Then by Theorem \ref{gethm}, in all subdomains~$\Omega'$ of~$\Omega$, solutions of the equations $Q'_{p,\mathcal{A},V-W}[u]=0$ for all~$W\in L^{\infty}_{\loc}(\Omega')$ and~$Q'[u]=f$ for all~$f\in C_{c}^{\infty}(\Omega')$ are in~$W^{1,\infty}_{\loc}(\Omega')$.} 
\eexample
%
\subsection{Two-sided estimates for the Bregman distances of~$|\cdot|^{p}_{s,a}$ }\label{dbregman}
In this subsection, we develop some two-sided estimates for the Bregman distances of~$|\cdot|_{s,a}^{p}$. 

We first introduce some symbols which are only used in Sections \ref{dbregman}, \ref{sesec}, and \ref{secmazya}.
\begin{notation}
\emph{\begin{itemize}
\item For every~$\xi\in\R^{n}$ and almost all~$x\in\Omega$, we write~$\xi_{i}'\triangleq \sqrt[s]{a_{i}(x)}\xi_{i}$ ($1<s<\infty$);
\item for every $1<s<\infty$, 
let $m\triangleq\min\{s,2\}$ and~$M\triangleq \max\{s,2\}$;
\item for all~$\xi\in\R^{n}$, we write~$f(\xi)\triangleq|\xi|_{s,a}^{s}$;
%
\item for all~$1<l<\infty$,~$1<p\neq s<\infty$, and~$\xi,\eta\in\R^{n}$, we employ the symbols: 
$$R_{1}(\xi,\eta; l)\triangleq|\eta|_{s,a}^{l}(|\eta|_{s,a}+|\xi|_{s,a})^{p-l} \quad 
\mbox{and} \quad
R_{2}(\xi,\eta)\triangleq |\xi|_{s,a}^{p-s}\sum_{i=1}^{n}|\eta'_{i}|^{2}\left(|\eta'_{i}|+|\xi'_{i}|\right)^{s-2},$$
where in the definition of $R_{2}$, when~$p<s$ we further assume that $\xi\neq 0$;
\item in the case of $H(x,\xi)=|\xi|_{s,a}$, we replace the subscript `$\mathcal{A}$' with the subscript `$s,a$'.
\end{itemize}
}
\end{notation}
\subsubsection{$p=s$}
We first consider the Bregman distances of~$|\xi|_{p,a}^{p}$, corresponding to the \emph{weighted pseudo~$p$-Laplacian}, i.e.,~$p=s$ and
\bequationn
|\xi|_{p,a}^{p}=\sum_{i=1}^{n}a_{i}(x)|\xi_{i}|^{p}.
\eequationn
%
\blemma\label{psevec}
For almost all~$x\in\Omega$ and all~$\xi,\eta\in\R^{n}$,
$$|\xi+\eta|_{p,a}^{p}-|\xi|_{p,a}^{p}-p\sum_{i=1}^{n}a_{i}(x)|\xi_{i}|^{p-2}\xi_{i}\eta_{i}\asymp\sum_{i=1}^{n}|\eta'_{i}|^{2}\left(|\eta'_{i}|+|\xi'_{i}|\right)^{p-2},$$ 
where both equivalence constants depend only on~$p$.
\elemma
\bproof
By virtue of Lemma~\ref{ABE}, we have, for all $\xi,\eta\in\R^{n}$, $i=1,2,\ldots,n$, and almost all~$x\in\Omega$,
$$a_{i}(x)|\xi_{i}+\eta_{i}|^{p}-a_{i}(x)|\xi_{i}|^{p}-pa_{i}(x)|\xi_{i}|^{p-2}\xi_{i}\eta_{i}\asymp |\eta'_{i}|^{2}\left(|\eta'_{i}|+|\xi'_{i}|\right)^{p-2},$$ where both equivalence constants depend only on~$p$. 
Therefore, \begin{equation*}|\xi+\eta|_{p,a}^{p}-|\xi|_{p,a}^{p}-p\sum_{i=1}^{n}a_{i}(x)|\xi_{i}|^{p-2}\xi_{i}\eta_{i}\asymp\sum_{i=1}^{n}|\eta'_{i}|^{2}\left(|\eta'_{i}|+|\xi'_{i}|\right)^{p-2}.\qquad \qedhere
\end{equation*}
\eproof
\subsubsection{$1<p,s<\infty$ ($p\neq s$)}
Next we consider~$|\xi|^{p}_{s,a}$ for $1<p,s<\infty$ ($p\neq s$). 
The following simple lemma, a \emph{chain rule} for the Bregman distance, is found in the proof of \cite[Proposition 1.2.7]{BI00}, and \cite[(1.108)]{BI00}. 
\blemma[{cf. \cite[Proposition 1.2.7 and (1.108)]{BI00}}]\label{bdde}
Let~$r>1$ be a constant and~$h:\R^{n}\rightarrow[0,\infty)$ be a positive definite function such that~$h\in C^{1}(\R^{n}\setminus\{0\}) $,~$h^{r}\in C^{1}(\R^{n})$, and~$\nabla(h^{r})(0)=0$. 
Then for all~$\xi,\eta\in\R^{n}$,
\begin{align*}
 D_{h^{r}}(\zeta,\xi)&=h(\zeta)^{r}-h(\xi)^{r}-rh(\xi)^{r-1}\nabla h(\xi)\cdot\eta\\
	&=h(\zeta)^{r}-h(\xi)^{r}-rh(\xi)^{r-1}(h(\zeta)-h(\xi))+rh(\xi)^{r-1}(h(\zeta)-h(\xi)-\nabla h(\xi)\cdot\eta),
\end{align*}
where~$\zeta\triangleq\xi+\eta$ and $\nabla h(0)\triangleq 0$.
\elemma
Note that  for all~$1<s<p<\infty$, Lemma \ref{pBE} implies that for all~$\xi,\eta\in \R^{n}$,
$$(|\xi+\eta|_{s,a}^{s})^{p/s}-(|\xi|_{s,a}^{s})^{p/s}-p/s(|\xi|_{s,a}^{s})^{p/s-1}(|\xi+\eta|_{s,a}^{s}-|\xi|_{s,a}^{s})\geq 0.$$ 
Then the following lemma is readily derived from Lemmas \ref{psevec} and \ref{bdde}. 
\blemma\label{pgs1}
Suppose that $1<s<p<\infty$. 
Then there exists a positive constant~$c(p,s)$ such that for almost all~$x\in\Omega$ and all~$\xi,\eta\in\R^{n}$,
\begin{align*}
D_{|\cdot|_{s,a}^{p}}(\xi+\eta,\xi)&=
|\xi+\eta|_{s,a}^{p}-|\xi|_{s,a}^{p}-p|\xi|_{s,a}^{p-s}\sum_{i=1}^{n}a_{i}(x)\left| \xi_{i}\right|^{s-2}\xi_{i}\eta_{i}\\
&\geq c(p,s)|\xi|_{s,a}^{p-s}\sum_{i=1}^{n}|\eta'_{i}|^{2}\left(|\eta'_{i}|+|\xi'_{i}|\right)^{s-2}=c(p,s)R_{2}(\xi,\eta).
\end{align*}
\elemma
\blemma\label{pls1}
Suppose that $1<p<s<\infty$. 
Then there exists a positive constant~$C(p,s)$ such that for almost all~$x\in\Omega$ and all~$\xi,\eta\in\R^{n}$ with $\xi\neq 0$,
\begin{align*}
D_{|\cdot|_{s,a}^{p}}(\xi+\eta,\xi) \leq C(p,s)|\xi|_{s,a}^{p-s}\sum_{i=1}^{n}|\eta'_{i}|^{2}\left(|\eta'_{i}|+|\xi'_{i}|\right)^{s-2}=C(p,s)R_{2}(\xi,\eta).
\end{align*}
\elemma
\bproof
First note that
\begin{align*}	
		D_{|\cdot|_{s,a}^{p}}(\xi+\eta,\xi)&=|\xi+\eta|_{s,a}^{p}-|\xi|_{s,a}^{p}-p|\xi|_{s,a}^{p-s}\sum_{i=1}^{n}a_{i}(x)\left| \xi_{i}\right|^{s-2} \xi_{i}\eta_{i}\\
	&=|\xi|_{s,a}^{p-s}\left( |\xi+\eta|_{s,a}^{p}|\xi|_{s,a}^{s-p}-|\xi|_{s,a}^{s}-p\sum_{i=1}^{n}a_{i}(x)\left| \xi_{i}\right|^{s-2} \xi_{i}\eta_{i}\right).	
\end{align*}
Then a standard application of Young's inequality yields
\begin{align*}
		D_{|\cdot|_{s,a}^{p}}(\xi+\eta,\xi)
&\leq |\xi|_{s,a}^{p-s}\left( \frac{p}{s}|\xi+\eta|_{s,a}^{s}+\left(1-\frac{p}{s}\right)|\xi|_{s,a}^{s}-|\xi|_{s,a}^{s}-p\sum_{i=1}^{n}a_{i}(x)\left| \xi_{i}\right|^{s-2} \xi_{i}\eta_{i}\right)\\
&=\frac{p}{s}|\xi|_{s,a}^{p-s}\left(|\xi+\eta|_{s,a}^{s}-|\xi|_{s,a}^{s}-s\sum_{i=1}^{n}a_{i}(x)\left| \xi_{i}\right|^{s-2} \xi_{i}\eta_{i}\right)\\
&=\frac{p}{s}|\xi|_{s,a}^{p-s}D_{|\cdot|_{s,a}^{s}}(\xi+\eta,\xi) 
\leq C(p,s)R_{2}(\xi,\eta),	
\end{align*}
where the last inequality is due to Lemma \ref{psevec}.
\eproof
The following lemma can be easily proved.
\blemma
In every normed linear space 
$(X,\Vert\cdot\Vert)$, for all~$x,y\in X$, 
$$\max\{\Vert x\Vert,\Vert y\Vert\}\asymp \Vert x\Vert+\Vert y-x\Vert.$$
\elemma
\noindent Hence,  we easily conclude the next lemma from \cite[(1.2), (1.4), and Remarks 1 and 4]{Roach}. 
\blemma\label{bregul}
Suppose that $1<p,s<\infty$. 
Then there exist positive constants~$c(p,s)$ and~$C(p,s)$ such that for almost all~$x\in\Omega$ and all~$\xi,\eta\in\R^{n}$,
\begin{align*}
c(p,s)R_{1}(\xi,\eta;M)&=c(p,s)|\eta|_{s,a}^{M}(|\eta|_{s,a}+|\xi|_{s,a})^{p-M}
\leq D_{|\cdot|_{s,a}^{p}}(\xi+\eta,\xi)\\
&\leq C(p,s)|\eta|_{s,a}^{m}(|\eta|_{s,a}+|\xi|_{s,a})^{p-m}=C(p,s)
	R_{1}(\xi,\eta;m).
\end{align*}
%
\elemma
As a byproduct, we give a large class of operators satisfying Assumptions \ref{ass9} and \ref{ass2}. 
\begin{example}\emph{Let~$H(x,\xi)=|\xi|_{s,a}$ for all~$\xi\in\R^{n}$ and almost all~$x\in\Omega$. Note that for almost all~$x\in\Omega$,~$|\cdot|_{s,a}$ is uniformly convex (see \cite[p.\,192]{Roach}) and continuously differentiable in~$\R^{n}\setminus\{0\}$ by Lemma \ref{qnorm}. Then~$H$ satisfies Assumption \ref{ass9}. By Lemma \ref{bregul}, for all~$1<p,s<\infty$,  almost all~$x\in\Omega$, and all~$\xi,\eta\in\R^{n}$,
\begin{equation*}
D_{|\cdot|_{s,a}^{p}}(\xi+\eta,\xi)\geq c(p,s)|\eta|_{s,a}^{M}(|\eta|_{s,a}+|\xi|_{s,a})^{p-M},
\end{equation*}
for some positive constant~$c(p,s)$. This verifies Assumption \ref{ass2}.
}
\end{example}
\blemma\label{pgls}
Suppose that $1<s,p<\infty$. 
Then for almost all~$x\in\Omega$, all~$\xi,\eta\in\R^{n}$ such that~$|\xi|_{s,a}\leq|\eta|_{s,a}$, and all~$l\in [m, M]$, we have
\bequationn
D_{|\cdot|_{s,a}^{p}}(\xi+\eta,\xi)\asymp |\eta|_{s,a}^{l}(|\eta|_{s,a}+|\xi|_{s,a})^{p-l}=R_{1}(\xi,\eta;l),
\eequationn
where both equivalence constants depend only on~$p$ and~$s$.
\elemma
\bproof
Due to Lemma \ref{bregul}, the desired conclusion boils down to
\bequation\label{R1mR1M}
	cR_{1}(\xi,\eta;m)\leq 	R_{1}(\xi,\eta;l)\leq CR_{1}(\xi,\eta;M),
\eequation
for some positive constants~$C$ and~$c$. 
Under the hypothesis that~$|\xi|_{s,a}\leq|\eta|_{s,a}$, the inequalities \eqref{R1mR1M} can be easily verified.
\eproof
\blemma\label{pgs2}
Suppose that $1<s<p<\infty$. 
Then there exists a positive constant~$C(p,s,n)$ such that for almost all~$x\in\Omega$ and all~$\xi,\eta\in\R^{n}$ with~$|\xi|_{s,a}>|\eta|_{s,a}$,
\begin{align*}
D_{|\cdot|_{s,a}^{p}}(\xi+\eta,\xi)\leq C(p,s,n)|\xi|_{s,a}^{p-s}\sum_{i=1}^{n}|\eta'_{i}|^{2}\left(|\eta'_{i}|+|\xi'_{i}|\right)^{s-2}= C(p,s,n)R_{2}(\xi,\eta).
\end{align*}
\elemma
\bproof
Exploiting Lemmas \ref{bdde} and \ref{pBE}, we conclude that for~$\zeta=\xi+\eta$ and $f(\cdot)=|\cdot|_{s,a}^{s}$,
\begin{align*}
 &D_{|\cdot|_{s,a}^{p}}(\zeta,\xi)\\
	&=f(\zeta)^{p/s}-f(\xi)^{p/s}-\frac{p}{s}f(\xi)^{p/s-1}(f(\zeta)-f(\xi))+\frac{p}{s}f(\xi)^{p/s-1}(f(\zeta)-f(\xi)-\nabla f(\xi)\cdot\eta)\\
	&\leq\! C(p/s)|f(\zeta)\!-\!f(\xi)|^{2}(|f(\zeta)\!-\!f(\xi)|\!+\!f(\xi))^{p/s-2}\!\!+\!C(s,p/s)f(\xi)^{p/s-1}\!\sum_{i=1}^{n}|\eta'_{i}|^{2}\left(|\eta'_{i}|\!+\!|\xi'_{i}|\right)^{s-2}\!\!.
\end{align*}
Since~$|\xi|_{s,a}>|\eta|_{s,a}$, for all~$1<s<p<\infty$, it can be plainly verified, respectively for~$p/s>2,~p/s=2,$ and~$1<p/s<2$, that
$$(|f(\zeta)-f(\xi)|+f(\xi))^{p/s-2}\leq C(s,p/s)f(\xi)^{p/s-2}.$$ Therefore, we are left with the task of controlling~$|f(\zeta)-f(\xi)|^{2}$ from above with $$f(\xi)\sum_{i=1}^{n}|\eta'_{i}|^{2}\left(|\eta'_{i}|+|\xi'_{i}|\right)^{s-2}$$ when~$|\xi|_{s,a}>|\eta|_{s,a}.$
By Lagrange's mean value theorem, we obtain the well-known inequality
	$$|x^{s}-y^{s}|\leq s|x-y|(x+y)^{s-1},$$
	for all~$x,y\geq 0$. 
 Then
\begin{align*}
		 &|f(\zeta)-f(\xi)|^{2}=\bigg|  \sum_{i=1}^{n}(|(\xi_{i}+\eta_{i})'|^{s}-|\xi'_{i}|^{s})\bigg|^{2}
	 \leq  C(s)\bigg( \sum_{i=1}^{n}|\eta'_{i}|(|\eta'_{i}|+|\xi'_{i}|)^{s-1}\bigg)^{2}\\
	 &\leq  C(s,n)\sum_{i=1}^{n}|\eta'_{i}|^{2}(|\eta'_{i}|+|\xi'_{i}|)^{2(s-1)}\leq  C(s,n)\sum_{i=1}^{n}|\eta'_{i}|^{2}(|\eta'_{i}|+|\xi'_{i}|)^{s-2}|\xi|_{s,a}^{s}.\qedhere 
\end{align*}
\eproof
\blemma\label{pls2}
Suppose that $1<p<s<\infty$. 
Then there exists a positive constant~$c(p,s)$ such that for almost all~$x\in\Omega$ and all~$\xi,\eta\in\R^{n}$ with~$|\xi|_{s,a}>|\eta|_{s,a}$,
$$D_{|\cdot|_{s,a}^{p}}(\xi+\eta,\xi)\geq c(p,s)|\xi|_{s,a}^{p-s}\sum_{i=1}^{n}|\eta'_{i}|^{2}\left(|\eta'_{i}|+|\xi'_{i}|\right)^{s-2}= c(p,s)R_{2}(\xi,\eta).$$
\elemma
\bproof
Clearly,~$\xi\neq 0$. Then Lemma \ref{bdde} implies that   
\begin{align}\label{glow}
		D_{|\cdot|_{s,a}^{p}}(\xi+\eta,\xi)&=D_{t^{p}}(|\xi+\eta|_{s,a},|\xi|_{s,a})+p|\xi|_{s,a}^{p-1}D_{|\cdot|_{s,a}}(\xi+\eta,\xi)\notag\\
	&\geq D_{t^{p}}(|\xi+\eta|_{s,a},|\xi|_{s,a})\notag\\
	&\geq c(p)\left||\xi+\eta|_{s,a}-|\xi|_{s,a}\right|^{2}\left(\left||\xi+\eta|_{s,a}-|\xi|_{s,a}\right|+|\xi|_{s,a}\right)^{p-2},
\end{align}
where the two inequalities are ensured respectively by 
the proof of \cite[Lemma 5.6]{HKM}, and Lemma \ref{pBE}. 
Note that $f(\xi)=|\xi|_{s,a}^{s}=(|\xi|_{s,a}^{p})^{s/p}$. Then 
by Lemma \ref{bdde}, 
$$D_{f}(\xi+\eta,\xi)=D_{t^{s/p}}(|\xi+\eta|_{s,a}^{p},|\xi|_{s,a}^{p})+\frac{s}{p}(|\xi|_{s,a}^{p})^{s/p-1}D_{|\cdot|_{s,a}^{p}}(\xi+\eta,\xi).$$
The lemma will be proved once we establish the inequality 
$$D_{f}(\xi+\eta,\xi)\leq C(p,s)|\xi|_{s,a}^{s-p}D_{|\cdot|_{s,a}^{p}}(\xi+\eta,\xi)=C(p,s)\frac{s}{p}(|\xi|_{s,a}^{p})^{s/p-1}D_{|\cdot|_{s,a}^{p}}(\xi+\eta,\xi),$$
because~$R_{2}(\xi,\eta)=|\xi|_{s,a}^{p-s}\sum_{i=1}^{n}|\eta'_{i}|^{2}\left(|\xi'_{i}|+|\eta'_{i}|\right)^{s-2}$, and by Lemma \ref{psevec},~$$D_{f}(\xi+\eta,\xi)\geq c(s)\sum_{i=1}^{n}|\eta'_{i}|^{2}\left(|\eta'_{i}|+|\xi'_{i}|\right)^{s-2}.$$ For this purpose, it suffices to show that
\begin{equation}\label{phig}
	D_{t^{s/p}}(|\xi+\eta|_{s,a}^{p},|\xi|_{s,a}^{p})\leq C(p,s)\frac{s}{p}(|\xi|_{s}^{p})^{s/p-1}D_{|\cdot|_{s,a}^{p}}(\xi+\eta,\xi). 
\end{equation}
By Lemma \ref{pBE}, we get
$$D_{t^{s/p}}(|\xi+\eta|_{s,a}^{p},|\xi|_{s,a}^{p})\leq C(s/p)\left||\xi+\eta|_{s,a}^{p}-|\xi|_{s,a}^{p}\right|^{2}\left(\left||\xi+\eta|_{s,a}^{p}-|\xi|_{s,a}^{p}\right|+|\xi|_{s,a}^{p}\right)^{s/p-2}. $$
Moreover,
\begin{align*}
\left||\xi+\eta|_{s,a}^{p}-|\xi|_{s,a}^{p}\right|&\leq p\left||\xi+\eta|_{s,a}-|\xi|_{s,a}\right|\left(|\xi+\eta|_{s,a}^{p-1}+|\xi|_{s,a}^{p-1}\right)\\
&\leq C(p)\left||\xi+\eta|_{s,a}-|\xi|_{s,a}\right||\xi|_{s,a}^{p-1}.
\end{align*}
On account of \eqref{glow}, if we prove that
\begin{align}\label{ofact}
&|\xi|_{s,a}^{2(p-1)}\left(\left||\xi+\eta|_{s,a}^{p}-|\xi|_{s}^{p}\right|+|\xi|_{s,a}^{p}\right)^{s/p-2}\notag\\
&\leq C(p,s)|\xi|_{s,a}^{s-p}\left(\left||\xi+\eta|_{s,a}-|\xi|_{s,a}\right|+|\xi|_{s,a}\right)^{p-2},	
\end{align} 
then the desired inequality \eqref{phig} follows. It is a simple matter to check, by virtue of the hypothesis $|\xi|_{s,a}>|\eta|_{s,a}$, that
$$\left||\xi+\eta|_{s,a}^{p}-|\xi|_{s,a}^{p}\right|+|\xi|_{s,a}^{p}\asymp|\xi|_{s,a}^{p}\quad \mbox{and}\quad\left||\xi+\eta|_{s,a}-|\xi|_{s,a}\right|+|\xi|_{s,a}\asymp|\xi|_{s,a},$$ and hence both sides in \eqref{ofact} are equivalent to~$|\xi|_{s,a}^{s-2}$, which implies \eqref{ofact}.
\eproof
\subsection{Simplified energies for~$Q_{p,s,a,V}$}\label{sesec}
In this subsection, we first recall a consequence of a Picone-type identity \cite[Lemma 4.9]{HPR}, and then deduce simplified energies for~$Q_{p,s,a,V}$.
\blemma[{\cite[Lemma 4.12]{HPR}}]\label{GRF}
For all positive
		solutions~$v\in W^{1,p}_{\loc}(\Omega)$ of $Q'[\psi]=0$ in~$\Omega$ and all nonnegative~$u\in W^{1,p}_{c}(\Omega)$ such that~$u^{p}/v^{p-1}\in W^{1,p}_{c}(\Omega)$, that the product and chain rules for~$u^{p}/v^{p-1}$ hold, and that~$vw$ satisfies the product rule for~$w\triangleq u/v$, it holds that
	$$Q[vw]=\int_{\Omega} \left(|v\nabla w+w\nabla v|^{p}_{\mathcal{A}}-w^{p}|\nabla v|^{p}_{\mathcal{A}}-pw^{p-1}v\mathcal{A}(x,\nabla v)\cdot\nabla w\right)\dx.$$
\elemma
By virtue of Lemma \ref{GRF} and the estimates in Section \ref{dbregman}, we may easily conclude the following simplified energies. 

Lemmas \ref{psevec} and \ref{GRF} imply the following lemma.
\begin{lem}[The simplified energy for~$Q_{p,p,a,V}$]\label{simpv} 
Fix a positive
solution $u\in W^{1,p}_{\loc}(\Omega)\cap C(\Omega)$ of $Q'_{p,p,a,V}[v]=0$ in $\Gw$.  Then for all nonnegative~$\psi\in W^{1,p}(\Omega)\cap C_{c}(\Omega)$,
$$Q_{p,p,a,V}[u\psi]\asymp\int_{\Omega}\sum_{i=1}^{n}u^{2}|(\partial_{i} \psi)'|^{2}\left(\psi|(\partial_{i} u)'|+u|(\partial_{i} \psi)'|\right)^{p-2}\dx,$$ 
where both equivalence constants depend only on~$p$.
\end{lem}
 Lemmas \ref{pgs1}, \ref{bregul}, \ref{pgls}, \ref{pgs2}, and \ref{GRF} imply the following lemma.
\begin{lem}[The simplified energy for~$Q_{p,s,a,V}:s<p$]\label{simpsv}
Suppose that $1<s<p<\infty$ and let~$H(x,\xi)=|\xi|_{s,a}$.
	Fix a positive
	solution $u\in W^{1,p}_{\loc}(\Omega)\cap C(\Omega)$ of $Q'_{p,s,a,V}[v]=0$ in $\Gw$. Then for some positive constants~$c(p,s),C(p,s)$, and~$C(p,s,n)$, and all nonnegative~$\psi\in W^{1,p}(\Omega)\cap C_{c}(\Omega)$,
		\bequationn
	Q_{p,s,a,V}[u\psi]\geq c(p,s)\int_{\Omega}R_{1}(\psi\nabla u,u\nabla\psi;M)\dx+c(p,s)\int_{\Omega}R_{2}(\psi\nabla u,u\nabla\psi)\dx,
	\eequationn
	and
	\bequationn
	Q_{p,s,a,V}[u\psi]\leq C(p,s)	\int_{\Omega_{1}}R_{1}(\psi\nabla u,u\nabla\psi;M)\dx+C(p,s,n)\int_{\Omega_{2}}R_{2}(\psi\nabla u,u\nabla\psi)\dx,
	\eequationn	
where $\Omega_{1}\triangleq \{x\in \Omega~|~|\psi\nabla u|_{s,a}\leq|u\nabla\psi|_{s,a}\}$ and  $\Omega_{2}\triangleq \{x\in \Omega~|~|\psi\nabla u|_{s,a}>|u\nabla\psi|_{s,a}\}$. In particular, for all nonnegative~$\psi\in W^{1,p}(\Omega)\cap C_{c}(\Omega)$,
\bequationn
	Q_{p,s,a,V}[u\psi]\asymp	\int_{\Omega_{1}}R_{1}(\psi\nabla u,u\nabla\psi;M)\dx+\int_{\Omega_{2}}R_{2}(\psi\nabla u,u\nabla\psi)\dx,
	\eequationn
 where both equivalence constants depend only on~$p,s$, and~$n$.
\end{lem}
We obtain the following lemma by Lemmas \ref{pls1}, \ref{bregul}, \ref{pgls}, \ref{pls2}, and \ref{GRF}. Note that in~$\Omega_{2}$, if $\varphi\nabla u=0$, then~$\varphi=0$ inasmuch as~$\psi\nabla u\neq 0$ and hence up to a Lebesgue null set,~$\nabla\varphi=0$.%
\begin{lem}[The simplified energy for~$Q_{p,s,a,V}:p<s$]\label{simpsvn}
	Suppose that $1<p<s<\infty$ and let~$H(x,\xi) = |\xi|_{s,a}$.
	Fix a positive
	solution $u\in W^{1,p}_{\loc}(\Omega)\cap C(\Omega)$ of $Q'_{p,s,a,V}[v]=0$ in $\Gw$. Then for all nonnegative~$\psi,\varphi\in W^{1,p}(\Omega)\cap C_{c}(\Omega)$,
	\bequationn
	Q_{p,s,a,V}[u\psi]\geq c(p,s)\int_{\Omega_{1}}R_{1}(\psi\nabla u,u\nabla\psi;m)\dx+c(p,s)\int_{\Omega_{2}}R_{2}(\psi\nabla u,u\nabla\psi)\dx,
	\eequationn
	and
	\bequationn
	Q_{p,s,a,V}[u\varphi]\leq C(p,s)\int_{\Omega_{1}}R_{1}(\varphi\nabla u,u\nabla\varphi;m)\dx+C(p,s)\int_{\Omega_{2}'}R_{2}(\varphi\nabla u,u\nabla\varphi)\dx,
	\eequationn
	where $\Omega_{1}\triangleq \{x\in \Omega~|~|\psi\nabla u|_{s,a}\leq|u\nabla\psi|_{s,a}\}$,~$\Omega_{2}\triangleq \{x\in \Omega~|~|\psi\nabla u|_{s,a}>|u\nabla\psi|_{s,a}\}$, and~$\Omega_{2}'\triangleq\Omega_{2}\setminus\{x\in\Omega~|~\varphi\nabla u=0\}$. In particular, for all nonnegative~$\psi\in W^{1,p}(\Omega)\cap C_{c}(\Omega)$,
 \bequationn
	Q_{p,s,a,V}[u\psi]\asymp \int_{\Omega_{1}}R_{1}(\psi\nabla u,u\nabla\psi;m)\dx+\int_{\Omega''_{2}}R_{2}(\psi\nabla u,u\nabla\psi)\dx,
	\eequationn
where~$\Omega_{2}''\triangleq\Omega_{2}\setminus\{x\in\Omega~|~\psi\nabla u=0\}$ and both equivalence constants depend only on~$p$ and~$s$.	
\end{lem}
\subsection{The Maz'ya-type characterization}\label{secmazya}
In this subsection, we establish the Maz'ya-type characterization for the two classes of norms~$|\cdot|_{s,a}$ and $\sqrt[p]{|\cdot|_{s,a}^{p}+|\cdot|_{A}^{p}}$.

Due to the simplified energies for~$Q_{p,s,a,V}$,  we obtain the following characterization of Hardy-weights of~$Q_{p,s,a,V}$, extending the recent result \cite[Theorem 1.2]{Das} of Das and Pinchover. See also \cite{Maz'ya}. Recall our notations $\Vert g\Vert_{\mathcal{H}}$ and $\Vert g\Vert_{u}$ (see Definition~\ref{hardy_weight_norm}). The proof is similar to that of \cite[Theorem 1.2]{Das}.
\btheorem\label{psmazya}
Suppose that~$1<p,s<\infty$ and let~$H(x,\xi)=|\xi|_{s,a}$ for almost all~$x\in\Omega$ and all~$\xi\in\R^{n}$.
Fix a positive solution~$u\in W^{1,p}_{\loc}(\Omega)\cap C(\Omega)$ of the equation~$Q'_{p,s,a,V}[v]=0$ in $\Gw$. Then for all
$g\in L^{1}_{\loc}(\Omega)$,~$g\in\mathcal{H}_{p,s,a,V}(\Omega)$ if and only if~$\Vert g\Vert_{u}<\infty$.
In addition, there exists a positive constant
\begin{align*}
C=\begin{cases}			C(p,s,n)&\mbox{if}~1<s<p<\infty,\\			C(p,s)&\mbox{if}~1<p\leq s<\infty,\\
		\end{cases}
\end{align*}
such that for all~$g\in \mathcal{H}_{p,s,a,V}(\Omega)$, $$\Vert g\Vert_{u}\leq \Vert g\Vert_{\mathcal{H}}\leq C\Vert g\Vert_{u}.$$
In particular,~$(\mathcal{H}_{p,s,a,V}(\Omega),\Vert\cdot\Vert_{u})$ is a Banach space.
\etheorem
\bproof
The proof of the necessity and the inequality~$\Vert g\Vert_{u}\leq \Vert g\Vert_{\mathcal{H}}$ follows from Lemma \ref{bhome}. Now, assume that~$\Vert g\Vert_{u}<\infty$ and~$p\neq s$ because the proof for~$p=s$ is similar and simpler by virtue of Lemma \ref{simpv}. Then for all compact subsets~$K$ of $\Omega$,
\begin{align*}
\int_{K}|g||u|^{p}\dx\leq \Vert g\Vert_{u}\capacity_{p,s,a,V}(K,u).
\end{align*} Let~$\mu=|g||u|^{p}\dx$. For every compact subset~$F$ of~$\Omega$ and nonnegative~$\psi\in W^{1,p}(\Omega)\cap C_{c}(\Omega)$, by \cite[Theorem 1.9]{Maly97}, 
\begin{align*}
&\int_{\Omega}|g||\psi u|^{p}\dx=p\int_{0}^{\infty}\mu(\{\psi\geq t\})t^{p-1}\dt=p\sum_{j=-\infty}^{j=\infty}\int_{2^{j}}^{2^{j+1}}\mu(\{\psi\geq t\})t^{p-1}\dt\notag\\
	&\leq (2^{p}-1)\sum_{j=-\infty}^{j=\infty}2^{pj}\mu(\{\psi\geq 2^{j}\})\leq (2^{p}-1)\Vert g\Vert_{u}\sum_{j=-\infty}^{j=\infty}2^{pj}\capacity_{p,s,a,V}(\{\psi\geq 2^{j}\},u).
\end{align*}
 We denote $ \max\{2/p,2/s,1,s/p\} $ briefly by $\theta$. Recall that~$m=\min\{s,2\}$ and~$M=\max\{s,2\}$. For every~$j\in \mathbb{Z}$, we define
	\begin{equation*}
		\psi_{j}\triangleq
		\begin{cases}
			0&\mbox{if}~\psi\leq 2^{j-1},\\
			\big(\frac{\psi}{2^{j-1}}-1\big)^{\theta}&\mbox{if}~2^{j-1}\leq\psi\leq 2^{j},\\
			1&\mbox{if}~2^{j}\leq\psi,
		\end{cases}
	\end{equation*}
	and denote~$\frac{\psi}{2^{j-1}}-1$ by~$\vartheta_{j}$. 
 Then~$\psi_{j}\in W^{1,p}(\Omega)\cap C_{c}(\Omega)$. Invoking Lemma \ref{simpsv} for~$s<p$  and Lemma \ref{simpsvn} for~$p<s$, we see at once that
	\begin{equation*}\label{pscapest2}
		\capacity_{p,s,a,V}(\{\psi\geq 2^{j}\},u)\leq Q_{p,s,a,V}[\psi_{j}u] \leq
		\begin{cases}
		C\Theta_{1}+C\Theta_{2}&~\mbox{if}~s<p,\\
		C\Gamma_{1}+C\Gamma_{2}&~\mbox{if}~p<s,\\
		\end{cases} 
	\end{equation*}
	where 	$$\Theta_{1}\triangleq \int_{\Omega_{1}^{j}}R_{1}(\psi_{j}\nabla u,u\nabla\psi_{j};M)\dx\quad \mbox{and}\quad \Theta_{2}\triangleq\int_{\Omega_{2}^{j}}R_{2}(\psi_{j}\nabla u,u\nabla\psi_{j})\dx$$
	with$$\Omega_{1}^{j}\triangleq \{x\in \Omega~|~|\psi_{j}\nabla u|\leq|u\nabla\psi_{j}|\}\quad \mbox{and}\quad \Omega_{2}^{j}\triangleq \{x\in \Omega~|~|\psi_{j}\nabla u|>|u\nabla\psi_{j}|\},$$
	and
	$$\Gamma_{1}\triangleq \int_{\Omega_{1}}R_{1}(\psi_{j}\nabla u,u\nabla\psi_{j};m)\dx\quad \mbox{and}\quad \Gamma_{2}\triangleq\int_{\Omega_{2}'}R_{2}(\psi_{j}\nabla u,u\nabla\psi_{j})\dx$$
	with$$\Omega_{1}\triangleq \{x\in \Omega~|~|\psi\nabla u|\leq|u\nabla\psi|\},\quad\Omega_{2}\triangleq \{x\in \Omega~|~|\psi\nabla u|>|u\nabla\psi|\},$$ 
$$\mbox{and}\quad\Omega_{2}'\triangleq\Omega_{2}\setminus\{x\in\Omega~|~\psi_{j}\nabla u=0\}.$$
 Moreover, let~$K_{j}\triangleq \{x\in \Omega~|~2^{j-1}\leq\psi\leq 2^{j}\}$.
	\begin{itemize}
		\item We first cope with the first terms~$\Gamma_{1}$ and~$\Theta_{1}$. We compute:	
		\begin{align*}
			\Gamma_{1}&=\int_{\Omega_{1}\cap K_{j}}u^{m}|\nabla\psi_{j}|_{s,a}^{m}(u|\nabla\psi_{j}|_{s,a}+\psi_{j}|\nabla u|_{s,a})^{p-m}\dx\\
			&\leq C(p,s)\int_{\Omega_{1}\cap K_{j}}u^{m}\frac{\vartheta_{j}^{m(\theta-1)}|\nabla\psi|_{s,a}^{m}}{2^{m(j-1)}}\bigg(u\frac{\vartheta_{j}^{\theta-1}|\nabla\psi|_{s,a}}{2^{j-1}}+\vartheta_{j}^{\theta}|\nabla u|_{s,a}\bigg)^{p-m}\dx\\
				&\leq C(p,s)\int_{\Omega_{1}\cap K_{j}}u^{m}\frac{\vartheta_{j}^{\sigma}|\nabla\psi|_{s,a}^{m}}{2^{m(j-1)}}\bigg(u\frac{|\nabla\psi|_{s,a}}{2^{j-1}}+|\nabla u|_{s,a}\bigg)^{p-m}\dx\\
				&\leq C(p,s)\int_{\Omega_{1}\cap K_{j}}u^{m}\frac{|\nabla\psi|_{s,a}^{m}}{2^{m(j-1)}}\bigg(u\frac{|\nabla\psi|_{s,a}}{2^{j-1}}+|\nabla u|_{s,a}\bigg)^{p-m}\dx\\
					&\leq \frac{C(p,s)}{2^{pj}}\int_{\Omega_{1}\cap K_{j}}u^{m}|\nabla\psi|_{s,a}^{m}\left(u|\nabla\psi|_{s,a}+\psi|\nabla u|_{s,a}\right)^{p-m}\dx\\
					&=\frac{C(p,s)}{2^{pj}}\int_{\Omega_{1}\cap K_{j}}R_{1}(\psi\nabla u,u\nabla\psi;m)\dx,	
	\end{align*}
where~$\sigma\triangleq\begin{cases}
p(\theta-1)~(\geq 0)& \mbox{if}~p\geq m,\\
p\theta-m~(\geq 0)& \mbox{if}~p<m,\\
\end{cases}$ and the last inequality is because on~$\Omega_{1}\cap K_{j}$,
\begin{align*}\bigg(u\frac{|\nabla\psi|_{s,a}}{2^{j-1}}+|\nabla u|_{s,a}\bigg)^{p-m}\leq \begin{cases}
	\bigg(u\frac{|\nabla\psi|_{s,a}}{2^{j-1}}+\frac{\psi}{2^{j-1}}|\nabla u|_{s,a}\bigg)^{p-m}&~\mbox{if}~p\geq m,\\[4mm]
	\bigg(u\frac{|\nabla\psi|_{s,a}}{2^{j}}+\frac{\psi}{2^{j}}|\nabla u|_{s,a}\bigg)^{p-m}&~\mbox{if}~p<m,
\end{cases}\end{align*}
derived from~$\psi/2^{j-1}\geq 1$ and~$\psi/2^{j}\leq 1$ respectively. Analogously, we obtain 
\begin{equation*}
	\Theta_{1}\leq \frac{C(p,s)}{2^{pj}}\int_{K_{j}}R_{1}(\psi\nabla u,u\nabla\psi;M)\dx.\\
\end{equation*}
\item Let us evaluate the second terms~$\Gamma_{2}$ and~$\Theta_{2}$. 
We calculate:
\begin{align*}
\hspace{-5mm}\Gamma_{2}&=\int_{\Omega'_{2}\cap K_{j}}\psi_{j}^{p-s}|\nabla u|_{s,a}^{p-s}\sum_{i=1}^{n}u^{2}|(\partial_{i}\psi_{j})'|^{2}(u|(\partial_{i}\psi_{j})'|+\psi_{j}|(\partial_{i} u)'|)^{s-2}\dx\\
		&\leq C\int_{\Omega'_{2}\cap K_{j}}\psi_{j}^{p-s}|\nabla u|_{s,a}^{p-s}\sum_{i=1}^{n}u^{2}\frac{\vartheta_{j}^{2(\theta-1)}|(\partial_{i}\psi)'|^{2}}{2^{2(j-1)}}\bigg(u\frac{\vartheta_{j}^{\theta-1}|(\partial_{i}\psi)'|}{2^{j-1}}+\vartheta_{j}^{\theta}|(\partial_{i} u)'|\bigg)^{s-2}\dx\\
		&\leq C\int_{\Omega'_{2}\cap K_{j}}\vartheta_{j}^{\tau}\psi_{j}^{p-s}|\nabla u|_{s,a}^{p-s}\sum_{i=1}^{n}u^{2}\frac{|(\partial_{i}\psi)'|^{2}}{2^{2(j-1)}}\bigg(u\frac{|(\partial_{i}\psi)'|}{2^{j-1}}+|(\partial_{i} u)'|\bigg)^{s-2}\dx\\
		&=\frac{C}{2^{sj}}\int_{\Omega'_{2}\cap K_{j}}\vartheta_{j}^{\tau}\psi_{j}^{p-s}|\nabla u|_{s,a}^{p-s}\sum_{i=1}^{n}u^{2}|(\partial_{i}\psi)'|^{2}\left(u|(\partial_{i}\psi)'|+\psi|(\partial_{i} u)'|\right)^{s-2}\dx\notag\\
		&=\frac{C}{2^{sj}}\int_{\Omega'_{2}\cap K_{j}}\vartheta_{j}^{\theta(p-s)+\tau}\left( \frac{\psi}{2^{j}}\right)^{p-s}|\nabla u|_{s,a}^{p-s}\sum_{i=1}^{n}u^{2}|(\partial_{i}\psi)'|^{2}\left(u|(\partial_{i}\psi)'|+\psi|(\partial_{i} u)'|\right)^{s-2}\dx\notag\\
		&\leq\frac{C}{2^{pj}}\int_{\Omega'_{2}\cap K_{j}}R_{2}(\psi\nabla u,u\nabla\psi)\dx,
\end{align*}
where all these constants~$C=C(p,s)$,$$\tau=\begin{cases}
s(\theta-1)~(\geq 0)& \mbox{if}~s\geq 2,\\
s\theta-2~(\geq 0)& \mbox{if}~s<2,\\
\end{cases}\quad\mbox{and}\quad\theta(p-s)+\tau=\begin{cases}
p\theta-s~(\geq 0)& \mbox{if}~s\geq 2,\\
p\theta-2~(\geq 0)& \mbox{if}~s<2.\\
\end{cases}$$ 
The strategy for~$\Theta_{2}$ is slightly different, that is,
\begin{align*}
 \Theta_{2}&=\int_{\Omega_{2}^{j}\cap K_{j}}\psi_{j}^{p-s}|\nabla u|_{s,a}^{p-s}\sum_{i=1}^{n}u^{2}|(\partial_{i}\psi_{j})'|^{2}(u|(\partial_{i}\psi_{j})'|+\psi_{j}|(\partial_{i} u)'|)^{s-2}\dx\\
	&=\int_{\Omega_{2}^{j}\cap K_{j}}\vartheta_{j}^{\theta(p-s)}|\nabla u|_{s,a}^{p-s}\sum_{i=1}^{n}u^{2}|(\partial_{i}\psi_{j})'|^{2}(u|(\partial_{i}\psi_{j})'|+\psi_{j}|(\partial_{i} u)'|)^{s-2}\dx\\
		&\leq\int_{\Omega_{2}^{j}\cap K_{j}}\vartheta_{j}^{(p-s)}|\nabla u|_{s,a}^{p-s}\sum_{i=1}^{n}u^{2}|(\partial_{i}\psi_{j})'|^{2}(u|(\partial_{i}\psi_{j})'|+\psi_{j}|(\partial_{i} u)'|)^{s-2}\dx\\
			&\leq\int_{\Omega_{2}^{j}\cap K_{j}}\left( \frac{\psi}{2^{j-1}}\right) ^{(p-s)}|\nabla u|_{s,a}^{p-s}\sum_{i=1}^{n}u^{2}|(\partial_{i}\psi_{j})'|^{2}(u|(\partial_{i}\psi_{j})'|+\psi_{j}|(\partial_{i} u)'|)^{s-2}\dx.	
\end{align*}
When it comes to~$\sum_{i=1}^{n}u^{2}|(\partial_{i}\psi_{j})'|^{2}(u|(\partial_{i}\psi_{j})'|+\psi_{j}|(\partial_{i} u)'|)^{s-2}$, we use the same argument as the counterpart in the estimate for~$\Gamma_{2}$ to conclude that
\begin{equation*}
\Theta_{2}\leq\frac{C(p,s)}{2^{pj}}\int_{K_{j}}R_{2}(\psi\nabla u,u\nabla\psi)\dx.	
\end{equation*}
\end{itemize}
We proceed by proving that for~$p<s$,
\begin{align*}
 &\int_{\Omega}|g||\psi u|^{p}\dx\\
  &\leq(2^{p}-1)\Vert g\Vert_{u}\sum_{j=-\infty}^{j=\infty}2^{pj}\capacity_{p,s,a,V}(\{\psi\geq 2^{j}\},u)\notag\\
		&\leq C(p,s)\Vert g\Vert_{u}\sum_{j=-\infty}^{j=\infty}2^{pj}\frac{1}{2^{pj}}\bigg(\int_{\Omega_{1}\cap K_{j}}R_{1}(\psi\nabla u,u\nabla\psi;m)\dx+\int_{\Omega'_{2}\cap K_{j}}R_{2}(\psi\nabla u,u\nabla\psi)\dx\bigg)\notag\\
		&\leq C(p,s)\Vert g\Vert_{u}\bigg(\sum_{j=-\infty}^{j=\infty}\int_{\Omega_{1}\cap K_{j}}R_{1}(\psi\nabla u,u\nabla\psi;m)\dx+\sum_{j=-\infty}^{j=\infty}\int_{\Omega'_{2}\cap A_{j}}R_{2}(\psi\nabla u,u\nabla\psi)\dx\bigg)\notag\\
		&\leq C(p,s)\Vert g\Vert_{u}\bigg(\int_{\Omega_{1}}R_{1}(\psi\nabla u,u\nabla\psi;m)\dx+\int_{\Omega_{2}}R_{2}(\psi\nabla u,u\nabla\psi)\dx\bigg)\notag\\
		&\leq C(p,s)\Vert g\Vert_{u}Q_{p,s,a,V}[u\psi],
\end{align*}
where the last step is due to Lemma \ref{simpsvn}.
For~$s<p$, we calculate:
\begin{align*}
	&\int_{\Omega}|g||\psi u|^{p}\dx\\
 &\leq(2^{p}-1)\Vert g\Vert_{u}\sum_{j=-\infty}^{j=\infty}2^{pj}\capacity_{p,s,a,V}(\{\psi\geq 2^{j}\},u)\notag\\
	&\leq C(p,s,n)\Vert g\Vert_{u}\sum_{j=-\infty}^{j=\infty}2^{pj}\frac{1}{2^{pj}}\bigg(\int_{K_{j}}R_{1}(\psi\nabla u,u\nabla\psi;M)\dx+\int_{K_{j}}R_{2}(\psi\nabla u,u\nabla\psi)\dx\bigg)\notag\\
	&\leq C(p,s,n)\Vert g\Vert_{u}\bigg(\sum_{j=-\infty}^{j=\infty}\int_{K_{j}}R_{1}(\psi\nabla u,u\nabla\psi;M)\dx+\sum_{j=-\infty}^{j=\infty}\int_{K_{j}}R_{2}(\psi\nabla u,u\nabla\psi)\dx\bigg)\notag\\
		&\leq C(p,s,n)\Vert g\Vert_{u}\bigg(\int_{\Omega}R_{1}(\psi\nabla u,u\nabla\psi;M)\dx+\int_{\Omega}R_{2}(\psi\nabla u,u\nabla\psi)\dx\bigg)\notag\\
	&\leq C(p,s,n)\Vert g\Vert_{u}Q_{p,s,a,V}[u\psi].
\end{align*}
Then for all~$\phi\in W^{1,p}(\Omega)\cap C_{c}(\Omega)$, 
$$\int_{\Omega}|g||\phi|^{p}\dx\leq C\Vert g\Vert_{u}Q_{p,s,a,V}[\phi],$$ 
where~$C=\begin{cases}			C(p,s,n)&\mbox{if}~1<s<p<\infty,\\			C(p,s)&\mbox{if}~1<p\leq s<\infty.\\		\end{cases}$ 
Then for the same constant~$C$,
\begin{equation*}
 \Vert g\Vert_{\mathcal{H}}\leq C\Vert g\Vert_{u}.
\end{equation*}
Since~$\Vert\cdot\Vert_{u}$ is equivalent to~$\Vert \cdot\Vert_{\mathcal{H}}$, the completeness of~$(\mathcal{H}_{p,s,a,V}(\Omega),\Vert\cdot\Vert_{u})$ follows from Corollary \ref{pshsp}.
\eproof
In the same manner, we may easily obtain the following characterization by virtue of Lemmas \ref{GRF} and \ref{ABE}. Note that for all~$1<p,s<\infty$ and almost all~$x\in\Omega$,~$D_{|\cdot|^{p}_{s,a}+|\cdot|_{A}^{p}}=D_{|\cdot|^{p}_{s,a}}+D_{|\cdot|_{A}^{p}}$.
\btheorem\label{newmazya}
Suppose that $1<p,s<\infty$ and let~$H(x,\xi)= \sqrt[p]{|\xi|^{p}_{s,a}+|\xi|_{A}^{p}}$ for almost all~$x\in\Omega$ and all~$\xi\in\R^{n}$, where the measurable matrix function~$A$ satisfies:
\begin{enumerate}
\item For almost all~$x\in\Omega$,~$A(x)$ is symmetric;
\item for every domain~$\omega\Subset\Omega$, there exist positive constants~$\theta_{1,\omega}$ and~$\theta_{2,\omega}$ such that for almost all~$x\in\omega$ and all~$\xi\in\R^{n}$, $$\theta_{1,\omega}|\xi|\leq|\xi|_{A}=\sqrt{A(x)\xi\cdot\xi}\leq \theta_{2,\omega}|\xi|.$$
\end{enumerate} 
Fix a positive solution~$u\in W^{1,p}_{\loc}(\Omega)\cap C(\Omega)$ of the equation~$Q'[v]=0$ in $\Gw$. Then for all
$g\in L^{1}_{\loc}(\Omega)$,~$g\in\mathcal{H}(\Omega)$ if and only if~$\Vert g\Vert_{u}<\infty$. Moreover,
there exists a positive constant
\begin{align*}
C=\begin{cases}			C(p,s,n)&\mbox{if}~p>s,\\			C(p,s)&\mbox{if}~p\leq s,\\
		\end{cases}
\end{align*}
such that for all~$g\in \mathcal{H}(\Omega)$, $$\Vert g\Vert_{u}\leq \Vert g\Vert_{\mathcal{H}}\leq C\Vert g\Vert_{u}.$$ Furthermore,~$(\mathcal{H}(\Omega),\Vert\cdot\Vert_{u})$ is a Banach space.
\etheorem
\section{Three Attainments of the Hardy constant}\label{attainmts}
In this section, our principal purpose is to demonstrate three results concerning the attainment of the~$Q$-Hardy constant. The main ingredients are, respectively, closed images of bounded sequences (up to a subsequence) under an operator~$T_{g}$, a positive spectral gap, and a concentration compactness method. To this end, we first define a~$Q$-Sobolev space.
\subsection{The $Q$-Sobolev space}
In this subsection, we define the~$Q$-Sobolev space~$\widetilde{W}^{1,p}(\Omega)$ and an important subspace~$\widetilde{W}^{1,p}_{0}(\Omega)$.  We also show some basic properties of the two spaces. An important assumption, i.e., the \emph{common uniform convexity} of~$H$, is invoked. In particular, we show that the subspace is reflexive, which is pivotal in the whole section.

\textbf{Throughout Section \ref{attainmts}, unless otherwise stated, we suppose that Assumptions \ref{ass2} and \ref{ngradb} also hold and that the functional~$Q$ is subcritical in~$\Omega$.}

We first recall that every subcritical functional has a positive continuous Hardy-weight.
\blemma[{\cite[Theorem 6.12 (3)]{HPR}}]\label{pcfw}
There exists a positive~$\mathcal{V}\in C(\Omega)\cap\mathcal{H}(\Omega)$. In particular,~$\mathcal{V}$ is locally bounded away from zero.
\elemma
The following definition is a generalization of \cite[Definition 1.8]{Kinnunen}.
\bdefinition\label{PAVHSS}
\emph{
Take a positive~$\mathcal{V}\in C(\Omega)\cap\mathcal{H}(\Omega)$. The \emph{$Q$-Sobolev space} is defined by
\begin{align*}
\widetilde{W}^{1,p}(\Omega) 
\!\triangleq\!\left\{\!u\!\in \!L^{p}(\Omega,(|V|+\mathcal{V})\!\dx)\bigg|\mbox{ the weak gradient }\nabla u\mbox{ exists and }~\Vert u\Vert_{\widetilde{W}^{1,p}(\Omega)}\!<\!\infty\!\right\}\!,
\end{align*}
where
$$\Vert u\Vert_{\widetilde{W}^{1,p}(\Omega)}\triangleq Q_{p,\mathcal{A},|V|+\mathcal{V}}[u]^{1/p}\triangleq\left(\Vert|\nabla u|_{\mathcal{A}}\Vert^{p}_{L^{p}(\Omega)}+\Vert u\Vert^{p}_{L^{p}\left(\Omega,\left(|V|+\mathcal{V}\right)\dx\right)}\right)^{1/p}.$$}
\edefinition
Recall that for almost all~$x\in\Omega$,~$\mathbb{R}^{n}_{x}=(\mathbb{R}^{n},H(x,\cdot))$. See Definition \ref{modcon} for the modulus of convexity~$\delta_{\mathbb{R}^{n}_{x}}$.
\begin{assumption}\label{cuc}
\emph{{\bf (Common uniform convexity)} There exists a Lebesgue null set~$N\subseteq\Omega$ such that for all~$\varepsilon\in(0,2]$, $$\inf_{x\in\Omega\setminus N}\delta_{\mathbb{R}^{n}_{x}}(\varepsilon)>0.$$}
\end{assumption}
\bexample
\begin{enumerate}
\item[\emph{(1)}] \emph{If~$H$ does not depend on~$x\in\Omega$, then Assumption \ref{cuc} is satisfied.}
\item[\emph{(2)}] \emph{If~$H(x,\xi)=|\xi|_{A}$, where the measurable matrix function~$A$ satisfies the conditions (1) and (2) in Theorem \ref{newmazya}, then Assumption \ref{cuc} is satisfied because for almost all~$x\in\Omega$,~$\delta_{\mathbb{R}^{n}_{x}}(\varepsilon)=1-\sqrt{1-\varepsilon^{2}/{4}}$ (see, e.g., \cite[(1.1)]{Roach}).}
\end{enumerate}
\eexample
\textbf{From now on, throughout Section \ref{attainmts}, unless otherwise stated, we suppose that Assumption \ref{cuc} also holds.
} 
\blemma\label{pavs}
The space~$\widetilde{W}^{1,p}(\Omega)$ is a reflexive Banach space. Moreover,~$\widetilde{W}^{1,p}(\Omega)\subseteq W^{1,p}_{\loc}(\Omega)$ and every bounded sequence in~$\widetilde{W}^{1,p}(\Omega)$ is bounded in~$W^{1,p}(\omega)$ for all domains~$\omega\Subset\Omega$.
\elemma
\bproof
The proof is analogous to those of \cite[Theorem 1.15]{Kinnunen}, \cite[Remarks 2.20 (3)]{Kinnunen}, and \cite[Remarks 4.1 (iii)]{Das}. 
Since~$\mathcal{V}$ is positive and locally bounded away from zero, we get the positive definiteness and also by Assumption \ref{ass9}, we conclude that~$\widetilde{W}^{1,p}(\Omega)\subseteq W^{1,p}_{\loc}(\Omega)$ and the last claim of the lemma. 
By Assumption \ref{cuc} and \cite[Theorem 2.2]{Dintegral},
the direct integral~$\big(\int_{\Omega}^{\oplus}\R^{n}_{x}\dx\big)_{L^{p}(\Omega)}$ (see \cite{Dintegral}) is a uniformly convex Banach space. Furthermore, we also obtain the completeness of~$\widetilde{W}^{1,p}(\Omega)$ (see the proof of \cite[Theorem 1.15]{Kinnunen}). Then~$\widetilde{W}^{1,p}(\Omega)$ is a Banach space. 
Recall that~$L^{p}(\Omega,(|V|+\mathcal{V})\dx)$ is uniformly convex (see \cite[Section 4.3]{Brezis}). Thus the product space $$\left(\left(\int_{\Omega}^{\oplus}\R^{n}_{x}\dx\right)_{L^{p}(\Omega)}\times L^{p}(\Omega,(|V|+\mathcal{V})\dx),\left(\Vert|\cdot|_{\mathcal{A}}\Vert^{p}_{L^{p}(\Omega)}+\Vert \cdot\Vert^{p}_{L^{p}(\Omega,(|V|+\mathcal{V})\dx)}\right)^{1/p}\right)$$ is uniformly convex by \cite[Theorem 3]{Day43}. It follows that~$\widetilde{W}^{1,p}(\Omega)$ is uniformly convex by the norm-preserving embedding and then reflexive by the Milman-Pettis theorem \cite[Theorem 3.31]{Brezis}. 
\eproof
\subsubsection{The space~$\widetilde{W}^{1,p}_{0}(\Omega)$}
Now we consider the subspace~$\widetilde{W}^{1,p}_{0}(\Omega)$ of~$\widetilde{W}^{1,p}(\Omega)$ which plays an important role in the present section. In \cite{Das}, the authors define a generalized Beppo Levi space as the completion of~$W^{1,p}(\Omega)\cap C_{c}(\Omega)$ with respect to another norm. See Remark \ref{rem1} for a discussion.
\bdefinition\label{def1}
\emph{
Let 
$\widetilde{W}^{1,p}_{0}(\Omega)$ (also denoted by~$\widetilde{W}^{1,p}_{0,\mathcal{V}}(\Omega)$) be the closure of~$W^{1,p}(\Omega)\cap C_{c}(\Omega)$ in~$\widetilde{W}^{1,p}(\Omega)$. }
\bremark
\emph{By an approximation argument (see the proof of \cite[Assertion 7.3]{Hou}), it can be easily proved that~$\widetilde{W}^{1,p}_{0}(\Omega)$ is the same as the closure of~$\core$ in~$\widetilde{W}^{1,p}(\Omega)$.}
\eremark
\edefinition
\blemma\label{pav0}
The space~$\widetilde{W}^{1,p}_{0}(\Omega)$ is uniformly convex and reflexive. Moreover, for every bounded sequence~$\{\phi_{k}\}_{k\in\mathbb{N}}$ in~$\widetilde{W}^{1,p}_{0}(\Omega)$, there exists~$\phi\in \widetilde{W}^{1,p}_{0}(\Omega)$ such that up to a subsequence,~$\{\phi_{k}\}_{k\in\mathbb{N}}$ converges to~$\phi$ weakly in~$\widetilde{W}^{1,p}_{0}(\Omega)$, a.e. in~$\Omega$, and  strongly in~$L^{p}(\omega)$ for all smooth domains~$\omega\Subset\Omega$. 
\elemma
\bproof Since~$\widetilde{W}^{1,p}_{0}(\Omega)$ is a closed subspace of~$\widetilde{W}^{1,p}(\Omega)$, by Lemma \ref{pavs},~$\widetilde{W}^{1,p}_{0}(\Omega)$ is uniformly convex and reflexive. The rest of the proof is motivated by \cite[Remarks 4.1 (iii)]{Das}. For every bounded sequence~$\{\phi_{k}\}_{k\in\mathbb{N}}$ in~$\widetilde{W}^{1,p}_{0}(\Omega)$, there exists~$\phi\in \widetilde{W}^{1,p}_{0}(\Omega)$  such that up to a subsequence,~$\{\phi_{k}\}_{k\in\mathbb{N}}$ converges weakly to~$\phi$ in~$\widetilde{W}^{1,p}_{0}(\Omega)$. Take an arbitrary smooth domain~$\omega\Subset\Omega$. Then~$\{\phi_{k}\}_{k\in\mathbb{N}}$ is bounded in~$W^{1,p}(\omega)$ by Lemma \ref{pavs}. Therefore,~$\{\phi_{k}\}_{k\in\mathbb{N}}$ converges weakly to some~$\phi'$ in~$W^{1,p}(\omega)$. By the Rellich-Kondrachov theorem, it follows that up to a subsequence,~$\{\phi_{k}\}_{k\in\mathbb{N}}$ converges strongly to~$\phi'$ in~$L^{p}(\omega)$ and a.e. in~$\omega$. 
For every~$v\in L^{p'}(\omega)$,~$L(u)\triangleq\int_{\omega}uv\dx$ is a bounded linear operator on~$\widetilde{W}^{1,p}_{0}(\Omega)$. Hence, $\lim_{k\rightarrow\infty}\int_{\omega}\phi_{k}v\dx=\int_{\omega}\phi v\dx$, which implies that~$\{\phi_{k}\}_{k\in\mathbb{N}}$ converges weakly to~$\phi$ in~$L^{p}(\omega)$. We thus get~$\phi'=\phi$. By virtue of a smooth exhaustion of~$\Omega$ and a diagonalization argument, we may ensure a common subsequence for all domains~$\omega\Subset\Omega$.
\eproof
The following lemma, motivated by \cite[Remark 4.2]{Das}, 
is mainly based on Lemma \ref{pav0} and Fatou's lemma. The proof is easy and similar and hence omitted.
\blemma\label{varhc}
For every~$g\in \mathcal{H}(\Omega)$, 
$Q_{p,\mathcal{A},V-S_{g}|g|}[\phi]\geq 0$ for all~$\phi\in \widetilde{W}^{1,p}_{0}(\Omega)$ and $$S_{g}=\inf\bigg\{Q[\phi]~\bigg|~\phi\in \widetilde{W}^{1,p}_{0}(\Omega) \mbox{~and} \int_{\Omega}|g||\phi|^{p}\dx=1\bigg\}.$$
\elemma
By virtue of Lemma \ref{varhc}, we may easily prove the following direct corollary, which shows that the space~$\widetilde{W}^{1,p}_{0,\mathcal{V}}(\Omega)$ does not depend on~$\mathcal{V}$.
\bcorollary\label{vdep}
\emph{For all positive~$\mathcal{V},\mathcal{V}'\in C(\Omega)\cap\mathcal{H}(\Omega)$, it holds that~$\widetilde{W}^{1,p}_{0,\mathcal{V}}(\Omega)=\widetilde{W}^{1,p}_{0,\mathcal{V}'}(\Omega)$.}
\ecorollary
\bremark\label{rem1}
\emph{
	Through Lemma \ref{varhc}, it can be checked that the space~$\widetilde{W}^{1,p}_{0}(\Omega)$ is actually the completion of~$W^{1,p}(\Omega)\cap C_{c}(\Omega)$ with respect to the norm~$Q_{p,\mathcal{A},V^{+}}^{1/p}$. Therefore, when~$H(x,\xi)=|\xi|_{A}$, with a measurable matrix function~$A$ satisfying the conditions (1) and (2) in Theorem \ref{newmazya}, the space~$\widetilde{W}^{1,p}_{0}(\Omega)$ is equal to the \emph{generalized Beppo Levi space~$\mathcal{D}^{1,p}_{A,V^{+}}(\Omega)$} in \cite{Das} up to an isometry.}
\eremark
\bdefinition
\emph{For every~$g\in\mathcal{H}(\Omega)\setminus\{0\}$, the Hardy constant~$S_{g}$ is \emph{attained at some~$\phi\in \widetilde{W}^{1,p}_{0}(\Omega)\setminus\{0\}$} if~$\int_{\Omega}|g||\phi|^{p}\dx>0$
 and $$S_{g}=\frac{Q[\phi]}{\int_{\Omega}|g||\phi|^{p}\dx}\, .$$}
\edefinition
The following lemma can be easily proved by virtue of the standard variational technique (see the proof of \cite[Lemma 5.13]{HKM}) and the method in \cite[Theorem 4.19]{HPR}.
\blemma\label{lemuni}
Suppose that~$g\in \mathcal{H}(\Omega)\cap M^{q}_{\loc}(p;\Omega)\setminus\{0\}$ and that~$S_{g}$ is attained at some~$\phi\in \widetilde{W}^{1,p}_{0}(\Omega)\setminus\{0\}$. 
Then~$\phi$ is a solution of~$Q'_{p,\mathcal{A},V-S_{g}|g|}[u]=0$ and either~$\phi>0$ in~$\Omega$ or~$\phi<0$ in~$\Omega$. For every two such functions~$\phi$ and~$\phi'$, there exists~$C\in \R\setminus\{0\}$ such that~$\phi=C\phi'$.
\elemma
\subsection{Closed images of bounded sequences}
In this subsection, by virtue of closed images of bounded sequences (up to a subsequence) in~$\widetilde{W}^{1,p}_{0}(\Omega)$ under an operator~$T_{g}$, we give the first attainment of the Hardy constant. We also provide some examples satisfying Condition (H0) (see Definition \ref{condih}).

We first define the operator~$T_{g}$, Condition (H0), and a space~$M^{q}_{c}(p;\Omega)$.
\bdefinition
\emph{
For every~$g\in \mathcal{H}(\Omega)$, 
we define~$T_{g}(\phi)\triangleq\int_{\Omega}|g||\phi|^{p}\dx$ for all~$\phi\in \widetilde{W}^{1,p}_{0}(\Omega)$.}
\edefinition
\bdefinition\label{condih}
\emph{The operator~$T_{g}$
satisfies \emph{Condition (H0)} if for every bounded sequence~$\{\phi_{k}\}_{k\in\mathbb{N}}$ in~$\widetilde{W}^{1,p}_{0}(\Omega)$, there exists~$\phi\in \widetilde{W}^{1,p}_{0}(\Omega)$ such that up to a subsequence, 
$\{\phi_{k}\}_{k\in\mathbb{N}}$ converges to~$\phi$ weakly in~$\widetilde{W}^{1,p}_{0}(\Omega)$ and~$\lim_{k\rightarrow\infty}T_{g}(\phi_{k})=T_{g}(\phi)$.}
\edefinition
\bdefinition
\emph{A function~$f\in M^{q}_{c}(p;\Omega)$ if for some domain~$\omega\Subset\Omega$,~$f\in M^{q}(p;\omega)$ and~$f|_{\Omega\setminus\omega}=0$.}
\edefinition
\bremark
\emph{We have $L^{\infty}_{c}(\Omega)\subseteq\mathcal{H}(\Omega)$. If the equation~$Q'[v]=0$ has a positive supersolution whose gradient is locally bounded in~$\Omega$ (cf. Theorem \ref{gethm} and Example \ref{gbexa}), then by \cite[Theorem 6.23 (2)]{HPR}, we may find a positive function~$W^{*}\in C(\Omega)$ such that for all~$\phi\in\core$,~$$Q[\phi]\geq\int_{\Omega}W^{*}(|\nabla\phi|_{\mathcal{A}}^{p}+|\phi|^{p})\dx.$$ Then by the Morrey-Adams theorem \cite[Theorem 2.16 (2)]{HPR}, we deduce that~$M^{q}_{c}(p;\Omega)\subseteq\mathcal{H}(\Omega)$.}
\eremark 
\bdefinition
\emph{
Let~$\mathcal{H}^{0}(\Omega)$ be the closure of~$L^{\infty}_{c}(\Omega)$ in~$(\mathcal{H}(\Omega),\Vert\cdot\Vert_{\mathcal{H}})$. In addition, if the equation~$Q'[v]=0$ has a positive supersolution whose gradient is locally bounded in~$\Omega$, let~$\widetilde{\mathcal{H}}^{0}(\Omega)$ be the closure of~$M^{q}_{c}(p;\Omega)$ in~$(\mathcal{H}(\Omega),\Vert\cdot\Vert_{\mathcal{H}})$.}
\edefinition
\bremark
\emph{Suppose that the equation~$Q'[v]=0$ has a positive supersolution whose gradient is locally bounded in~$\Omega$. If~$p>n$, then~$\mathcal{H}^{0}(\Omega)=\widetilde{\mathcal{H}}^{0}(\Omega)$. To this end, it suffices to show that~$M^{q}_{c}(p;\Omega)\subseteq\mathcal{H}^{0}(\Omega)$. Note that in this case,~$M^{q}_{c}(p;\Omega)=L^{1}_{c}(\Omega)$. For any~$g\in L^{1}_{c}(\Omega)$, suppose that~$g\in L^{1}(\omega)$ for some Lipschitz domain~$\omega\Subset\Omega$ and~$\supp g\subseteq\omega$. Then by \cite[Corollary 4.23]{Brezis}, there exists a sequence~$\{g_{k}\}_{k\in\mathbb{N}}\subseteq C^{\infty}_{c}(\omega)\subseteq \mathcal{H}^{0}(\Omega)$ converging in~$L^{1}(\omega)$ to~$g$. It remains to prove that~$\lim_{k\rightarrow\infty}\Vert g_{k}-g\Vert_{\mathcal{H}}=0$. Take a domain~$\omega'\Subset\Omega$ such that~$\omega\Subset\omega'$. Checking the proof of \cite[Theorem 2.4 (ii)]{PPAPDE}, we see that for all~$\varepsilon>0$,~$k\in\mathbb{N}$, and~$\phi\in W^{1,p}(\omega)$,
\begin{align*}
\int_{\omega}|g_{k}-g||\phi|^{p}\dx&\leq C(n,p,\omega,\omega')\varepsilon\Vert\phi\Vert_{W^{1,p}(\omega)}^{p}+\frac{C(n,p,\omega,\omega')}{\varepsilon^{n/(p-n)}}\Vert g_{k}-g\Vert_{L^{1}(\omega)}^{p/(p-n)}\Vert \phi\Vert_{L^{p}(\omega)}^{p}.
\end{align*}
For each~$\varepsilon>0$, there exists~$k_{0}\in\mathbb{N}$ such that for all~$k\geq k_{0}$,~$\Vert g_{k}-g\Vert_{L^{1}(\omega)}^{p/(p-n)}<\varepsilon^{n/(p-n)+1}$. 
Hence for all such $k$, we have 
$$\int_{\omega}|g_{k}-g||\phi|^{p}\dx\leq C(n,p,\omega,\omega')\varepsilon\Vert\phi\Vert_{W^{1,p}(\omega)}^{p}.$$
Still by \cite[Theorem 6.23 (2)]{HPR}, there exists a positive function~$W^{*}\in C(\Omega)$ such that for all~$\phi\in C^{\infty}_{c}(\Omega)$ and hence all~$\phi\in W^{1,p}(\Omega)\cap C_{c}(\Omega)$ by approximation,~$$Q[\phi]\geq\int_{\Omega}W^{*}(|\nabla\phi|_{\mathcal{A}}^{p}+|\phi|^{p})\dx.$$ Then for all~$\phi\in W^{1,p}(\Omega)\cap C_{c}(\Omega)$,~$$Q[\phi]\geq\int_{\omega}W^{*}(|\nabla\phi|_{\mathcal{A}}^{p}+|\phi|^{p})\dx\geq C_{0}\int_{\omega}(|\nabla\phi|^{p}+|\phi|^{p})\dx,$$
where the constant~$C_{0}$ does not depend on~$\phi$. 
Therefore, for each~$\varepsilon>0$, all~$k\geq k_{0}$, and all~$\phi\in W^{1,p}(\Omega)\cap C_{c}(\Omega)$ such that~$\int_{\Omega}|g_{k}-g||\phi|^{p}\dx>0$, we may estimate (cf. \cite[p. 1337]{PPAPDE}):
\begin{align*}
\frac{\int_{\Omega}|g_{k}-g||\phi|^{p}\dx}{Q[\phi]}=\frac{\int_{\omega}|g_{k}-g||\phi|^{p}\dx}{Q[\phi]}
\leq \frac{C(n,p,\omega,\omega')\varepsilon\Vert\phi\Vert_{W^{1,p}(\omega)}^{p}}{Q[\phi]}\leq \frac{C_{1}\varepsilon Q[\phi]}{Q[\phi]}=C_{1}\varepsilon,
\end{align*}
where the constant~$C_{1}$ does not depend on~$k$ or~$\phi$.
Recall that for all~$k\in\mathbb{N}$,
\begin{align*}
\Vert g_{k}-g\Vert_{\mathcal{H}}&=\sup\bigg\{\frac{\int_{\Omega}|g_{k}-g||\phi|^{p}\dx}{Q[\phi]}~\bigg|~\phi\in W^{1,p}(\Omega)\cap C_{c}(\Omega) \mbox{~and} \int_{\Omega}|g_{k}-g||\phi|^{p}\dx>0\bigg\}.
\end{align*}
Therefore,~$\lim_{k\rightarrow\infty}\Vert g_{k}-g\Vert_{\mathcal{H}}=0$.
}
\eremark
Now we come to the central lemma of this subsection (see \cite[Theorem 4.3]{Das}). 
\btheorem\label{conh}
\begin{enumerate}
    \item 
For every $g\in \mathcal{H}^{0}(\Omega)$,~$T_{g}$ satisfies Condition (H0). 
    \item Assume that the equation~$Q'[v]=0$ has a positive supersolution whose gradient is locally bounded in~$\Omega$. Then for every~$g\in\widetilde{\mathcal{H}}^{0}(\Omega)$,~$T_{g}$ satisfies Condition (H0). 
\end{enumerate}
\etheorem
\bproof
 Let $\{\phi_{k}\}_{k\in\mathbb{N}}$ be a bounded sequence in~$\widetilde{W}^{1,p}_{0}(\Omega)$. Then by Lemma \ref{pav0}, there exists~$\phi\in \widetilde{W}^{1,p}_{0}(\Omega)$ such that up to a subsequence,~$\{\phi_{k}\}_{k\in\mathbb{N}}$ converges to~$\phi$ weakly in~$\widetilde{W}^{1,p}_{0}(\Omega)$, a.e. in~$\Omega$, and strongly in~$L^{p}(\omega)$ for all smooth domains~$\omega\Subset\Omega$.

(1): The proof is similar to that of \cite[Theorem 4.3]{Das}. First consider an arbitrary~$g\in L^{\infty}_{c}(\Omega)$. Suppose that~$g$ is supported in a smooth domain~$\omega\Subset\Omega$. Then$$\lim_{k\rightarrow \infty}\int_{\omega}|\phi_{k}|^{p}\dx=\int_{\omega}|\phi|^{p}\dx.$$ Note that a.e. in~$\Omega$,~$|g||\phi_{k}|^{p}\leq C|\phi_{k}|^{p}$ for some constant~$C$ as~$g$ is bounded in~$\Omega$. Plainly,~$\lim_{k\rightarrow \infty}|g||\phi_{k}|^{p}=|g||\phi|^{p}$ a.e. in~$\Omega$. By \cite[Section 4.4, Theorem 19]{reala}, we get~$\lim_{k\rightarrow \infty}\int_{\omega}|g||\phi_{k}|^{p}=\int_{\omega}|g||\phi|^{p}$.

Now consider~$g\in \mathcal{H}^{0}(\Omega)$. By definition, there exists a sequence~$\{g_{l}\}_{l\in\mathbb{N}}\subseteq L^{\infty}_{c}(\Omega)$ converging to~$g$ in~$(\mathcal{H}(\Omega),\Vert\cdot\Vert_{\mathcal{H}})$. 
  Then for all~$l\in\mathbb{N}$, $\lim_{k\rightarrow\infty}\int_{\Omega}|g_{l}||\phi_{k}|^{p}\dx=\int_{\Omega}|g_{l}||\phi|^{p}\dx$. Since~$\{\phi_{k}\}_{k\in\mathbb{N}}$ is bounded in~$\widetilde{W}^{1,p}_{0}(\Omega)$, we can certainly assume that for some constant~$M>0$ and all~$k\in\mathbb{N}$,~$\Vert\phi_{k}\Vert_{\widetilde{W}^{1,p}(\Omega)}\leq M$ and~$\Vert\phi\Vert_{\widetilde{W}^{1,p}(\Omega)}\leq M$. For each~$\varepsilon>0$, there exists~$l_{0}\in\mathbb{N}$ such that~$\Vert g-g_{l_{0}}\Vert_{\mathcal{H}}<\varepsilon$, and then~$N\in \mathbb{N}$ such that for all~$k\geq N$,$$\left|\int_{\Omega}|g_{l_{0}}||\phi_{k}|^{p}\dx-\int_{\Omega}|g_{l_{0}}||\phi|^{p}\dx\right|<\varepsilon.$$
  We continue by calculating:
\begin{align*}
&\int_{\Omega}|g-g_{l_{0}}|\left||\phi_{k}|^{p}-|\phi|^{p}\right|\dx\leq p\int_{\Omega}|g-g_{l_{0}}||\phi_{k}-\phi|\left(|\phi_{k}|^{p-1}+|\phi|^{p-1}\right)\dx\\
&\leq C\left(\int_{\Omega}|g-g_{l_{0}}||\phi_{k}-\phi|^{p}\dx\right)^{1/p}\left(\int_{\Omega}|g-g_{l_{0}}||\phi_{k}|^{p}+|g-g_{l_{0}}||\phi|^{p}\dx\right)^{1/p'}\\
&\leq C\left(\frac{1}{S_{g-g_{l_{0}}}}Q[\phi_{k}-\phi]\right)^{1/p}\left(\frac{1}{S_{g-g_{l_{0}}}}Q[\phi_{k}]+\frac{1}{S_{g-g_{l_{0}}}}Q[\phi]\right)^{1/p'}\\
&\leq C\Vert g-g_{l_{0}}\Vert_{\mathcal{H}}Q[\phi_{k}-\phi]^{1/p}M^{p/p'}
\leq C\Vert g-g_{l_{0}}\Vert_{\mathcal{H}}\Vert\phi_{k}-\phi\Vert_{\widetilde{W}^{1,p}(\Omega)}M^{p/p'}\\
&\leq C\Vert g-g_{l_{0}}\Vert_{\mathcal{H}}M^{1+p/p'} < CM^{p}\varepsilon,
\end{align*}
where the third inequality is due to Lemma \ref{varhc}.
Furthermore, for all~$k\geq N$, we have:
\begin{align*}
\left|T_{g}(\phi_{k})-T_{g}(\phi)\right|&=\left|\int_{\Omega}|g||\phi_{k}|^{p}\dx-\int_{\Omega}|g||\phi|^{p}\dx\right|\\
&\leq\left|\int_{\Omega}(|g|-|g_{l_{0}}|+|g_{l_{0}}|)(|\phi_{k}|^{p}-|\phi|^{p})\dx\right|\\
&\leq\int_{\Omega}|g-g_{l_{0}}|||\phi_{k}|^{p}-|\phi|^{p}|\dx+\left|\int_{\Omega}|g_{l_{0}}|\left(|\phi_{k}|^{p}-|\phi|^{p}\right)\dx\right|\\
&<CM^{p}\varepsilon+\varepsilon.
\end{align*}

(2): First suppose that~$g\in M^{q}_{c}(p;\Omega)$. By definition, there exists a domain~$\omega\Subset\Omega$ such that~$g\in M^{q}(p;\omega)$ and~$g|_{\Omega\setminus\omega}=0$. Then~$g\in M^{q}(p;\omega')$ for all smooth domains~$\omega'$ such that~$\omega\Subset\omega'\Subset\Omega$. In particular, we pick two smooth domains~$\omega'$ and~$\omega''$ such that~$\omega\Subset\omega'\Subset\omega''\Subset\Omega$.  By the Morrey-Adams theorem \cite[Theorem 2.16 (2)]{HPR}, for every sufficiently small~$\varepsilon>0$, there exists~$C=C(n,p,q,\varepsilon,\omega',\omega'',\Vert g\Vert_{M^{q}(p;\omega'')})>0$ and~$N\in\mathbb{N}$ such that for all~$k\geq N$,
\begin{align*}	\int_{\omega}|g||\phi_{k}-\phi|^{p}\dx=\int_{\omega'}|g||\phi_{k}-\phi|^{p}\dx&\leq \varepsilon\int_{\omega'}|\nabla(\phi_{k}-\phi)|^{p}\dx+C\int_{\omega'}|\phi_{k}-\phi|^{p}\dx\\	&\leq\varepsilon\int_{\omega'}|\nabla(\phi_{k}-\phi)|^{p}\dx+\varepsilon.
\end{align*}
Since~$\{\phi_{k}\}_{k\in\mathbb{N}}$ is bounded in~$\widetilde{W}^{1,p}_{0}(\Omega)$, we deduce that~$\lim_{k\rightarrow\infty}\int_{\omega}|g||\phi_{k}-\phi|^{p}\dx=0.$ Then$$\lim_{k\rightarrow \infty}\int_{\omega}|g||\phi_{k}|^{p}=\int_{\omega}|g||\phi|^{p}.$$ For~$g\in\widetilde{\mathcal{H}}^{0}(\Omega)$, the proof is similar to the counterpart of (1) and hence omitted.
\eproof
The following theorem is a generalization of \cite[Theorem 6.2.2]{Anoop} and the proof is similar.
\btheorem\label{Anoopmini}
Suppose that
~$V$ is nonnegative in~$\Omega$. For every~$g\in\mathcal{H}(\Omega)\setminus\{0\}$, 
if~$T_{g}$ satisfies  Condition (H0), 
then~$S_{g}$ is attained in~$\widetilde{W}^{1,p}_{0}(\Omega)$.
\etheorem
\bproof
By considering~$|g|$, we may assume that~$g\geq 0$. Inasmuch as~$g\in\mathcal{H}(\Omega)\setminus\{0\}$, we may find a sequence~$\{\phi_{k}\}_{k\in\mathbb{N}}\subseteq W^{1,p}(\Omega)\cap C_{c}(\Omega)$ such that~$S_{g}=\lim_{k\rightarrow\infty}Q[\phi_{k}]$ and that for all~$k\in\mathbb{N}$,~$\int_{\Omega}|g||\phi_{k}|^{p}\dx=1$. Then~$\{\phi_{k}\}_{k\in\mathbb{N}}$ is bounded in~$\widetilde{W}^{1,p}_{0}(\Omega)$ as~$V\geq 0$ and~$\mathcal{V}$ is a Hardy-weight of~$Q$ in~$\Omega$. Due to Condition (H0) and Lemma \ref{pav0}, there exists~$\phi\in \widetilde{W}^{1,p}_{0}(\Omega)$ such that up to a subsequence, $\{\phi_{k}\}_{k\in\mathbb{N}}$ converges to~$\phi$ weakly  in~$\widetilde{W}^{1,p}_{0}(\Omega)$ and a.e. in~$\Omega$, and~$\lim_{k\rightarrow\infty}T_{g}(\phi_{k})=T_{g}(\phi)$. Then~$T_{g}(\phi)=1$ since~$T_{g}(\phi_{k})=1$ for all~$k\in\mathbb{N}$. A trivial verification shows that $\{\nabla\phi_{k}\}_{k\in\mathbb{N}}$ converges weakly to~$\nabla\phi$ in~$\big(\int_{\Omega}^{\oplus}\R^{n}_{x}\dx\big)_{L^{p}(\Omega)}$. Then~$\int_{\Omega}|\nabla\phi|^{p}_{\mathcal{A}}\dx\leq \liminf_{k\rightarrow\infty}\int_{\Omega}|\nabla\phi_{k}|^{p}_{\mathcal{A}}\dx$. Since~$V\geq 0$ and~$\{\phi_{k}\}_{k\in\mathbb{N}}$ converges to~$\phi$ a.e. in~$\Omega$, Fatou's lemma leads to~$\int_{\Omega}V|\phi|^{p}\dx\leq \liminf_{k\rightarrow\infty}\int_{\Omega}V|\phi_{k}|^{p}\dx$.
It follows that 
$$S_{g}\leq Q[\phi]\leq \liminf_{k\rightarrow\infty}Q[\phi_{k}]=\lim_{k\rightarrow\infty}Q[\phi_{k}]=S_{g},$$
where the first inequality is justified by Lemma \ref{varhc}. We readily get~$S_{g}=Q[\phi]$. 
\eproof
\subsection{A spectral gap phenomenon}
In this subsection, we show that for every~$g\in\mathcal{H}(\Omega)$, under some further assumptions, if~$g$ has a spectral gap, then the Hardy constant is achieved.

Note that by Lemma \ref{pcfw}, every subcritical functional~$Q$ in~$\Omega$ is also subcritical in all subdomains of~$\Omega$. 

\textbf{Throughout the present subsection, we assume that~$\Omega$ has a smooth exhaustion~$\{\omega_{i}\}_{i\in\mathbb{N}}$ such that $\omega_{i+1}\setminus\overline{\omega_{i}}$ are domains for all~$i\in\mathbb{N}$, and let us fix such an exhaustion}.
\bremark
\emph{It can be easily proved that for every smooth exhaustion~$\{\omega_{i}\}_{i\in\mathbb{N}}$ of~$\Omega$, if~$\omega_{i+1}\setminus\overline{\omega_{i}}$ is a domain for all~$i\in\mathbb{N}$, then~$\Omega\setminus\overline{\omega_{i}}$ is a domain for all~$i\in\mathbb{N}$.}
\eremark
\bdefinition
\emph{For every~$g\in\mathcal{H}(\Omega)$, 
let
$$S_{g}^{\infty}\triangleq S_{g}^{\infty}(\Omega)\triangleq\lim_{i\rightarrow\infty}\inf\left\{Q[\phi;\Omega\setminus\overline{\omega_{i}}]~\bigg|~\phi\in \widetilde{W}^{1,p}_{0}(\Omega\setminus\overline{\omega_{i}})~\mbox{and} \int_{\Omega\setminus\overline{\omega_{i}}}|g||\phi|^{p}\dx=1\right\}.$$}
\edefinition
\bremark
\emph{It is easy to see that~$S_{g}\leq S_{g}^{\infty}$.}
\eremark
\bdefinition
\emph{Let~$g\in\mathcal{H}(\Omega)$. If~$S_{g}< S_{g}^{\infty}$, we say that~$g$ has a \emph{spectral gap} for $Q$ in $\Gw$.} 
\edefinition
In the following lemma, we show that under certain extra regularity conidtions, if $g\in\mathcal{H}(\Omega)$ has a spectral gap for $Q$ in $\Gw$,  then the functional~$Q_{p,\mathcal{A},V-S_{g}|g|}$ is critical in~$\Omega$. The criticality of $Q_{p,\mathcal{A},V-S_{g}|g|}$ is used in Theorem \ref{att2}. See Theorem \ref{gethm} and Example \ref{gbexa} for sufficient conditions and specific examples on local gradient boundedness. 
\blemma\label{crucial}
Suppose that~$V\in\widetilde{M}^{q}_{\loc}(p;\Omega)$.
Consider~$g\in\mathcal{H}(\Omega)\cap \widetilde{M}^{q}_{\loc}(p;\Omega)\setminus\{0\}$ such that~$g$ has a spectral gap for $Q$ in~$\Omega$. 
\begin{enumerate}
\item For every~$S_{g}<\mu_{1}<S_{g}^{\infty}$, there exists an~$i(\mu_{1})\in\mathbb{N}$  such that~$Q_{p,\mathcal{A},V-\mu_{1}|g|}$ is nonnegative in~$\Omega\setminus \overline{\omega_{i(\mu_{1})}}$, and there exists a positive solution~$v$ of the equation~$Q'_{p,\mathcal{A},V-\mu_{1}|g|}[u]=0$ in~$\Omega\setminus \overline{\omega_{i(\mu_{1})}}$ such that~$v$ is continuous up to~$\partial\omega_{i(\mu_{1})}$ and~$v|_{\partial\omega_{i(\mu_{1})}}=0$.
\item  
Suppose further that~$\nabla v$ is locally bounded in~$\omega_{i(\mu_{1})+2}\setminus\overline{\omega_{i(\mu_{1})}}$. 
Then there exists a nonnegative~$\mathcal{W}\in M^{q}_{c}(p;\Omega)\setminus\{0\}$ 
such that~$Q_{p,\mathcal{A},V-\mu_{1}|g|+\mathcal{W}}$ is nonnegative in~$\Omega$. 
\item If, in addition, for every~$t,\tau\in\R$, the equation~$Q'_{p,\mathcal{A},V-t|g|+\tau\mathcal{W}}[u]=0$ has a positive supersolution with locally bounded gradient in~$\Omega$ and Assumption \ref{ngradb} holds for the equation~$Q'_{p,\mathcal{A},V-t|g|+\tau\mathcal{W}}[u]=0$, then the functional~$Q_{p,\mathcal{A},V-S_{g}|g|}$ is critical in~$\Omega$. 
\end{enumerate}
\elemma
\bproof 
The proof is similar to that of \cite[Lemma 2.3]{Lamberti} (see also \cite[Lemma 5.2]{Das}).

(1): The first statement is due to the definition of~$S_{g}^{\infty}$. Pick a smooth exhaustion~$\{\omega_{j}'\}_{j\in\mathbb{N}}$
of~$\Omega\setminus\overline{\omega_{i(\mu_{1})}}$. Then up to a subsequence, we have~$\omega'_{j}\Subset \omega_{i(\mu_{1})+j}\setminus\overline{\omega_{i(\mu_{1})}}$ for every~$j\in\mathbb{N}$. For every~$j>i(\mu_{1})+1$, by \cite[Theorem 4.22]{HPR}, let~$v_{j}\in W^{1,p}_{0}(\omega_{j}\setminus \overline{\omega_{i(\mu_{1})}})$ be a positive solution of the equation~$Q'_{p,\mathcal{A},V-\mu_{1}|g|+1/j}[u]=f_{j}$ in~$\omega_{j}\setminus \overline{\omega_{i(\mu_{1})}}$, where~$0\leq f_{j}\in C^{\infty}_{c}(\omega_{j}\setminus\overline{\omega_{j-1}})\setminus\{0\}$ and~$v_{j}(x_{0})=1$ for some~$x_{0}\in\omega_{1}'$. Then~$v_{j}$ is also a positive solution of~$Q'_{p,\mathcal{A},V-\mu_{1}|g|+1/j}[u]=0$ in~$\omega_{j-1}\setminus \overline{\omega_{i(\mu_{1})}}$.
 By the Harnack convergence principle \cite[Theorem 3.5]{HPR}, up to a subsequence,~$v_{j}$ converges locally uniformly to a positive solution~$v$ of the equation~$Q'_{p,\mathcal{A},V-\mu_{1}|g|}[u]=0$ in~$\Omega\setminus \overline{\omega_{i(\mu_{1})}}$.
By \cite[Theorem 2.3]{Byun} for~$p\leq n$ and \cite[Corollary 9.14]{Brezis} for~$p>n$,
$v_{j}$ is 
continuous on~$\overline{\omega_{i(\mu_{1})+1}}\setminus\omega_{i(\mu_{1})}$ for every~$j>i(\mu_{1})+1$. %
By virtue of the weak comparison principle \cite[Theorem 4.25]{HPR}, we may easily show that~$v_{j}$ is uniformly bounded on~$\overline{\omega_{i(\mu_{1})+1}}\setminus\omega_{i(\mu_{1})}$. By \cite[Theorem 3.2]{HPR}, for all sufficiently large~$j$,~$v_{j}$ is uniformly H\"older continuous on~$\partial\omega_{i(\mu_{1})+1}$ and hence on~$\partial(\omega_{i(\mu_{1})+1}\setminus \overline{\omega_{i(\mu_{1})}})$. 
\begin{itemize}
\item~$p\leq n$. By \cite[Lemma 5.2]{Lieberman93} and \cite[p.~251, Theorem 1.1]{LA68}, for all sufficiently large~$j$,~$v_{j}$ is uniformly H\"older continuous on~$\overline{\omega_{i(\mu_{1})+1}}\setminus \omega_{i(\mu_{1})}$.
\item~$p>n$. Take~$\psi\in\core$ such that~$0\leq\psi\leq 1$,~$\psi|_{\overline{\omega_{i(\mu_{1})+1}}}= 1$, and~$\supp\psi\subseteq\omega_{i(\mu_{1})+2}$. 
Then for all sufficiently large~$j$,~$\psi v_{j}\in W^{1,p}_{0}(\omega_{j}\setminus \overline{\omega_{i(\mu_{1})}})$. 
For all sufficiently large~$j$,
We may test~$\psi v_{j}$ in the definition of~$v_{j}$ as a solution of~$Q'_{p,\mathcal{A},V-\mu_{1}|g|+1/j}[u]=f_{j}$ in~$\omega_{j}\setminus \overline{\omega_{i(\mu_{1})}}$ to deduce  that~$\Vert\nabla v_{j}\Vert_{L^{p}(\omega_{i(\mu_{1})+1}\setminus\overline{\omega_{i(\mu_{1})}})}$ is bounded. Then by \cite[Corollary 9.14]{Brezis},~$v_{j}$ is uniformly H\"older continuous on~$\overline{\omega_{i(\mu_{1})+1}}\setminus \omega_{i(\mu_{1})}$.
\end{itemize}
By the Arzel\`a-Ascoli theorem, up to a subsequence,~$v_{j}$ converges uniformly on~$\overline{\omega_{i(\mu_{1})+1}}\setminus \omega_{i(\mu_{1})}$ to a continuous function~$v'$.  Evidently,~$v|_{\overline{\omega_{i(\mu_{1})+1}}\setminus \overline{\omega_{i(\mu_{1})}}}=v'$. Then~$v$ is continuous up to~$\partial\omega_{i(\mu_{1})}$ and~$v|_{\partial\omega_{i(\mu_{1})}}=0$.

(2): Let~$m\triangleq\min_{\partial\omega_{i(\mu_{1})+1}}v(x)$. Then~$m>0$. For each~$0<\varepsilon<m/8$, take a~$C^{2}$ function~$f$ from~$[0,\infty)$ to~$[0,\infty)$ such that~$f|_{[0,2\varepsilon]} = \varepsilon$,~$f(t)=t$ for all~$t\in [4\varepsilon, \infty)$, and~$f'(t)>0$ for all~$t>2\varepsilon$. Simultaneously, we may also require that~$\lim_{t\rightarrow 2\varepsilon}|f'(t)|^{p-2}f''(t)=0$, where it is understood that~$(|f'(t)|^{p-2}f''(t))|_{[0,2\varepsilon]} =  
 0$. Then~$|f'(t)|^{p-2}f''(t)$ is continuous on~$[0,\infty)$. For example, we extend~$f$ sufficiently around~$2\varepsilon$ by~$f(t)\triangleq\varepsilon+(t-2\varepsilon)^{\gamma}$ with the constant~$\gamma>\max\{p',2\}$.

 Now let~$\bar{v}(x)\triangleq f(v(x))$ for all~$x\in\overline{\omega_{i(\mu_{1})+1}}\setminus\omega_{i(\mu_{1})}$. Since~$v$ is continuous up to~$\partial\omega_{i(\mu_{1})}$ and~$v|_{\partial\omega_{i(\mu_{1})}}=0$, we may find open neighborhoods~$U$ of~$\partial\omega_{i(\mu_{1})}$ and~$U'$ of~$\partial\omega_{i(\mu_{1})+1}$ such that$$\bar{v}|_{U\setminus\omega_{i(\mu_{1})}} = \varepsilon\quad\mbox{and}\quad\bar{v}|_{U'\cap\overline{\omega_{i(\mu_{1})+1}}}=v.$$ Then we define~$\bar{v}|_{\omega_{i(\mu_{1})}} = \varepsilon$ and~$\bar{v}|_{\Omega\setminus\overline{\omega_{i(\mu_{1})+1}}} =  v$. Let
 $$\mathcal{W}\triangleq\begin{cases}|V-\mu_{1}|g||&~\mbox{in}~\Gamma_{1}\triangleq U\cup\omega_{i(\mu_{1})},\\		0&~\mbox{in}~\Gamma_{2}\triangleq(\Omega\setminus\omega_{i(\mu_{1})+1})\cup(U'\cap \omega_{i(\mu_{1})+1}),\\	
V'&~\mbox{in}~\Omega\setminus(\Gamma_{1}\cup\Gamma_{2}),
 \end{cases}
  $$
  where$$V'\triangleq\frac{|-(p-1)|f'(v)|^{p-2}f''(v)|\nabla v|_{\mathcal{A}}^{p}-f'(v)^{p-1}(V-\mu_{1}|g|)v^{p-1}+(V-\mu_{1}|g|)\bar{v}^{p-1}|}{\bar{v}^{p-1}}.$$ It is a simple matter to check that~$\mathcal{W}\in M^{q}_{c}(p;\Omega)$. For each nonnegative~$\phi\in \core$, let
$$\psi\triangleq\begin{cases} 0&~\mbox{in}~\Gamma_{1},\\		\phi &~\mbox{in}~\Gamma_{2},\\	
f'(v)^{p-1}\phi&~\mbox{in}~\Gamma_{3}.
 \end{cases}$$
Then~$\psi\in W^{1,p}_{c}(\Omega\setminus\overline{\omega_{i(\mu_{1})}})$. It follows that$$\int_{\Gamma_{2}\cup\Gamma_{3}}\mathcal{A}(x,\nabla v)\cdot\nabla\psi\dx+\int_{\Gamma_{2}\cup\Gamma_{3}}(V-\mu_{1}|g|)v^{p-1}\psi\dx=0.$$
We are going to show that for all~$\phi\in\core$, $$\int_{\Omega}\mathcal{A}(x,\nabla \bar{v})\cdot\nabla\phi\dx+\int_{\Omega}(V-\mu_{1}|g|+\mathcal{W})\bar{v}^{p-1}\phi\dx\geq 0.$$ Then~$\bar{v}$ is a positive supersolution of~$Q'_{p,\mathcal{A},V-\mu_{1}|g|+\mathcal{W}}[u]=0$ in~$\Omega$. By the AAP theorem,~$Q_{p,\mathcal{A},V-\mu_{1}|g|+\mathcal{W}}$ is nonnegative in~$\Omega$. Note that\begin{align}\label{gamma1}
&\int_{\Gamma_{1}}\mathcal{A}(x,\nabla \bar{v})\cdot\nabla\phi\dx+\int_{\Gamma_{1}}(V-\mu_{1}|g|+\mathcal{W})\bar{v}^{p-1}\phi\dx\notag\\
&=\int_{\Gamma_{1}}\varepsilon^{p-1}(V-\mu_{1}|g|+\mathcal{W})\phi\dx\geq 0.
\end{align} 
Moreover,\begin{align*}
\int_{\Gamma_{2}\cup\Gamma_{3}}\mathcal{A}(x,\nabla \bar{v})\cdot\nabla\phi\dx=\int_{\Gamma_{2}}\mathcal{A}(x,\nabla v)\cdot\nabla\phi\dx+\int_{\Gamma_{3}}\mathcal{A}(x,\nabla \bar{v})\cdot\nabla\phi\dx.
\end{align*}
By the chain and product rules, we calculate:  \begin{align*}
&\int_{\Gamma_{3}}\mathcal{A}(x,\nabla \bar{v})\cdot\nabla\phi\dx\\
&=-\int_{\Gamma_{3}}(\mathcal{A}(x,\nabla v)\cdot\nabla f'(v)^{p-1})\phi\dx+\int_{\Gamma_{3}}\mathcal{A}(x,\nabla v)\cdot\nabla(f'(v)^{p-1}\phi)\dx\\
&=-\int_{\Gamma_{3}}(p-1)|f'(v)|^{p-2}f''(v)|\nabla v|_{\mathcal{A}}^{p}\phi\dx+\int_{\Gamma_{3}}\mathcal{A}(x,\nabla v)\cdot\nabla\psi\dx.
\end{align*}
Then \begin{align*}
&\int_{\Gamma_{2}\cup\Gamma_{3}}\mathcal{A}(x,\nabla \bar{v})\cdot\nabla\phi\dx\\
&=-\int_{\Gamma_{3}}(p-1)|f'(v)|^{p-2}f''(v)|\nabla v|_{\mathcal{A}}^{p}\phi\dx+\int_{\Gamma_{2}\cup\Gamma_{3}}\mathcal{A}(x,\nabla v)\cdot\nabla\psi\dx\\
&=-\int_{\Gamma_{3}}(p-1)|f'(v)|^{p-2}f''(v)|\nabla v|_{\mathcal{A}}^{p}\phi\dx-\int_{\Gamma_{2}\cup\Gamma_{3}}(V-\mu_{1}|g|)v^{p-1}\psi\dx.
\end{align*}
Furthermore,
\begin{align*}
&\int_{\Gamma_{2}\cup\Gamma_{3}}\mathcal{A}(x,\nabla \bar{v})\cdot\nabla\phi\dx+\int_{\Gamma_{2}\cup\Gamma_{3}}(V-\mu_{1}|g|+\mathcal{W})\bar{v}^{p-1}\psi\dx\\
&=-\int_{\Gamma_{3}}(p-1)|f'(v)|^{p-2}f''(v)|\nabla v|_{\mathcal{A}}^{p}\phi\dx-\int_{\Gamma_{3}}(V-\mu_{1}|g|)v^{p-1}\psi\dx\\
&+\int_{\Gamma_{3}}(V-\mu_{1}|g|+\mathcal{W})\bar{v}^{p-1}\psi\dx\geq 0,
\end{align*}
which, together with the inequality \eqref{gamma1}, completes the proof.

(3): Let
$$M\triangleq\{(t,\tau)\in [S_{g}, \mu_1]\times\R~|~Q_{p,\mathcal{A},V-t|g|+\tau\mathcal{W}}\geq 0~\mbox{in }\Omega\}.$$ 
By \cite[Corollary 6.17]{HPR},~$M$ is convex. Note that~$(S_{g},0)$ and $(\mu_1,1)$ both belong to~$M$, which allows us to define~$\nu(t)\triangleq \min\{\tau\in\R~|~Q_{p,\mathcal{A},V-t|g|+\tau\mathcal{W}}\geq 0~\mbox{in }\Omega\}$ for all~$t\in [S_{g}, \mu_1]$. Then~$\nu$ is a convex function on $[S_{g}, \mu_1]$ by \cite[Corollary 6.17]{HPR}. Moreover,~$Q_{p,\mathcal{A},V-t|g|+\nu(t)\mathcal{W}}$ is critical for all~$t\in [S_{g}, \mu_1]$.
Otherwise, by \cite[Theorem 6.23]{HPR},~$\mathcal{W}$ is a Hardy-weight of~$Q_{p,\mathcal{A},V-t|g|+\nu(t)\mathcal{W}}$ for some~$t_{0}\in [S_{g}, \mu_1]$, which contradicts the definition of~$\nu(t_{0})$. We may easily verify that~$\nu(t)>0$ for all~$t\in (S_{g},\mu_1]$ and that~$\nu(S_{g})\leq 0$. Furthermore, the convexity of~$\nu$ implies that~$\nu(S_{g})=0$. Hence,~$Q_{p,\mathcal{A},V-S_{g}|g|}$ is critical.
\eproof
Now we show that under further assumptions, the following \emph{spectral gap phenomenon} holds true. The proof is analogous to that of \cite[Theorem 5.1]{Das}. See also \cite[Theorem 1.1]{Das2}.
\btheorem\label{att2}
Let~$V\in\widetilde{M}^{q}_{\loc}(p;\Omega)$ and~$g\in\mathcal{H}(\Omega)\cap \widetilde{M}^{q}_{\loc}(p;\Omega)\setminus\{0\}$. Suppose that~$S_{g}<S_{g}^{\infty}$, that~$Q_{p,\mathcal{A},V-S_{g}|g|}$ is critical in~$\Omega$, and that~$Q'_{p,\mathcal{A},V-S_{g}|g|}[u]=0$ satisfies Assumption \ref{ngradb} in~$\Omega$ and~$\Omega\setminus\overline{\omega_{i}}$ for all~$i\in\mathbb{N}$. If~$\int_{\Omega}V^{-}\Phi^{p}\dx<\infty$ for the ground state~$\Phi$ of ~$Q_{p,\mathcal{A},V-S_{g}|g|}$, then~$\Phi\in \widetilde{W}^{1,p}_{0}(\Omega)$ and~$S_{g}$ is attained at~$\Phi$.
\etheorem
\bproof
 By Lemma \ref{inull}, we may find a sequence~$\{\hat{\phi}_{k}\}_{k\in\mathbb{N}}\subseteq W^{1,p}(\Omega)\cap C_{c}(\Omega)$ 
such that~$0\leq\hat{\phi}_{k}\leq\Phi$ for all~$k\in\mathbb{N}$, that~$\lim_{k\rightarrow\infty}Q_{p,\mathcal{A},V-S_{g}|g|}[\hat{\phi}_{k}]=0$, and that~$\{\hat{\phi}_{k}\}_{k\in\mathbb{N}}$ converges to~$\Phi$ a.e. in~$\Omega$ and in~$L^{p}_{\loc}(\Omega)$. For every~$S_{g}<t<S_{g}^{\infty}$, there exists~$i(t)\in\mathbb{N}$ such that~$Q_{p,\mathcal{A},V-t|g|}$ is nonnegative in~$\Omega\setminus\overline{\omega_{i(t)}}$. By virtue of \cite[Theorem 7.7]{HPR} and the method for Theorem \ref{mdecay}, we may easily show that~$\int_{\Omega}|g||\Phi|^{p}\dx<\infty$. Note that for all~$k\in\mathbb{N}$,
$$Q_{p,\mathcal{A},V^{+}}[\hat{\phi}_{k}]=Q_{p,\mathcal{A},V-S_{g}|g|}[\hat{\phi}_{k}]+\int_{\Omega}V^{-}\hat{\phi}_{k}^{p}\dx+S_{g}\int_{\Omega}|g|\hat{\phi}_{k}^{p}\dx.$$
Since~$0\leq \hat{\phi}_{k}\leq \Phi$ for all~$k\in\mathbb{N}$ and~$\lim_{k\rightarrow\infty}Q_{p,\mathcal{A},V-S_{g}|g|}[\hat{\phi}_{k}]=0$, we conclude that~$Q_{p,\mathcal{A},V^{+}}[\hat{\phi}_{k}]$ and~$\int_{\Omega}\mathcal{V}|\hat{\phi}_{k}|^{p}\dx$ (see Definition \ref{PAVHSS}) are both bounded with respect to~$k$. Now it is clear that~$\{\hat{\phi}_{k}\}_{k\in\mathbb{N}}$ is bounded in~$\widetilde{W}^{1,p}_{0}(\Omega)$. By Lemma \ref{pav0}, up to a subsequence,~$\{\hat{\phi}_{k}\}_{k\in\mathbb{N}}$ converges to some~$\phi$ weakly in~$\widetilde{W}^{1,p}_{0}(\Omega)$ and strongly in~$L^{p}_{\loc}(\Omega)$. Then $\phi=\Phi$ since~$\{\hat{\phi}_{k}\}_{k\in\mathbb{N}}$ converges to~$\Phi$ in~$L^{p}_{\loc}(\Omega)$. Hence,~$\Phi\in \widetilde{W}^{1,p}_{0}(\Omega)$. Furthermore, by Lemma \ref{varhc},$$Q[\Phi]\geq S_{g}\int_{\Omega}|g|\Phi^{p}\dx.$$ Repeating the counterpart in the proof of Theorem \ref{Anoopmini}, we may estimate:
\begin{align*}
Q_{p,\mathcal{A},V^{+}}[\Phi]&=\int_{\Omega}|\nabla\Phi|_{\mathcal{A}}^{p}\dx+\int_{\Omega}V^{+}\lim_{k\rightarrow\infty}|\hat{\phi}_{k}|^{p}\dx\\
&\leq \liminf_{k\rightarrow\infty}\int_{\Omega}|\nabla\hat{\phi}_{k}|_{\mathcal{A}}^{p}\dx+\liminf_{k\rightarrow\infty}\int_{\Omega}V^{+}|\hat{\phi}_{k}|^{p}\dx\\
&\leq \liminf_{k\rightarrow\infty}Q_{p,\mathcal{A},V^{+}}[\hat{\phi}_{k}]\\
&\leq \lim_{k\rightarrow\infty}Q_{p,\mathcal{A},V-S_{g}|g|}[\hat{\phi}_{k}]+\int_{\Omega}V^{-}\Phi^{p}\dx+S_{g}\int_{\Omega}|g|\Phi^{p}\dx\\
&=\int_{\Omega}V^{-}\Phi^{p}\dx+S_{g}\int_{\Omega}|g|\Phi^{p}\dx,
\end{align*}
which implies that
$$Q[\Phi]\leq S_{g}\int_{\Omega}|g|\Phi^{p}\dx.$$
In conclusion,~$Q[\Phi]= S_{g}\int_{\Omega}|g|\Phi^{p}\dx.$
\eproof
The proof of the following lemma is similar to that of \cite[Lemma 5.6]{Das} and hence omitted.
\blemma
If~$g\in\mathcal{H}(\Omega)\setminus\{0\}$ and~$S_{g}=S_{g}^{\infty}$, then~$g\notin\mathcal{H}^{0}(\Omega)$.
\elemma
\subsection{The concentration compactness method}
In the present subsection, we use the concentration compactness method to obtain the third attainment. As we shall see, the Riesz representation theorem is a substantial ingredient in our approach. For more on this method, see \cite{ADas,Das} and references therein. The counterpart in \cite{Das} only deals with the case of $\Omega=\R^{n}$.

\textbf{Throughout this subsection, we assume that for every~$x\in\partial\Omega$,~$\Omega\cap B_{r}(x)$ is a domain for all sufficiently small~$r>0$.} For example, Lipschitz domains satisfy this condition.

For every Hardy-weight~$g$, we first define another two constants~$\mathbb{S}_{g}(x)$, depending also on~$x\in\overline{\Omega}$, and~$S_{g}^{*}$, and a portion~$\Sigma_{g}$ of~$\overline{\Omega}$.
\bdefinition\label{sigmadfn}
\emph{For every~$x\in\overline{\Omega}$ and~$g\in\mathcal{H}(\Omega)$, let
$$\mathbb{S}_{g}(x)\triangleq \lim_{r\rightarrow 0^{+}}\inf\left\{Q[\phi;\Omega\cap B_{r}(x)]~\bigg|~\phi\in \widetilde{W}^{1,p}_{0}(\Omega\cap B_{r}(x))~\mbox{and} \int_{\Omega\cap B_{r}(x)}|g||\phi|^{p}\dx=1\right\},$$~$S_{g}^{*}\triangleq \inf_{x\in\overline{\Omega}}\mathbb{S}_{g}(x)$, and~$\Sigma_{g}\triangleq\{x\in\overline{\Omega}~|~\mathbb{S}_{g}(x)<\infty\}$.
}
\edefinition
\bremark
\emph{It can be easily seen that~$S_{g}\leq S_{g}^{*}$.
}
\eremark
The proofs of the following three lemmas are respectively analogous to the counterparts of \cite[Lemma 5.7]{Das}, \cite[p.~22]{Lieb}, and \cite[Lemma 6.3]{Das}, and hence omitted. 
\blemma\label{sginfty}
For all~$g\in\mathcal{H}(\Omega)\cap M^{q}_{\loc}(p;\Omega)$ and~$x\in\Omega$,~$\mathbb{S}_{g}(x)=\infty$. Equivalently, for all~$g\in\mathcal{H}(\Omega)\cap M^{q}_{\loc}(p;\Omega)$,~$\Sigma_{g}\subseteq\partial\Omega$.
\elemma
\blemma\label{numineq}
For every~$\varepsilon>0$, there exists a constant~$C(\varepsilon,p)>0$ such that for almost all~$x\in\Omega$ and all~$\xi,\eta\in\R^{n}$,
$$||\xi+\eta|_{\mathcal{A}}^{p}-|\eta|_{\mathcal{A}}^{p}|\leq\varepsilon|\eta|_{\mathcal{A}}^{p}+C(\varepsilon,p)|\xi|_{\mathcal{A}}^{p}.$$
\elemma
\blemma\label{elimit}
Let~$\{\phi_{k}\}_{k\in\mathbb{N}}$ converge weakly in~$\widetilde{W}^{1,p}_{0}(\Omega)$ to~$\phi$. For every bounded~$\Phi\in C^{1}(\Omega)$ such that~$\supp\nabla\Phi$ is a compact subset of~$\Omega$, up to a subsequence,
$$\lim_{k\rightarrow\infty}\int_{\Omega}|(\phi_{k}-\phi)\nabla\Phi+\Phi\nabla(\phi_{k}-\phi)|_{\mathcal{A}}^{p}\dx=\lim_{k\rightarrow\infty}\int_{\Omega}|\nabla(\phi_{k}-\phi)|_{\mathcal{A}}^{p}|\Phi|^{p}\dx.$$
\elemma
\bdefinition
\emph{The potential~$V$ satisfies  \emph{Condition (H1)} if for all sequences~$\{\phi_{k}\}_{k\in\mathbb{N}}$ converging to~$\phi$ weakly in~$\widetilde{W}^{1,p}_{0}(\Omega)$, up to a subsequence,~$\lim_{k\rightarrow\infty}\int_{\Omega}V^{-}|\phi_{k}|^{p}\dx=\int_{\Omega}V^{-}|\phi|^{p}\dx$, and for every~$\phi\in\core$, up to a subsequence,~$\lim_{k\rightarrow\infty}\int_{\Omega}\phi V^{-}|\phi_{k}|^{p}\dx=\int_{\Omega}\phi V^{-}|\phi|^{p}\dx$.}
\edefinition
Now we present some examples realizing Condition (H1). See also \cite[p.~353]{Das}.
\bexample
\emph{\begin{itemize}
\item Trivially, all nonnegative potentials~$V\in M^{q}_{\loc}(p;\Omega)$ satisfy Condition (H1);
\item 
   If~$V^{-}\in \mathcal{H}^{0}(\Omega)$, then by Theorem \ref{conh} (1),~$V$ satisfies Condition (H1).
\end{itemize}}
\eexample
\textbf{From now on, throughout this subsection, we assume that~$V$ satisfies Condition (H1) and that~$g\in \mathcal{H}(\Omega)\setminus\{0\}$.} 

Let $\{\phi_{k}\}_{k\in\mathbb{N}}$ converge  to~$\phi$ weakly in~$\widetilde{W}^{1,p}_{0}(\Omega)$. See also \cite[p. 354]{Das}.
\begin{itemize}
\item For simplicity, we first assume that~$V\geq 0$. Consider the three sequences of measures: for every~$k\in\mathbb{N}$ and Borel set~$E$ in~$\Omega$,
$$\nu_{k}(E)\triangleq\int_{E}|g||\phi_{k}-\phi|^{p}\dx,\quad\gamma_{k}(E)\triangleq\int_{E}(|\nabla(\phi_{k}-\phi)|_{\mathcal{A}}^{p}+V|\phi_{k}-\phi|^{p})\dx,$$
and
$$\quad\gamma_{k}'(E)\triangleq\int_{E}(|\nabla \phi_{k}|_{\mathcal{A}}^{p}+V|\phi_{k}|^{p})\dx.$$ Then for all~$k\in\mathbb{N}$,~$\nu_{k},\gamma_{k},$ and~$\gamma_{k}'$ are 
finite. By two simple extensions (see \cite[p.~62]{Federer} and \cite[Definition 1.2]{evansf}), we may also assume that these measures are Radon measures 
on~$\R^{n}$. Moreover, the three sequences~$\{\nu_{k}(\R^{n})\}_{k\in\mathbb{N}},\{\gamma_{k}(\R^{n})\}_{k\in\mathbb{N}},$ and~$\{\gamma'_{k}(\R^{n})\}_{k\in\mathbb{N}}$ are all bounded since $\{\phi_{k}\}_{k\in\mathbb{N}}$ converges  to~$\phi$ weakly in~$\widetilde{W}^{1,p}_{0}(\Omega)$. By the weak compactness for measures, 
there exist three Radon measures~$\nu,\gamma$, and~$\gamma'$ on~$\R^{n}$ such that up to a subsequence, for all~$\varphi\in C_{c}(\Omega)$, 
$$\lim_{k\rightarrow\infty}\int_{\Omega}\varphi \dnuk=\int_{\Omega}\varphi \dnu,\quad\lim_{k\rightarrow\infty}\int_{\Omega}\varphi \dgammak=\int_{\Omega}\varphi\dgamma,$$
and
$$\quad\lim_{k\rightarrow\infty}\int_{\Omega}\varphi \dgammaak=\int_{\Omega}\varphi\dgammaa.$$
Furthermore, by the proof of the Riesz representation theorem \cite[Theorem 2.14]{Rudin}, and the two previous extensions of measures, we may also assume that for every open set~$U$ and arbitrary set~$A$ in~$\Omega$,
$$\nu(U)=\sup\left\{\int_{\Omega}\varphi \dnu: \varphi \in C_{c}(U)~\mbox{and}~0\leq \varphi\leq 1\right\}$$
and $$\nu(A)=\inf\{\nu(U):A\subseteq U\subseteq\Omega~\mbox{and}~U~\mbox{is open}\}.$$ In addition,  we ensure~$\nu(A)=\nu(\Omega\cap A)$ for all~$A\subseteq \R^{n}$.
For~$\gamma$ and~$\gamma'$, we also have similar representations. In particular,~$\nu$,$\gamma$, and~$\gamma'$ are all finite.
\item Now we consider a general~$V\in M^{q}_{\loc}(p;\Omega)$. Recall that~$C_{c}(\Omega)$ is separable. 
By virtue of the beginning argument of Lemma \ref{ng431}, and \cite[Theorem 2.14]{Rudin}, and repeating the proof of \cite[Theorem 1.41]{evansf}, we may find two Radon measures~$\gamma$ and~$\gamma'$ on~$\R^{n}$ with the same properties as the nonnegative case. The Radon measure~$\nu$ is obtained in the same way as the nonnegative case. In particular, up to a subsequence, for all~$\varphi\in C_{c}(\Omega)$, 
$$\lim_{k\rightarrow\infty}\int_{\Omega}\varphi \dnuk=\int_{\Omega}\varphi \dnu,\quad\lim_{k\rightarrow\infty}\int_{\Omega}\varphi(|\nabla(\phi_{k}-\phi)|_{\mathcal{A}}^{p}+V|\phi_{k}-\phi|^{p})\dx=\int_{\Omega}\varphi\dgamma,
$$
and$$\lim_{k\rightarrow\infty}\int_{\Omega}\varphi(|\nabla \phi_{k}|_{\mathcal{A}}^{p}+V|\phi_{k}|^{p})\dx=\int_{\Omega}\varphi\dgammaa.$$
\end{itemize}
\blemma\label{ng431}
 Let~$\{\phi_{k}\}_{k\in\mathbb{N}}$ converge to~$\phi$ weakly in~$\widetilde{W}^{1,p}_{0}(\Omega)$. Then for all Borel sets~$E$ in~$\Omega$,
 $$S^{*}_{g}\nu(E)\leq\gamma(E).$$In addition,~$\nu(\Omega\setminus\overline{\Sigma_{g}})=0$.
\elemma
\bproof
The proof is analogous to that of \cite[Lemma 6.4]{Das}.
For all~$\varphi\in C_{c}^{\infty}(\Omega)$ and~$k\in\mathbb{N}$,~$(\phi_{k}-\phi)\varphi\in \widetilde{W}^{1,p}_{0}(\Omega)$ (cf. \cite{Kinnunen}). By Lemma \ref{varhc}, for all~$k\in\mathbb{N}$,
\begin{align*}
\int_{\Omega}|g||(\phi_{k}-\phi)\varphi|^{p}\dx\leq \Vert g\Vert_{\mathcal{H}}\int_{\Omega}(|\nabla((\phi_{k}-\phi)\varphi)|_{\mathcal{A}}^{p}+V|(\phi_{k}-\phi)\varphi|^{p})\dx.
\end{align*}
By Lemma \ref{elimit} and letting~$k\rightarrow \infty$ in a suitable subsequence, we obtain
$$\int_{\Omega}|\varphi|^{p}\dnu\leq \Vert g\Vert_{\mathcal{H}}\int_{\Omega}|\varphi|^{p}\dgamma.$$ By virtue of \cite[Lemma 1.18]{Kinnunen}, we may deduce that for all nonnegative~$\varphi\in C_{c}(\Omega)$,
$$\int_{\Omega}\varphi\dnu\leq \Vert g\Vert_{\mathcal{H}}\int_{\Omega}\varphi\dgamma.$$ Then we may conclude that for all sets~$E$ in~$\R^{n}$,
\begin{equation}\label{mcompare}
\nu(E)\leq \Vert g\Vert_{\mathcal{H}}\gamma(E).
\end{equation}
The rest is easy and hence omitted.
\eproof
\blemma\label{measurec}
 Let~$\{\phi_{k}\}_{k\in\mathbb{N}}$ converge  to~$\phi$ weakly in~$\widetilde{W}^{1,p}_{0}(\Omega)$. Suppose that~$\vol (\overline{\Sigma_{g}})=0$. Then$$\gamma'(\Omega)\geq \int_{\Omega}(|\nabla \phi|_{\mathcal{A}}^{p}+V|\phi|^{p})\dx+S_{g}^{*}\nu(\Omega).$$
\elemma
\bproof
The proof, composed of three steps, is similar to that of \cite[Lemma 6.5]{Das}.
\begin{itemize}
\item For all Borel sets~$E$ in~$\Omega$,~$\gamma'(E)\geq \int_{E}(|\nabla \phi|_{\mathcal{A}}^{p}+V|\phi|^{p})\dx$. By \cite[Theorem 2.14]{Rudin} and \cite[Lemma 1.18]{Kinnunen}, it suffices to show that for all~$\varphi\in \core$ such that~$0\leq\varphi\leq 1$,$$\int_{\Omega}\varphi\dgammaa\geq \int_{\Omega}\varphi(|\nabla \phi|_{\mathcal{A}}^{p}+V|\phi|^{p})\dx.$$ By virtue of Condition (H1) and Fatou's lemma, up to a subsequence,
\begin{align*}
\int_{\Omega}\varphi\dgammaa&=\lim_{k\rightarrow\infty}\int_{\Omega}\varphi(|\nabla \phi_{k}|_{\mathcal{A}}^{p}+V|\phi_{k}|^{p})\dx\\
&= \liminf_{k\rightarrow\infty}\int_{\Omega}\varphi(|\nabla \phi_{k}|_{\mathcal{A}}^{p}+V^{+}|\phi_{k}|^{p})\dx-\lim_{k\rightarrow\infty}\int_{\Omega}\varphi V^{-}|\phi_{k}|^{p}\dx\\
&\geq\liminf_{k\rightarrow\infty}\int_{\Omega}\varphi|\nabla \phi_{k}|_{\mathcal{A}}^{p}\dx+\liminf_{k\rightarrow\infty}\int_{\Omega}\varphi V^{+}|\phi_{k}|^{p}\dx-\int_{\Omega}\varphi V^{-}|\phi|^{p}\dx\\
&\geq \liminf_{k\rightarrow\infty}\int_{\Omega}\varphi|\nabla \phi_{k}|_{\mathcal{A}}^{p}\dx+\int_{\Omega}\varphi V|\phi|^{p}\dx\\
&\geq \int_{\Omega}\varphi|\nabla \phi|_{\mathcal{A}}^{p}\dx+\int_{\Omega}\varphi V|\phi|^{p}\dx.
\end{align*}
\item For all Borel sets~$E\subseteq\overline{\Sigma_{g}}\cap \Omega$,~$\gamma(E)=\gamma'(E)$. For every~$m\in\mathbb{N}$, there exists an open set~$U_{m}\subseteq\Omega$ such that~$E\subseteq
U_{m}$ and~$\vol(U_{m}\setminus E)<1/m$. By taking intersections, we may also require~$\lim_{m\rightarrow\infty}\gamma(U_{m})=\gamma(E)$ and~$\lim_{m\rightarrow\infty}\gamma'(U_{m})=\gamma'(E)$.
Since~$E\subseteq\overline{\Sigma_{g}}$ and~$\vol (\overline{\Sigma_{g}})=0$,~$\vol(U_{m})=\vol(U_{m}\setminus E)<1/m$ for all~$m\in\mathbb{N}$. For every~$m\in\mathbb{N}$,~$\varphi\in C_{c}^{\infty}(U_{m})$ with~$0\leq\varphi\leq 1$, and~$\varepsilon>0$, we may estimate:
\begin{align*}
&\left|\int_{\Omega}\varphi(|\nabla(\phi_{k}-\phi)|_{\mathcal{A}}^{p}+V|\phi_{k}-\phi|^{p})\dx-\int_{\Omega}\varphi(|\nabla \phi_{k}|_{\mathcal{A}}^{p}+V|\phi_{k}|^{p})\dx\right|\\
&\leq \int_{\Omega}\varphi\left||\nabla(\phi_{k}-\phi)|_{\mathcal{A}}^{p}-|\nabla \phi_{k}|_{\mathcal{A}}^{p}\right|\dx+\int_{\Omega}\varphi |V|\left||\phi_{k}-\phi|^{p}-|\phi_{k}|^{p}\right|\dx\\
&\leq\varepsilon\int_{\Omega}\varphi|\nabla \phi_{k}|_{\mathcal{A}}^{p}\dx+C(\varepsilon,p)\int_{\Omega}\varphi|\nabla \phi|_{\mathcal{A}}^{p}\dx+\varepsilon\int_{\Omega}\varphi |V||\phi_{k}|^{p}\dx\\
&+C(\varepsilon,p)\int_{\Omega}\varphi |V||\phi|^{p}\dx\leq C\varepsilon+C(\varepsilon,p)\int_{U_{m}}(|\nabla \phi|_{\mathcal{A}}^{p}+|V||\phi|^{p})\dx,
\end{align*}
where the last inequality is based on the boundedness of~$\{\phi_{k}\}_{k\in\mathbb{N}}$ in~$\widetilde{W}^{1,p}_{0}(\Omega)$.
Furthermore, letting~$k\rightarrow\infty$ in a suitable subsequence, we get
$$\left|\int_{\Omega}\varphi\dgamma-\int_{\Omega}\varphi\dgammaa\right|\leq C\varepsilon+C(\varepsilon,p)\int_{U_{m}}(|\nabla \phi|_{\mathcal{A}}^{p}+|V||\phi|^{p})\dx.$$
Then by \cite[Lemma 1.18]{Kinnunen}, for all~$m\in\mathbb{N}$,~$\varphi\in C_{c}(U_{m})$ with~$0\leq\varphi\leq 1$, and~$\varepsilon>0$,$$\left|\int_{\Omega}\varphi\dgamma-\int_{\Omega}\varphi\dgammaa\right|\leq C\varepsilon+C(\varepsilon,p)\int_{U_{m}}(|\nabla \phi|_{\mathcal{A}}^{p}+|V||\phi|^{p})\dx.$$ For every~$m\in\mathbb{N}$, by taking the maximum function sequence of two sequences, we may find a sequence~$\{\varphi_{k}^{(m)}\}_{k\in\mathbb{N}}\subseteq C_{c}(U_{m})$ such that~$\lim_{k\rightarrow\infty} \int_{\Omega}\varphi_{k}^{(m)}\dgamma=\gamma(U_{m})$ and~$\lim_{k\rightarrow\infty} \int_{\Omega}\varphi_{k}^{(m)}\dgammaa=\gamma'(U_{m})$. It follows that for all~$m\in\mathbb{N}$ and~$\varepsilon>0$,
\begin{align*}
\left|\gamma(U_{m})-\gamma'(U_{m})\right|&=\lim_{k\rightarrow\infty}\left|\int_{\Omega}\varphi_{k}^{(m)}\dgamma-\int_{\Omega}\varphi_{k}^{(m)}\dgamma\right|\\
&\leq C\varepsilon+C(\varepsilon,p)\int_{U_{m}}(|\nabla \phi|_{\mathcal{A}}^{p}+|V||\phi|^{p})\dx.
\end{align*}
Letting~$m\rightarrow\infty$, we deduce that for all~$\varepsilon>0$,
$$\left|\gamma(E)-\gamma'(E)\right|\leq C\varepsilon.$$
Hence~$\gamma(E)=\gamma'(E)$.
\item Now by the first two steps and Lemma \ref{ng431},  we calculate: 
\begin{align*}
\gamma'(\Omega)&=\gamma'(\Omega\setminus\overline{\Sigma_{g}})+\gamma'(\Omega\cap\overline{\Sigma_{g}})\\
&\geq \int_{\Omega\setminus\overline{\Sigma_{g}}}(|\nabla \phi|_{\mathcal{A}}^{p}+V|\phi|^{p})\dx+S_{g}^{*}\nu(\Omega\cap\overline{\Sigma_{g}})\\
&\geq \int_{\Omega}(|\nabla \phi|_{\mathcal{A}}^{p}+V|\phi|^{p})\dx+S_{g}^{*}\nu(\Omega),
\end{align*}
where in the last inequality, we also use~$\vol (\overline{\Sigma_{g}})=0$.\qedhere
\end{itemize}
\eproof
\textbf{From now on, in this subsection, we assume that~$\Omega$ has a smooth exhaustion~$\{\omega_{i}\}_{i\in\mathbb{N}}$ such that~$\Omega\setminus\overline{\omega_{i}}$ 
is a domain for all~$i\in\mathbb{N}$, and let us fix such an exhaustion.}
\blemma\label{manyl}
 Let~$\{\phi_{k}\}_{k\in\mathbb{N}}$ converge  to~$\phi$ weakly in~$\widetilde{W}^{1,p}_{0}(\Omega)$. For every~$i\in\mathbb{N}$, take~$\Phi_{i}\in C(\Omega)$ such that~$0\leq\Phi_{i}\leq 1$,~$\Phi_{i}|_{\overline{\omega_{i}}}=0$, and~$\Phi_{i}|_{\Omega\setminus\omega_{i+1}}= 1$. Suppose that~$g$ is nonnegative. Then up to a subsequence with respect to~$k$, 
 \begin{enumerate} \item$\myd{\lim_{i\rightarrow\infty}\lim_{k\rightarrow\infty}\int_{\Omega\setminus\overline{\omega_{i}}}g|\phi_{k}|^{p}\dx=\lim_{i\rightarrow\infty}\lim_{k\rightarrow\infty}\nu_{k}(\Omega\setminus\overline{\omega_{i}})}$; \item$\myd{\lim_{i\rightarrow\infty}\lim_{k\rightarrow\infty} \nu_{k}(\Omega\setminus\overline{\omega_{i}})=\lim_{i\rightarrow\infty}\lim_{k\rightarrow\infty}\int_{\Omega}\Phi_{i}\dnuk=\lim_{i\rightarrow\infty}\lim_{k\rightarrow\infty}\int_{\Omega}\Phi_{i}g|\phi_{k}|^{p}\dx}$;
\item
\begin{align*}
&\lim_{i\rightarrow\infty}\lim_{k\rightarrow\infty}\int_{\Omega\setminus\overline{\omega_{i}}}(|\nabla \phi_{k}|_{\mathcal{A}}^{p}+V|\phi_{k}|^{p})\dx\\
&=\lim_{i\rightarrow\infty}\lim_{k\rightarrow\infty}\int_{\Omega\setminus\overline{\omega_{i}}}(|\nabla (\phi_{k}-\phi)|_{\mathcal{A}}^{p}+V|\phi_{k}-\phi|^{p})\dx;
\end{align*}
\item\begin{align*}
&\lim_{i\rightarrow\infty}\lim_{k\rightarrow\infty}\int_{\Omega\setminus\overline{\omega_{i}}}(|\nabla (\phi_{k}-\phi)|_{\mathcal{A}}^{p}+V|\phi_{k}-\phi|^{p})\dx\\ &=\lim_{i\rightarrow\infty}\lim_{k\rightarrow\infty}\int_{\Omega\setminus\overline{\omega_{i}}}\Phi_{i}(|\nabla (\phi_{k}-\phi)|_{\mathcal{A}}^{p}+V|\phi_{k}-\phi|^{p})\dx\\
&=\lim_{i\rightarrow\infty}\lim_{k\rightarrow\infty}\int_{\Omega\setminus\overline{\omega_{i}}}\Phi_{i}(|\nabla \phi_{k}|_{\mathcal{A}}^{p}+V|\phi_{k}|^{p})\dx.
\end{align*}
 \end{enumerate}
\elemma
\bproof
By virtue of a diagonalization argument, up to a subsequence with respect to~$k$, we may replace all the limits superior in \cite[Lemma 6.6]{Das} with limits. For (3) and (4), we split~$V$ into~$V^{+}-V^{-}$. Then the proof is analogous to that of \cite[Lemma 6.6]{Das} and hence omitted.
\eproof
The proof of the following lemma is analogous to that of \cite[Lemma 6.7]{Das} and hence omitted. But we point out the major differences.
\blemma\label{llemma}
Suppose that~$g$ is nonnegative and~$\{\phi_{k}\}_{k\in\mathbb{N}}$ converges to~$\phi$ weakly in~$\widetilde{W}^{1,p}_{0}(\Omega)$. Up to a subsequence with respect to~$k$, let$$\nu_{\infty}\triangleq\lim_{i\rightarrow\infty}\lim_{k\rightarrow\infty}\nu_{k}(\Omega\setminus\overline{\omega_{i}})\quad\mbox{and}\quad \gamma_{\infty}\triangleq\lim_{i\rightarrow\infty}\lim_{k\rightarrow\infty}\int_{\Omega\setminus\overline{\omega_{i}}}(|\nabla (\phi_{k}-\phi)|_{\mathcal{A}}^{p}+V|\phi_{k}-\phi|^{p})\dx.$$ Then up to a subsequence with respect to~$k$,
\begin{enumerate}
\item$\myd{S^{\infty}_{g}\nu_{\infty}\leq \gamma_{\infty}}$;
\item$\myd{\lim_{k\rightarrow\infty}\int_{\Omega}g|\phi_{k}|^{p}\dx=\int_{\Omega}g|\phi|^{p}\dx+\nu(\Omega)+\nu_{\infty}}$;
\item if, in addition,~$\vol (\overline{\Sigma_{g}})=0$, then
$$\lim_{k\rightarrow\infty}Q[\phi_{k}]\geq Q[\phi]+S^{*}_{g}\nu(\Omega)+\gamma_{\infty}.$$
\end{enumerate}
\elemma
\bproof
(1): For every~$i\in\mathbb{N}$, take~$\Phi_{i}\in C^{\infty}(\Omega)$ such that~$0\leq\Phi_{i}\leq 1$,~$\Phi_{i}|_{\overline{\omega_{i+1}}}= 0$, and~$\Phi_{i}|_{\Omega\setminus\omega_{i+2}}= 1$. 

(2): For every~$i\in\mathbb{N}$, take~$\Phi_{i}\in C^{\infty}(\Omega)$ such that~$0\leq\Phi_{i}\leq 1$,~$\Phi_{i}|_{\overline{\omega_{i}}}= 0$, and~$\Phi_{i}|_{\Omega\setminus\omega_{i+1}}= 1$.
\eproof
By virtue of the concentration compactness method, we prove that under suitable conditions, the Hardy constant is attained.
\btheorem\label{ccmthm}
Let~$g\in \mathcal{H}(\Omega)\setminus\{0\}$. Suppose that~$S_{g}<\min\{S_{g}^{*},S_{g}^{\infty}\}$ and~$\vol (\overline{\Sigma_{g}})=0$. Then~$S_{g}$ is attained at some~$\phi\in \widetilde{W}^{1,p}_{0}(\Omega)\setminus\{0\}$. 
\etheorem
\bproof
The proof is analogous to that of \cite[Theorem 6.1]{Das}.
For every~$\phi\in W^{1,p}(\Omega)\cap C_{c}(\Omega)$ such that~$ \int_{\Omega}|g||\phi|^{p}\dx>0$, we define
$$\mathcal{R}(\phi)\triangleq\frac{Q[\phi]}{\int_{\Omega}|g||\phi|^{p}\dx}.$$
Note that
\begin{align*}
S_{g}&=\inf\bigg\{Q[\phi]~\bigg|~\phi\in W^{1,p}(\Omega)\cap C_{c}(\Omega) \mbox{~and} \int_{\Omega}|g||\phi|^{p}\dx=1\bigg\}\\
&=\inf\bigg\{\mathcal{R}(\phi)~\bigg|~\phi\in W^{1,p}(\Omega)\cap C_{c}(\Omega) \mbox{~and} \int_{\Omega}|g||\phi|^{p}\dx>0\bigg\}\\
&=\inf\bigg\{\mathcal{R}(\phi)~\bigg|~\phi\in W^{1,p}(\Omega)\cap C_{c}(\Omega), \int_{\Omega}|g||\phi|^{p}\dx>0,~\mbox{and}~Q_{p,\mathcal{A},V^{+}}[\phi]=1~\bigg\}.
\end{align*}
Take a sequence~$\{\phi_{k}\}_{k\in\mathbb{N}}\subseteq W^{1,p}(\Omega)\cap C_{c}(\Omega)$ such that for all~$k\in\mathbb{N}$,$$ \int_{\Omega}|g||\phi_{k}|^{p}\dx>0\quad\mbox{and}\quad Q_{p,\mathcal{A},V^{+}}[\phi_{k}]=1,$$
and~$\lim_{k\rightarrow\infty}\mathcal{R}(\phi_{k})=S_{g}$. The sequence~$\{\phi_{k}\}_{k\in\mathbb{N}}$ is bounded in~$\widetilde{W}^{1,p}_{0}(\Omega)$ because for all~$k\in\mathbb{N}$,$$Q_{p,\mathcal{A},V^{+}}[\phi_{k}]=1,\quad Q_{p,\mathcal{A},V^{+}}[\phi_{k}]\geq\int_{\Omega}V^{-}|\phi_{k}|^{p}\dx,\quad\mbox{and}\quad Q[\phi_{k}]\geq S_{\mathcal{V}}\int_{\Omega}\mathcal{V}|\phi_{k}|^{p}\dx.$$ Then up to a subsequence,$$\lim_{k\rightarrow\infty}\int_{\Omega}|g||\phi_{k}|^{p}\dx=\Vert g\Vert_{\mathcal{H}}\lim_{k\rightarrow\infty}Q[\phi_{k}].$$  By Lemma \ref{pav0}, up to a subsequence,~$\{\phi_{k}\}_{k\in\mathbb{N}}$ converges to some~$\phi$ weakly in~$\widetilde{W}^{1,p}_{0}(\Omega)$. According to the argument preceding Lemma \ref{ng431}, there exist three Radon measures~$\nu,\gamma$, and~$\gamma'$ on~$\R^{n}$ such that up to a subsequence, for all~$\varphi\in C_{c}(\Omega)$, 
$$\lim_{k\rightarrow\infty}\int_{\Omega}\varphi \dnuk=\int_{\Omega}\varphi \dnu,\quad\lim_{k\rightarrow\infty}\int_{\Omega}\varphi(|\nabla(\phi_{k}-\phi)|_{\mathcal{A}}^{p}+V|\phi_{k}-\phi|^{p})\dx=\int_{\Omega}\varphi\dgamma,
$$
and$$\lim_{k\rightarrow\infty}\int_{\Omega}\varphi(|\nabla \phi_{k}|_{\mathcal{A}}^{p}+V|\phi_{k}|^{p})\dx=\int_{\Omega}\varphi\dgammaa.$$ 
By Lemma \ref{llemma} (2),
$$\lim_{k\rightarrow\infty}\int_{\Omega}g|\phi_{k}|^{p}\dx=\int_{\Omega}g|\phi|^{p}\dx+\nu(\Omega)+\nu_{\infty}.$$ 
We continue by showing that~$\nu(\Omega)=\nu_{\infty}=0$. On the contrary, suppose that~$\nu(\Omega)+\nu_{\infty}>0$. Then by Lemma \ref{llemma} (1) and (3) and the condition that~$S_{g}<\min\{S_{g}^{*},S_{g}^{\infty}\}$,
\begin{align*}
&\lim_{k\rightarrow\infty}\int_{\Omega}g|\phi_{k}|^{p}\dx=\Vert g\Vert_{\mathcal{H}}\lim_{k\rightarrow\infty}Q[\phi_{k}]\geq \Vert g\Vert_{\mathcal{H}}\left(\int_{\Omega}(|\nabla \phi|_{\mathcal{A}}^{p}+V|\phi|^{p})\dx+S^{*}_{g}\nu(\Omega)+\gamma_{\infty}\right)\\
&\geq \Vert g\Vert_{\mathcal{H}}\left(S_{g}\int_{\Omega}g|\phi|^{p}\dx+S^{*}_{g}\nu(\Omega)+\gamma_{\infty}\right)\geq \Vert g\Vert_{\mathcal{H}}\left(S_{g}\int_{\Omega}g|\phi|^{p}\dx+S^{*}_{g}\nu(\Omega)+S^{\infty}_{g}\nu_{\infty}\right)\\
&>\Vert g\Vert_{\mathcal{H}}\left(S_{g}\int_{\Omega}g|\phi|^{p}\dx+S_{g}\nu(\Omega)+S_{g}\nu_{\infty}\right)=\lim_{k\rightarrow\infty}\int_{\Omega}g|\phi_{k}|^{p}\dx,
\end{align*}
which is absurd and validates our claim. In other words, $$\lim_{k\rightarrow\infty}\int_{\Omega}g|\phi_{k}|^{p}\dx=\int_{\Omega}g|\phi|^{p}\dx.$$
Because~$\{\phi_{k}\}_{k\in\mathbb{N}}$ converges to~$\phi$ weakly in~$\widetilde{W}^{1,p}_{0}(\Omega)$, up to a subsequence,
$$\int_{\Omega}(|\nabla \phi|_{\mathcal{A}}^{p}+V^{+}|\phi|^{p})\dx\leq\liminf_{k\rightarrow\infty}\int_{\Omega}(|\nabla \phi_{k}|_{\mathcal{A}}^{p}+V^{+}|\phi_{k}|^{p})\dx=1.$$ By Condition (H1), up to a subsequence,
$$\lim_{k\rightarrow\infty}\int_{\Omega}V^{-}|\phi_{k}|^{p}\dx=\int_{\Omega}V^{-}|\phi|^{p}\dx.$$
We claim that~$\int_{\Omega}g|\phi|^{p}\dx>0$. Otherwise,$$0=\int_{\Omega}g|\phi|^{p}\dx=\lim_{k\rightarrow\infty}\int_{\Omega}|g||\phi_{k}|^{p}\dx=\Vert g\Vert_{\mathcal{H}}\lim_{k\rightarrow\infty}Q[\phi_{k}].$$ It follows that$$0=\lim_{k\rightarrow\infty}Q[\phi_{k}]=1-\lim_{k\rightarrow\infty}\int_{\Omega}V^{-}|\phi_{k}|^{p}\dx=1-\int_{\Omega}V^{-}|\phi|^{p}\dx.$$
Then$$0\leq Q[\phi]=Q_{p,\mathcal{A},V^{+}}[\phi]-\int_{\Omega}V^{-}|\phi|^{p}\dx\leq 0.$$ Therefore,$$0=Q[\phi]\geq S_{\mathcal{V}}\int_{\Omega}\mathcal{V}|\phi|^{p}\dx.$$ 
Since~$\mathcal{V}$ is positive in~$\Omega$,~$\phi=0$ a.e. in~$\Omega$, which is 
 impossible because~$\int_{\Omega}V^{-}|\phi|^{p}\dx=1$.
Finally, we calculate:
\begin{align*}
S_{g}&\leq \frac{\int_{\Omega}(|\nabla \phi|_{\mathcal{A}}^{p}+V|\phi|^{p})\dx}{\int_{\Omega}g|\phi|^{p}\dx}\\
&\leq \frac{\lim_{k\rightarrow\infty}\int_{\Omega}(|\nabla \phi_{k}|_{\mathcal{A}}^{p}+V|\phi_{k}|^{p})\dx}{\lim_{k\rightarrow\infty}\int_{\Omega}g|\phi_{k}|^{p}\dx}\\
&=\lim_{k\rightarrow\infty}\mathcal{R}[\phi_{k}]\\
&=S_{g},
\end{align*}
which gives the desired equality.
\eproof
The following corollary is a method for constructing Hardy-weights whose Hardy constants are attained. The proof is similar to that of \cite[Proposition 6.10]{Das} and hence omitted.
\bcorollary\label{lcorollary}
For every nonnegative~$g,g_{0}\in\mathcal{H}(\Omega)\setminus\{0\}$ such that
~$S^{*}_{g_{0}}=S^{\infty}_{g_{0}}=\infty$, and every~$\varepsilon>S_{g_{0}}/S_{g}$,$$S_{g+\varepsilon g_{0}}<\min\{S_{g+\varepsilon g_{0}}^{*},S^{\infty}_{g+\varepsilon g_{0}}\}.$$ If, in addition, $\vol(\overline{\Sigma_{g+\varepsilon g_{0}}})=0$, then~$S_{g+\varepsilon g_{0}}$ is attained at some~$\phi\in \widetilde{W}^{1,p}_{0,g+g_{0}}(\Omega)$. 
\ecorollary
\bremark
\emph{If~$g_{0}\in M^{q}_{c}(p;\Omega)$, then by Lemma \ref{sginfty} and the definitions of~$S^{\infty}_{g_{0}}$ and~$S^{*}_{g_{0}}$, we have~$S^{*}_{g_{0}}=S^{\infty}_{g_{0}}=\infty$.}
\eremark
In the end, we characterize Condition (H0) (see Definition \ref{condih}) for nonnegative potentials even though our conclusion is stronger than the characterization. The proof is similar to that of \cite[Theorem 6.9]{Das} and hence omitted. 
\btheorem
If~$S^{*}_{g}\!=\!S^{\infty}_{g}\!=\!\infty$, then the operator~$T_{g}$ satisfies Condition (H0). Conversely, if~$V$ is nonnegative and the operator~$T_{g}$ satisfies Condition (H0), then~$S^{*}_{g}\!=\!S^{\infty}_{g}\!=\!\infty$.  
\etheorem
{\small 
\appendix
\section{Two-sided estimates for the Bregman distances of~$|\cdot|^{p}$ and~$|\cdot|_{A}^{p}$}\label{appendix}
 In this appendix, we collect two two-sided estimates for the convenience of readers, which were proved successively in \cite{Pinliou} and \cite{Regev}. 
\blemma[{\cite[(2.19)]{Pinliou}}]\label{pBE}
For all~$\xi,\eta\in\R^{n}$ ($n\geq 1$),
$$|\xi+\eta|^{p}-|\xi|^{p}-p|\xi|^{p-2}\xi\cdot\eta\asymp |\eta|^{2}(|\eta|+|\xi|)^{p-2},$$
where both equivalence constants depend only on~$p$.
\elemma
\blemma[{\cite[(3.11)]{Regev}}]\label{ABE}
Let~$A\in \R^{n\times n}$ ($n\geq 1$) be a symmetric positive definite matrix. Then for all~$\xi,\eta\in\R^{n}$,
$$|\xi+\eta|_{A}^{p}-|\xi|_{A}^{p}-p|\xi|^{p-2}A\xi\cdot\eta\asymp |\eta|_{A}^{2}(|\eta|_{A}+|\xi|_{A})^{p-2},$$
where~$|\xi|_{A} = \sqrt{A\xi\cdot\xi}$ and both equivalence constants depend only on~$p$.
\elemma
\subsection*{Acknowledgments}
This paper is based on part of the author's ongoing PhD thesis of the Technion - Israel Institute of Technology under the supervision of Professor Yehuda Pinchover.
The author thanks him for constant instruction and encouragement.  Thanks are also due to the author's associate advisor Prof. Dr. Matthias Keller of the University of Potsdam for professional suggestions. 
The author acknowledges the support of the Technion and the Israel Science Foundation (grant 637/19) founded by the Israel Academy of Sciences and Humanities. 
\bibliographystyle{amsplain}
{}}
{\tiny\MakeUppercase{Department of Mathematics, Technion -
	Israel Institute of Technology, Haifa 3200003, Israel}} 

{\footnotesize \emph{E-mail addresses:} \texttt{yongjun.hou@campus.technion.ac.il; houmathlaw@outlook.com}}
\end{document}